\documentclass[a4paper,10pt]{amsart}
\usepackage{amsmath,amsthm,amssymb,enumerate}
\usepackage{epsfig}
\usepackage{amssymb}
\usepackage{amsmath}
\usepackage{amssymb}
\usepackage{amsmath,amsthm}
\usepackage[latin1]{inputenc}
\usepackage{bbm}
\usepackage{esint}
\usepackage{amsfonts}
\usepackage{amsxtra}
\usepackage{euscript,mathrsfs}
\usepackage{color}
\usepackage[left=3.6cm,right=3.6cm,top=4cm,bottom=3.3cm]{geometry}
\allowdisplaybreaks
\usepackage{tikz}
\usepackage{pgfplots}

\usepackage{pst-plot}
\psset{unit=5mm,plotpoints=1000,algebraic}

\usepackage{enumitem}
\setenumerate{label={\rm (\alph{*})}}

\usepackage{xcolor}
\colorlet{ColorPink}{red!30}
\definecolor{Gump}{rgb}{0,0.6,0.4}
\usepackage{graphicx}

\usepackage[colorlinks=true, linktocpage=true, linkcolor=red!70!black,
 citecolor=green!50!black]{hyperref}

\usepackage{amsfonts}
\usepackage{amsxtra}

\usepackage[frak,hyperref,theorem_section]{paper_diening}
\numberwithin{equation}{section}

  \theoremstyle{remark}

\newtheorem{remarkise}[theorem]{Remark}
\newtheorem{conj}[theorem]{Conjecture}

\newcommand{\R}{\mathbb R}
\newcommand{\N}{\mathbb N}

\newcommand{\C}{\mathbb C}
 
\newcommand{\mA}{\mathscr A}

\newcommand{\dif}{\mathrm{d}}

\DeclareMathOperator{\dist}{dist}

\renewcommand{\dif}{\operatorname{d}\!}
\newcommand{\lebe}{\operatorname{L}}
\newcommand{\sobo}{\operatorname{W}}

\newcommand{\hold}{\operatorname{C}}

\newcommand{\ball}{B}
\newcommand{\di}{\operatorname{div}}

\newcommand{\bmo}{\operatorname{BMO}}

\newcommand{\T}{\mathbb{T}}
\newcommand{\test}{\mathcal{T}}
\renewcommand{\AA}{\mathfrak{A}}
\newcommand{\rsym}{\mathbb{R}_{\operatorname{sym}}^{n\times n}}
\newcommand{\rdrei}{\mathbb{R}_{\operatorname{sym}}^{3\times 3}}

\allowdisplaybreaks

\renewcommand{\dashint}{\fint}

\begin{document}

%\begin{frontmatter}

\title[divsym-free $\lebe^{\infty}$-truncations and divsym-quasiconvex hulls]{On Symmetric div-quasiconvex hulls and \\ divsym-free $\lebe^{\infty}$-truncations }
\author[L.~Behn]{Linus Behn}
\address[L.~Behn]{Fakult\"{a}t f\"{u}r Mathematik, Universit\"{a}t Bielefeld, Universit\"{a}tsstra\ss e 25, 33615 Bielefeld, Germany}
\email{lbehn@math.uni-bielefeld.de}
\author[F.~Gmeineder]{Franz Gmeineder}
\address[F.~Gmeineder]{Fachbereich Mathematik und Statistik, Universit\"{a}t Konstanz, Universit\"{a}ts- stra\ss e 10, 78464 Konstanz, Germany}
\email{franz.gmeineder@uni-konstanz.de}
\author[S.~Schiffer]{Stefan Schiffer}
\address[S.~Schiffer]{Institut f\"{u}r angewandte Mathematik, Universit\"{a}t Bonn, Endenicher Allee 60, 53115 Bonn, Germany}
\email{schiffer@iam.uni-bonn.de}

\maketitle
\begin{abstract}
We establish that for any non-empty, compact set $K\subset\rdrei$ the $1$- and $\infty$-symmetric div-quasiconvex hulls $K^{(1)}$ and $K^{(\infty)}$ coincide. This settles a conjecture in a recent work of \textsc{Conti, M\"{u}ller \& Ortiz} \cite{CMO19} in the affirmative. As a key novelty, we construct an $\lebe^{\infty}$-truncation that preserves both symmetry and solenoidality of matrix-valued maps in $\lebe^{1}$. 
\end{abstract}

\section{Introduction} 
\subsection{Aim and scope}\label{sec:plasticity}
One of the key problems in continuum mechanics is the mathematical description of the plasticity behaviour of solids. Such solids are usually modelled by reference configurations $\Omega\subset\R^{3}$ subject to loads or forces and corresponding \emph{velocity fields} $v\colon\Omega\to\R^{3}$. The  (elasto)plastic behaviour of the material is mathematically described in terms of the stress tensor $\sigma\colon\Omega\to\rdrei$ and is dictated by the precise target $K\subset\rdrei$ where it takes values; $K$ is usually referred to as the \emph{elastic domain}. When ideal plasticity is assumed and potential hardening effects are excluded, $K$ is a compact set in $\rdrei$ with non-empty interior. As prototypical examples, in the \textsc{Von Mises} or \textsc{Tresca} models used for the description of metals or alloys, we have $K=\{\sigma\in\rdrei\colon\;\mathbf{f}(\sigma^{D})\leq \theta\}$ with a threshold $\theta>0$, the deviatoric stress $\sigma^{D}:=\sigma-\frac{1}{3}\mathrm{tr}(\sigma)E_{3\times 3}$ and \emph{convex} $\mathbf{f}\colon\rdrei\to\R$. Generalising this to $K=\{\sigma\in\rdrei\colon\;\mathbf{f}(\sigma^{D})+\vartheta\mathrm{tr}(\sigma)\leq \theta\}$ for $\vartheta>0$ as in the \textsc{Drucker-Prager} or \textsc{Mohr-Coulomb} models for concrete or sand (cf.~\cite{DruckerPrager,Lubliner}), such models take into account persisting volumetric changes induced by the hydrostatic pressure as plasticity effects. In all of these models, $K$ is a \emph{convex} set. This opens the gateway to the techniques from convex analysis, and we refer to \cite{FuchsSeregin,Lubliner} for a more detailled treatment of the matter.

As the main motivation for the present paper, the convexity assumption on the elastic domain $K$ is \emph{not satisfied} by all materials. A prominent example where the non-convexity of $K$ can be observed explicitely is fused silica glass (cf.~ \textsc{Meade \& Jeanloz} \cite{MJ}). Slightly more generally, for amorphous solids being deformed subject to shear, experiments on the molecular dynamics (cf.~\textsc{Maloney \& Robbins} \cite{Maloney08}) exhibit the formation of characteristic patterns in the underlying deformation fields. As a possible explanation of this phenomenon, the emergence of such patterns on the \emph{microscopic} level displays the effort of the material to cope with the enduring  \emph{macroscopic} deformations. Within the framework of limit analysis \cite{Lubliner}, \textsc{Schill} et al. \cite{Schill} offer a link between the non-convexity of $K$ and the appearance of such fine microstructure. Working from plastic dissipation principles, the corresponding static problem is identified in \cite{Schill} as 
\begin{align}\label{eq:varprin}
\sup_{\sigma}\inf_{v}\left\{\int_{\Omega}\sigma\cdot \nabla v\dif x\colon\;\sigma\in\lebe_{\mathrm{div}}^{\infty}(\Omega;K),\;\;v\in\sobo^{1,1}(\Omega;\R^{n}),\;\;\;v=g\;\text{on}\;\partial\Omega \right\}
\end{align}
for given boundary data $g\colon\partial\Omega\to\R^{3}$. Here, $\lebe_{\mathrm{div}}^{\infty}(\Omega;K)$ is the space of all $\lebe^{\infty}(\Omega;K)$-maps which are row-wise divergence-free (or solenoidal) in the sense of distributions; note that, if even we admitted general $\sigma\in\lebe^{\infty}(\Omega;K)$  in \eqref{eq:varprin}, the variational principle would be non-trivial only for $\sigma\in\lebe_{\mathrm{div}}^{\infty}(\Omega;K)$. Stability under microstructure formation, in turn, is linked to the existence of solutions of \eqref{eq:varprin}; cf. \textsc{M\"{u}ller} \cite{Mul98} for a discussion of the underlying principles. Towards the existence of solutions, the direct method of the Calculus of Variations requires semicontinuity, and it is here where the set $K$ must be relaxed. By the constraints on $\sigma$, this motivates the passage to the \emph{symmetric div-quasiconvex hull} of $K$ as studied by \textsc{Conti, M\"{u}ller \& Ortiz} \cite{CMO19}. In the present paper, we complete the characterisation of such hulls (cf.~Theorem~\ref{thm:main1} below) and thereby answer a conjecture posed in \cite{CMO19} in the affirmative. To state our result, we pause and introduce the requisite terminology first. 
\subsection{Divsym-quasiconvexity and the main result}
Following \cite{CMO19}, we call a Borel measurable, locally bounded function $F\colon\rsym\to\R$ \emph{symmetric div-quasiconvex} if 
\begin{align}\label{eq:defDivsymQC}
F(\xi) \leq \int_{\T_n}F(\xi+\varphi(x))\dif x 
\end{align}
holds for all $\xi\in\rsym$ and all admissible test maps\begin{align}\label{eq:testmaps}
\varphi \in \test :=\left\{ \phi \in \hold^{\infty}(\T_n;\rsym)\;\;\;\;\di(\varphi)=0, \int_{\T_n} \phi \dif x=0\right\},
\end{align}
where $\T_n$ denotes the $n$-dimensional torus. Here, the divergence is understood in the row- (or equivalently, column-)wise manner. Accordingly, the  \emph{symmetric div-quasiconvex (or divsym-quasiconvex) envelope} of a Borel measurable, locally bounded function $F\colon\rsym\to\R$ is defined as the largest symmetric div-quasiconvex function below $F$; more explicitely, 
\begin{align}\label{eq:DefEnvelope}
\mathscr{Q}_{\mathrm{sdqc}}F(\xi):=\inf\left\{\int_{\T_n}F(\xi+\varphi(x))\dif x\colon\;\varphi\in\test\right\}.
\end{align}
Divsym-quasiconvexity is a strictly weaker notion than convexity, which can be seen \textsc{Tartar}'s example \cite{Tartar} $f\colon\rsym\ni \xi \mapsto (n-1)|\xi|^{2}-\mathrm{tr}(\xi)^{2}$. 

The discussion in Section~\ref{sec:plasticity} necessitates a notion of divsym-quasiconvexity \emph{for sets}. Inspired by the separation theory from convex analysis, we call a compact set $K\subset\rsym$ \emph{symmetric div-quasiconvex} provided for each $\xi\in\rsym\setminus K$ there exists a symmetric div-quasiconvex $g\in\hold(\rsym;[0,\infty))$ such that $g(\xi)>\max_{K}g$. The relaxation of the elastic domains $K\subset\rsym$ in turn is defined in terms of the symmetric div-quasiconvex envelopes of distance functions. For a compact subset $K\subset\rsym$ and $1\leq p<\infty$, put $f_{p}(\xi):=\mathrm{dist}^{p}(\xi,K)$. The \emph{$p$-symmetric div-quasiconvex hull} of $K$ then is defined by 
\begin{align}\label{eq:Kp}
K^{(p)}:=\{\xi\in\rsym\colon\;\mathscr{Q}_{\mathrm{sdqc}}f_{p}(\xi)=0\}, 
\end{align}
whereas we set for $p=\infty$: 
\begin{align}\label{eq:Kinfty}
K^{(\infty)}:=\left\{\xi\in\rsym\colon\;\begin{array}{c}g(\xi)\leq \max_{K}g\;\text{for all symmetric} \\ \text{div-quasiconvex $g\in\hold(\rsym;[0,\infty))$}\end{array}\right\}.
\end{align}
Both \eqref{eq:Kp} and \eqref{eq:Kinfty} are the natural generalisations of the usual convex hulls to the symmetric div-quasiconvex context, and one easily sees that $K^{(\infty)}$ is the smallest symmetric div-quasiconvex, compact set containing $K$. 

By our discussion in Section~\ref{sec:plasticity}, it is  particularly important to understand the properties of the symmetric div-quasiconvex hulls. In \cite{CMO19}, \textsc{Conti, M\"{u}ller \& Ortiz} established that $K^{(p)}$ is independent of $1<p<\infty$. Specifically, they conjectured in \cite[Rem.~3.9]{CMO19} that $K^{(1)}=K^{(\infty)}$ in analogy with the usual quasiconvex envelopes (see \textsc{Zhang} \cite{Zhang97} or \textsc{M\"{u}ller} \cite[Thm.~4.10]{Mul98}). The present paper answers this question in the affirmative, leading us to our main result:
\begin{theorem}[Main result]\label{thm:main1}
Let $K\subset\rdrei$ be compact. Then $K^{(1)}=K^{(\infty)}$ and so 
\begin{align}\label{eq:main}
K^{(p)}=K^{(1)}=K^{(\infty)}\qquad \text{for all}\;\;\;1\leq p \leq \infty.
\end{align}
\end{theorem}
Let us note that the $p$-symmetric div-quasiconvex hulls satisfy the antimonotonicity property with respect to inclusions, i.e., if $1\leq p \leq q \leq \infty$, then $K^{(q)}\subset K^{(p)}$. For Theorem~\ref{thm:main1}, it thus suffices to establish $K^{(1)}\subset K^{(\infty)}$, and this is exactly what shall be achieved in Section~\ref{sec:calcvar}. From a proof perspective, any underlying argument must use an $\lebe^{\infty}$-truncation of suitable recovery sequences, simultaneously keeping track of the differential constraint. Contrary to routine mollification, truncations leave the input functions unchanged on a large set and display an important tool in the study of nonlinear problems \cite{AcFu84,BallZhang90,FrMaSt,Frehse,Mul99,Zhang90}. It is here where Theorem~\ref{thm:main1} cannot be established by analogous means as in \cite[Sec.~3]{CMO19}, where a higher order truncation argument in the spirit of \textsc{Acerbi \& Fusco} \cite{AcFu88} and \textsc{Zhang} \cite{Zhang92} is employed. More precisely, for $1<p < q < \infty$, the critical inclusion $K^{(p)} \subset K^{(q)}$ is established in \cite{CMO19} by passing to the corresponding potentials of divsym-free fields, and as these potentials are of second order,  performing a $\sobo^{2,\infty}$-truncation on the potentials; this shall be referred to as \emph{potential truncation}. The underlying potential operators are obtained as suitable Fourier multiplier operators, which is why they only satisfy strong $\lebe^{p}$-$\lebe^{p}$-bounds for $1<p<\infty$ (cf.~Lemma~\ref{lem:proj} below). It is well-known that such Fourier multiplier operators do not map $\lebe^{1}\to\lebe^{1}$ boundedly (cf.~\textsc{Ornstein} \cite{Ornstein}), and so this approach is bound to fail in view of Theorem~\ref{thm:main1}. In the regime $1<p<\infty$, this strategy can readily be employed in the general context of $\mathscr{A}$-quasiconvex hulls in the sense of \textsc{Fonseca \& M\"{u}ller} \cite{FonMul99} (cf.~Proposition~\ref{prop:modtrunc} and Section~\ref{sec:revisit}) but is not even required for the inclusion $K^{(p)}\subset K^{(q)}$, $p<q$ and can be established by more elementary means; cf.~ Lemma~\ref{lem:simple} and its proof for the simplifying argument.
\subsection{A truncation theorem and its context}
The key tool in establishing Theorem~\ref{thm:main1} therefore consists in the following truncation result, allowing us to truncate a $\mathrm{div}$-free $\lebe^{1}$-map $u\colon\R^{3}\to\R_{\mathrm{sym}}^{3\times 3}$ while still preserving the constraint $\mathrm{div}(u)=0$:
\begin{theorem}[Main truncation theorem]\label{thm:main2}
There exists a constant $C>0$ solely depending on the underlying space dimension $n=3$ with the following property: For all $u \in \lebe^1(\R^3;\R_{\mathrm{sym}}^{3\times 3})$ with $\di(u)=0$ in $\mathscr{D}'(\R^3;\R^3)$ and all $\lambda>0$ there exists $u_{\lambda} \in \lebe^1(\R^3;\rdrei)$ satisfying the \begin{enumerate}
\item\label{item:thmmain1} \emph{$\lebe^{\infty}$-bound:} 
\begin{align*}
\Vert u_{\lambda} \Vert_{\lebe^{\infty}(\R^{3})} \leq C \lambda.
\end{align*} 
\item\label{item:thmmain2} \emph{strong stability:} 
\begin{align*}
\Vert u - u_{\lambda} \Vert_{\lebe^{1}(\R^{3})} \leq C \int_{\{\vert u \vert > \lambda\}} \vert u \vert \dif x.
\end{align*} 
\item\label{item:thmmain3} \label{property:bonus} \emph{small change:} 
\begin{align*}
\mathscr{L}^{3} (\{ u \neq u_{\lambda} \}) \leq C \lambda^{-1} \int_{\{\vert u \vert > \lambda\}} \vert u \vert \dif x.
\end{align*} 
\item\label{item:thmmain4} \emph{differential constraint:} $\di (u_{\lambda})= 0$ in $\mathscr{D}'(\R^{3};\R^{3})$.
\end{enumerate}
The same remains valid when replacing the underlying domain $\R^{3}$ by the torus $\T_{3}$.
\end{theorem}
The way in which Theorem~\ref{thm:main2} implies Theorem~\ref{thm:main1} can be accomplished by  analogous means as in \cite{CMO19} (also see the discussion by the third author \cite{Schiffer21}), and is sketched for the reader's convenience in Section~\ref{sec:calcvar}. Here we heavily rely on the \emph{strong stability property} from item~\ref{item:thmmain2}, without which the proof of Theorem~\ref{thm:main1} is not clear to us. The detailled construction that underlies the proof of Theorem~\ref{thm:main2}, reminiscent of a geometric version of the \textsc{Whitney} smoothing or extension procedure \cite{Whitney}, is explained in Section~\ref{sec:consttrunc} and carried out in detail in Section~\ref{sec:construction}. Here we understand by \emph{geometric} that the construction is directly taylored to the problem at our disposal, meaning that the solenoidality constraint $\mathrm{div}(u)=0$ is visible in our construction in terms of the Gau\ss -Green theorem on certain simplices. 

Working on a higher a priori regularity level, \emph{Lipschitz truncations} that preserve solenoi- dality constraints are not new and have been studied most notably by \textsc{Diening} et al. \cite{BDF12,BDS13}, originally developed for problems from mathematical fluid mechanics and since then having been fruitfully used in a variety of related problems; see, e.g., \textsc{S\"{u}li} et al.  \cite{Die13,SuliTscherpel}. Let us note that the two key approaches in \cite{BDF12,BDS13} either hinge on locally correcting divergence contributions on certain bad sets \cite{BDF12} or performing the potential truncation \cite{BDS13}. Whereas the ansatz in \cite{BDF12} in principle  is imaginable to work in the present setting apart from technical intricacies (cf.~Remark~\ref{rem:BDF}), the key drawback of the potential truncation is the non-availability of the strong stability estimate. This is essentially a consequence of singular integrals only mapping $\lebe^{\infty}\to\bmo$ in general but \emph{not} $\lebe^{\infty}\to\lebe^{\infty}$; see Section~\ref{sec:revisit} and Proposition~\ref{prop:modtrunc}, where the corresponding potential truncations are revisited and discussed in the general framework of constant rank operators $\mathscr{A}$ a l\'{a} \textsc{Schulenberger \& Wilcox} \cite{SchulWilcox} or \textsc{Murat} \cite{Murat}.

\subsection{Organisation of the paper} 
Apart from this introductory section, the paper is organised as follows: In Section~\ref{sec:prelims}, we fix notation and gather auxiliary material on maximal operators and basic facts from harmonic analysis. Section~\ref{sec:consttrunc} then explains the idea underlying the construction employed in the proof of Theorem~\ref{thm:main2}, and is then carried out in detail in Section~\ref{sec:construction}. Section~\ref{sec:calcvar} is devoted to the proof of Theorem~\ref{thm:main1}, and the paper is concluded in Section~\ref{sec:revisit} by revisiting potential truncations. The Appendix, Section~\ref{sec:appendix}, gathers various instrumental computations that underlie some of the results presented in Section~\ref{sec:construction}.
{\small
\subsection*{Acknowledgment} 
The authors are grateful to \textsc{Stefan M\"{u}ller} for bringing our attention to the theme considered in the paper. The research has received funding from the German Research Association (DFG) via the International Research Training Group 2235 'Searching for the regular in the irregular: Analysis of singular and random systems' (L.B.), the Hector foundation (F.G.) and the DFG through the graduate school BIGS of the Hausdorff Center for Mathematics (GZ EXC 59 and 2047/1, Projekt-ID 390685813) (S.Sc.).}
\section{Preliminaries}\label{sec:prelims}
\subsection{Notation}
The linear operators between two finite-dimensional real vector spaces $V,W$ are denoted $\mathscr{L}(V;W)$. We denote $\mathscr{L}^{n}$ and $\mathscr{H}^{n-1}$ the $n$-dimensional Lebesgue or $(n-1)$-dimensional Hausdorff measures, respectively. For notational brevity, we shall also write $\dif^{n-1}=\dif\mathscr{H}^{n-1}$. Given $n$- or $(n-1)$-dimensional measurable subsets $\Omega$ and $\Sigma$ of $\R^{n}$ with $\mathscr{L}^{n}(\Omega),\mathscr{H}^{n-1}(\Sigma)\in (0,\infty)$, respectively, we use the shorthand
\begin{align*}
\dashint_{\Omega}u\dif x := \frac{1}{\mathscr{L}^{n}(\Omega)}\int_{\Omega}u\dif x\;\;\;\text{and}\;\;\;\dashint_{\Sigma}v\dif^{n-1}x:=\frac{1}{\mathscr{H}^{n-1}(\Sigma)}\int_{\Sigma}v\dif^{n-1}x
\end{align*}
for $\mathscr{L}^{n}$- or $\mathscr{H}^{n-1}$-measurable maps $u\colon\Omega\to\R^{m}$ and $v\colon\Sigma\to\R^{m}$. As we shall mostly assume $n=3$, we denote $\ball_{r}(z)$ the open ball of radius $r$ centered at $z\in\R^{3}$, whereas we reserve the notation $\mathbb{B}_{r}(z)$ to denote the corresponding open balls in the symmetric $(3\times 3)$-matrices $\rdrei$. By \emph{cubes} $Q$ we understand non-degenerate cubes throughout, and use $\ell(Q)$ to denote their sidelength. Lastly, for $x_{1},...,x_{j}\in\R^{3}$, we denote $\langle x_{1},...,x_{j}\rangle$ the convex hull of the vectors $x_{1},...,x_{j}$, and if $x_{1},x_{2},x_{3}$ do not lie on a joint line, $\mathrm{aff}(x_{1},x_{2},x_{3})$ the affine hyperplane containing $x_{1},x_{2},x_{3}$. 
\subsection{Maximal operator, bad sets and Whitney covers}\label{sec:maxop}
For a finite dimensional real vector space $V$, $w\in\lebe^{1}(\R^{n};V)$ and $R>0$, we recall the (restricted) \emph{centered Hardy-Littlewood maximal operators} to be defined by 
\begin{align}\label{eq:HLMO}
\begin{split}
&\mathcal{M}_{R}w(x):=\sup_{0<r<R}\dashint_{\ball_{r}(x)}|w|\dif y,\qquad x\in\R^{n}, \\ 
&\mathcal{M}w(x):= \sup_{r>0}\dashint_{\ball_{r}(x)}|w|\dif y,\qquad x\in\R^{n}.
\end{split}
\end{align}
Note that, by lower semicontinuity of $\mathcal{M}_{R}w$, the superlevel sets $\{\mathcal{M}_{R}w>\lambda\}$ are open for all $\lambda>0$. Moreover, we record that $\mathcal{M}$ is of weak-$(1,1)$-type, meaning that there exists $c=c(n)>0$ such that 
\begin{align}\label{eq:weak11}
\mathscr{L}^{n}(\{\mathcal{M}w>\lambda\})\leq \frac{c}{\lambda}\|w\|_{\lebe^{1}(\R^{n})}\qquad\text{for all}\;w\in\lebe^{1}(\R^{n};V). 
\end{align}
See \cite{Grafakos,Stein} for more background information. Now let $\Omega\subset\R^{n}$ be open. Then there exists a \emph{Whitney cover} $\mathscr{W}=(Q_{j})$ for $\Omega$. By this we understand a sequence of open cubes $Q_{j}$ with the following properties: 
\begin{enumerate}[label={(W\arabic{*})},start=1]
\item\label{item:W1} $\Omega=\bigcup_{j\in\mathbb{N}}Q_{j}$.
\item\label{item:W2} $\frac{1}{5}\ell(Q_{j})\leq\mathrm{dist}(Q_{j},\Omega^{\complement})\leq 5\ell(Q_{j})$ for all $j\in\mathbb{N}$. 
\item\label{item:W3} \emph{Finite overlap:} There exists a number $\mathtt{N}=\mathtt{N}(n)>0$ such that at most $\mathtt{N}$ elements of $\mathscr{W}$ overlap; i.e., for each $i\in\mathbb{N}$, 
\begin{align*}
|\{j\in\mathbb{N}\colon\;Q_{j}\in\mathscr{W}\;\text{and}\,Q_{i}\cap Q_{j}\neq\emptyset\}|\leq \mathtt{N}. 
\end{align*}
\item\label{item:W4} \emph{Comparability for touching cubes:} There exists a constant $c(n)>0$ such that if $Q_{i},Q_{j}\in\mathscr{W}$ satisfy $Q_{i}\cap Q_{j}\neq\emptyset$, then 
\begin{align*}
\frac{1}{c(n)}\ell(Q_{i})\leq \ell(Q_{j})\leq c(n)\ell(Q_{i}). 
\end{align*}
\end{enumerate}
Whenever such a Whitney cover is considered, we tacitly understand $x_{j}$ to be the \emph{centre} of the corresponding cube $Q_{j}$. Based on the Whitney cover $\mathscr{W}$ from above, we choose a partition of unity $(\varphi_{j})$ subject to $\mathscr{W}$ with the following properties:
\begin{enumerate}[label={(P\arabic{*})},start=1]
\item\label{item:P1} For any $j\in\mathbb{N}$, $\varphi_{j}\in\hold_{c}^{\infty}(Q_{j};[0,1])$. 
\item\label{item:P2} $\sum_{j\in\mathbb{N}}\varphi_{j}=1$ in $\Omega$.
\item\label{item:P3} For each $l\in\mathbb{N}$, there exists a constant $c=c(n,l)>0$ such that 
\begin{align*}
|\nabla^{l}\varphi_{j}|\leq\frac{c}{\ell(Q_{j})^{l}}\qquad\text{for all}\;j\in\mathbb{N}.
\end{align*}
\end{enumerate}
\subsection{Differential operators and projection maps}\label{sec:HA}
For the following sections, we require some terminology for differential operators and a suitable projection property to be gathered in the sequel. Let $\mA$ be a constant coefficient, linear and homogeneous differential operator of order $k\in\mathbb{N}$ on $\R^{n}$ (or $\T_{n}$) between $\R^{d}$ and $\R^{N}$, so $\mA$ has a representation 
\begin{align}\label{eq:Aform}
\mA u = \sum_{|\alpha|=k}\mA_{\alpha}\partial^{\alpha}u,\qquad u\colon\R^{n}\to\R^{d},
\end{align}
with fixed $\mA_{\alpha}\in\mathscr{Lin}(\R^{d};\R^{N})$ for $|\alpha|=k$. Following \cite{Murat,SchulWilcox} we say that $\mA$ has \emph{constant rank} (in $\R$) provided the rank of the Fourier symbol $\mA[\xi]=\sum_{|\alpha|=k}\mA_{\alpha}\xi^{\alpha}\colon \R^{d}\to\R^{N}$ is independent of $\xi\in\R^{n}\setminus\{0\}$.  A constant coefficient differential operator $\mathbb{A}$ of order $j\in\mathbb{N}$ on $\R^{n}$ (or $\T_{n}$) between $\R^{\ell}$ and $\R^{d}$ consequently is called a \emph{potential} of $\mathscr{A}$ provided for each $\xi\in\R^{n}\setminus\{0\}$ the Fourier symbol sequence 
\begin{align*}
\R^{\ell}\stackrel{\mathbb{A}[\xi]}{\longrightarrow}\R^{d}\stackrel{\mathscr{A}[\xi]}{\longrightarrow}\R^{N} 
\end{align*}
is exact at every $\xi\in\R^{n}\setminus\{0\}$, i.e., $\mathbb{A}[\xi](\R^{\ell})=\ker(\mA[\xi])$ for each such $\xi$. We moreover say that $\mA$ has \emph{constant rank} (in $\mathbb{C}$) provided $\mA[\xi]\colon\mathbb{C}^{d}\to\mathbb{C}^{N}$ has rank independent of $\xi\in\mathbb{C}^{n}\setminus\{0\}$. If we only speak of \emph{constant rank}, then we tacitly understand constant rank in $\R$.
In Section~\ref{sec:revisit}, we require the following two auxiliary results, ensuring both the existence of potentials and suitable projection operators. 
\begin{lemma}[Existence of potentials, {\cite[Thm.~1, Lem.~5]{Raita19}}]\label{lem:potential} Let $\mA$ be a differential operator with constant rank over $\R$. Then $\mA$ possesses a potential $\mathbb{A}$. Moreover, if $u\in\hold^{\infty}(\T_{n};\R^{d})$ satisfies $\int_{\T_{n}}u\dif x = 0$ and $\mA u=0$, there exists $v\in\hold^{\infty}(\T_{n};\R^{\ell})$ with $\mathbb{A}v=u$. Equally, for each $u\in\mathscr{S}(\R^{n};\R^{d})$ with $\mA u=0$ there exists $v\in\mathscr{S}(\R^{n};\R^{\ell})$ with $\mathbb{A}v=u$.
\end{lemma}
\begin{lemma}[Projection maps on the torus, {\cite[Lem.~2.14]{FonMul99}}]\label{lem:proj}
Let $1<p<\infty$ and let $\mA$ be a differential operator of order $k$ with constant rank in $\R$. Then there is a bounded, linear projection map $P_{\mA} \colon \lebe^p(\T_n;\R^d) \to \lebe^p(\T_n;\R^d)$ with the following properties: 
\begin{enumerate}
    \item\label{item:proj1} $P_{\mA} u \in \ker \mA$ and $P_{\mA} \circ P_{\mA} = P_{\mA}$.
    \item\label{item:proj2} $\Vert u - P_{\mA} u \Vert_{\lebe^{p}(\T_{n})} \leq C_{\mA,p} \Vert \mA u \Vert_{\sobo^{-k,p}(\T_{n})}$ whenever $\dashint_{\T_{n}}u\dif x =0$.
    \item\label{item:proj3} If $(u_j)\subset\lebe^{p}(\T_{n};\R^{d})$ is bounded and $p$-equiintegrable, i.e., \begin{align*}
        \lim_{\varepsilon \searrow 0} \left( \sup_{j \in \N} \sup_{E \colon \mathscr{L}^n(E) < \varepsilon} \int_{E} \vert u_j \vert^p \dif x\right) =0,
    \end{align*}
then also $(P_{\mA} u_j )$ is $p$-equiintegrable.
\end{enumerate}
\end{lemma}
As alluded to in the introduction, Lemma~\ref{lem:proj} does not extend to $p=1$ in general, the reason being \textsc{Ornstein}'s Non-Inequality \cite{Ornstein}; also see \cite{CFM,KirchheimKristensen} for more recent approaches to the matter and \textsc{Grafakos} \cite[Thm.~4.3.4]{Grafakos} for a full characterisation of $\lebe^{1}$-multipliers.
\section{On the construction of $\mathrm{divsym}$-free truncations}\label{sec:consttrunc}
Before embarking on the proof of Theorem~\ref{thm:main2}, we comment on the underlying idea of the proof. To this end, we streamline terminology as follows. Let $\Omega$ either be $\T_{n}$ or $\R^{n}$.  Given a constant rank differential operator $\mathbb{A}$ on $\Omega$ between $\R^{\ell}$ and $\R^{d}$ and $1\leq p\leq\infty$, we define Sobolev-type spaces $\sobo^{\mathbb{A},p}(\Omega):=\{u\in\lebe^{p}(\Omega;\R^{\ell})\colon\;\mathbb{A}u\in\lebe^{p}(\Omega;\R^{d})\}$. A family of operators $(S_{\lambda})_{\lambda>0}$ with $S_{\lambda}\colon\sobo^{\mathbb{A},p}(\Omega)\to\sobo^{\mathbb{A},\infty}(\Omega)$ is called an $\sobo^{\mathbb{A},p}$-$\sobo^{\mathbb{A},\infty}$\emph{-truncation} provided there exists a constant $c=c(\mathbb{A},p)>0$ such that, for all $u\in\sobo^{\mathbb{A},p}(\Omega)$ and $\lambda>0$, 
\begin{enumerate}
\item  $\Vert S_{\lambda}u \Vert_{\lebe^{\infty}(\Omega)} + \Vert \mathbb{A}S_{\lambda}u \Vert_{\lebe^{\infty}(\Omega)} \leq c\lambda$.
\item $\Vert u- S_{\lambda}u\Vert_{\lebe^p(\Omega)} + \Vert \mathbb{A}u - \mathbb{A}S_{\lambda}u \Vert_{\lebe^p (\Omega)} \leq c\int_{\{\vert u \vert + \vert \mathbb{A} u \vert > \lambda \}} \vert u \vert^{p} + \vert \mathbb{A}u \vert^{p} \dif x$.
 \item $\mathscr{L}^n(\{ u \neq S_{\lambda}u\}) \leq \frac{c}{\lambda^{p}}\int_{\{\vert u \vert + \vert \mathbb{A}u \vert > \lambda \}} \vert u \vert^{p} + \vert \mathbb{A}u \vert^{p} \dif x$.
\end{enumerate}
If $\mathbb{A}=\nabla^{k}$, then we simply speak of a \emph{$\sobo^{k,p}$-$\sobo^{k,\infty}$-truncation}. Conversely, if $\mathbb{A}$ is a potential of the differential operator $\mathscr{A}$ having the form \eqref{eq:Aform} and $1\leq p\leq\infty$, we define $\lebe_{\mA}^{p}(\Omega):=\{u\in\lebe^{p}(\Omega;\R^{d})\colon\;\mA u=0\}$. A family of operators $(T_{\lambda})_{\lambda>0}$ with $T_{\lambda}\colon\lebe_{\mathscr{A}}^{p}(\Omega)\to\lebe_{\mA}^{\infty}(\Omega)$ is called an \emph{$\mA$-free $\lebe^{p}$-$\lebe^{\infty}$-truncation} (or simply $\mA$-free $\lebe^{\infty}$-truncation) provided there exists $c=c(\mathbb{A},p)>0$ such that the following hold for all $u\in\lebe_{\mathscr{A}}^{\infty}(\Omega)$ and $\lambda>0$:  
\begin{enumerate}
\item  $\Vert T_{\lambda}u \Vert_{\lebe^{\infty}(\Omega)} \leq c\lambda$.
\item $\Vert u- T_{\lambda}u\Vert_{\lebe^p(\Omega)} + \leq c\int_{\{\vert u \vert > \lambda \}} \vert u \vert^{p} \dif x$.
 \item $\mathscr{L}^n(\{ u \neq T_{\lambda}u\}) \leq \frac{c}{\lambda^{p}}\int_{\{\vert u \vert> \lambda \}} \vert u \vert^{p} \dif x$.
\end{enumerate}
Originally, $\sobo^{1,p}$-$\sobo^{1,\infty}$-truncations as in \textsc{Acerbi \& Fusco} \cite{AcFu88} leave $u\in\sobo^{1,p}(\Omega)$ unchanged on $\{\mathcal{M}u\leq\lambda\}\cap\{\mathcal{M}(\nabla u)\leq \lambda\}$. Here, the functions satisfy the Lipschitz estimate 
\begin{align*}
|u(x)-u(y)|\lesssim |x-y|(\mathcal{M}(\nabla u)(x)+\mathcal{M}(\nabla u)(y))\lesssim \lambda |x-y|
\end{align*}
for $\mathscr{L}^{n}$-a.e. $x,y\in\{\mathcal{M}(\nabla u)\leq\lambda\}$ and thus can be extended to a $c\lambda$-Lipschitz function $S_{\lambda}u$ by virtue of \textsc{Mc Shane}'s extension theorem \cite[Chpt.~3.1.1., Thm.~1]{EvansGariepy}. Note that, if $u$ is divergence-free, then $S_{\lambda}u$ is not in general. In view of preserving differential constraints, this necessitates a more flexible approach. Instead of appealing to the \textsc{Mc Shane} extension, one may directly perform an \textsc{Whitney}-type extension \cite{Whitney} and truncate $u\in\sobo^{1,1}(\Omega)$ on the bad set $\mathcal{O}_{\lambda} = \{ \mathcal{M} u >\lambda\} \cup \{\mathcal{M}(\nabla u) > \lambda\}$ via
\begin{align*}
\mathbf{\tilde{S}}_{\lambda}u(x) = \begin{cases} \sum_{j \in \N} \phi_j (u)_{Q_{j}}& x \in \mathcal{O}_{\lambda}, \\
                               u(x) & x \in \mathcal{O}_{\lambda}^{\complement},\end{cases}\;\;\;\text{or}\;\;\; \mathbf{S}_{\lambda}u(x) = \begin{cases} \sum_{j \in \N} \phi_j u(y_j)& x \in \mathcal{O}_{\lambda}, \\
                               u(x) & x \in \mathcal{O}_{\lambda}^{\complement}, \end{cases}
\end{align*}
where $y_{j} \in \mathcal{O}_{\lambda}^{\complement}$ are chosen suitably. Then $\mathbf{\tilde{S}}_{\lambda}$ and $\mathbf{S}_{\lambda}$ define $\sobo^{1,1}$-$\sobo^{1,\infty}$-truncations; cf.~\cite{Die13,Stein}. Setting $v = \nabla u$, this formula gives a $\curl$-free $\lebe^1$-$\lebe^{\infty}$-truncation, as $\curl(v) = 0 \Leftrightarrow v = \nabla u$ for some function $u$. Using \ref{item:P1}--\ref{item:P3}, we can, however, rewrite $\tilde{v}:= \nabla \mathbf{S}_{\lambda}u$ purely in terms of $v$, i.e.
\begin{align} \label{schweinsteiger}
    \tilde{v}(x) = \begin{cases} \sum_{i,j \in \N} \phi_i \nabla \phi_j \int_0^1 v(t y_j +(1-t) y_i) \cdot (y_i-y_j) \dif t& x \in \mathcal{O}_{\lambda}, \\
                               v(x) & x \in \mathcal{O}_{\lambda}^{\complement}. \end{cases}  
\end{align}
The key observation is that the truncation formula \eqref{schweinsteiger} does not only give a $\curl$-free $\lebe^1$-$\lebe^{\infty}$-truncation, but is stronger and gives a $\sobo^{\curl,1}$-$\sobo^{\curl,\infty}$-truncation, if we redefine the bad set to be $\tilde{\mathcal{O}}_{\lambda} := \{\mathcal{M}v > \lambda\} \cup \{\mathcal{M}\curl(v) > \lambda\}$. We are then able to formulate an $\mA$-free $\lebe^1$-$\lebe^{\infty}$-truncation of the annihilator of $\curl$, which is $\di$ in three dimensions. As discussed by the third author \cite{Schiffer21}, this approach works for all potential-annihilator pairs along the exact sequence of exterior derivatives. This is the exact sequence of differential operators starting with $\nabla$, that is 
\begin{align*}
    0 &\longrightarrow \hold^{\infty,0}(\T_n;\R) \overset{\nabla}{\longrightarrow} \hold^{\infty,0}(\T_n;\R^n) \overset{\curl }{\longrightarrow} \hold^{\infty,0}(\T_n;\R^{n \times n}_{\textup{skew}}) \longrightarrow ... \\ &\longrightarrow \hold^{\infty,0}(\T_n;\R^n) \overset{\di}{\longrightarrow} \hold^{\infty,0}(\T_n;\R) \longrightarrow 0,
\end{align*}
where $\hold^{\infty,0}(\T_n;\R^m)$ denotes the space of smooth functions on the torus with average $0$. However, $\sobo^{\mathbb{A},1}$-$\sobo^{\mathbb{A},\infty}$-truncations are also known in settings where $\mathbb{A}\neq \nabla$. In this work, we use that such a truncation exists for the symmetric gradient, i.e. $\mathbb{A}=\varepsilon= \frac{1}{2}(\nabla + \nabla^\top)$ (cf.~\cite{Ebobisse,Behn20}). We use the truncation and the exact sequence\begin{align}\label{podolski}
    0 &\longrightarrow \hold^{\infty,0}(\T_3;\R^3) \overset{\varepsilon}{\longrightarrow} \hold^{\infty,0}(\T_{3};\rdrei) \xrightarrow{\curl \curl^\top} \hold^{\infty,0}(\T_3;\rdrei) \\ &\overset{\di}{\longrightarrow} \hold^{\infty,0}(\T_3;\R^3) \longrightarrow 0 \nonumber,
\end{align}
where $\curl \curl^\top v$ for $v \in \hold^2(\R^3;\rdrei)$ is defined as \begin{align*}
   \curl \curl^\top v = \left( \begin{array}{lll} w_{2323} & w_{2331} & w_{2312} \\
                                        w_{3123} & w_{3131} & w_{3112} \\
                                        w_{1223} & w_{1231} & w_{1212}
                                        \end{array} \right), \\
                                        \quad w_{abcd} := \partial_a \partial_c v_{bd} + \partial_{b} \partial_d v_{ac} - \partial_a \partial_d v_{bc} - \partial_b \partial_c v_{ad}.
\end{align*} 
The truncation of the symmetric gradient is used to find an analogue of \eqref{schweinsteiger} for $\curl \curl^\top$, giving us a $\curl \curl^\top$-free truncation. However, this can be used to get a $\sobo^{\curl \curl^\top,1}$-$\sobo^{\curl \curl^\top,\infty}$-truncation, giving us the divergence-free $\lebe^1$-$\lebe^{\infty}$-truncation of Section \ref{sec:construction} below.
\section{Construction of the truncation and the proof of Theorem~\ref{thm:main2}} \label{sec:construction}
In this section, we establish Theorem~\ref{thm:main2}. As a main ingredient, we shall prove the following variant for smooth maps that will be shown to imply Theorem~\ref{thm:main2} in Section~\ref{sec:proofThmmain2}: 
\begin{proposition}\label{prop:smoothtrunc}
Let $w\in(\hold^{\infty}\cap\lebe^{1})(\R^{3};\mathbb{R}_{\mathrm{sym}}^{3\times 3})$ satisfy $\mathrm{div}(w)=0$. Then there exists a constant $c>0$ such that for all $\lambda>0$ there exists an open set $\mathcal{U}_{\lambda}\subset\R^{3}$ and a function $w_{\lambda}\in (\lebe^{1}\cap\lebe^{\infty})(\R^{3};\R_{\mathrm{sym}}^{3\times 3})$ with the following properties: 
\begin{enumerate}
\item\label{item:truncsmooth1} $w=w_{\lambda}$ on $\mathcal{U}_{\lambda}^{\complement}$ and $\mathscr{L}^{3}(\{w\neq w_{\lambda}\})<\frac{c}{\lambda} \int_{\{\vert w \vert > \frac{\lambda}{2}\}} \vert w \vert \dif x$.
\item\label{item:truncsmooth2} $\mathrm{div}(w_{\lambda})=0$ in $\mathscr{D}'(\R^{3};\R^{3})$.
\item\label{item:truncsmooth3} $\|w_{\lambda}\|_{\lebe^{\infty}(\R^{3})}\leq c\lambda$. 
\end{enumerate}
\end{proposition}
\subsection{Definition of $T_{\lambda}$.} Let $w\in(\hold^{\infty}\cap\lebe^{1})(\R^{3};\R_{\mathrm{sym}}^{3\times 3})$ satisfy $\mathrm{div}(w)=0$. In view of locally redefining our given map $w$ on $\mathcal{O}_{\lambda}=\{\mathcal{M}w>\lambda\}$, we put 
\begin{align}\label{eq:ABdefine}
\begin{split}
\mathfrak{A}_{\alpha,\beta}(i,j,k)(y) &:=\dashint_{\langle x_{i},x_{j},x_{k}\rangle}((y-\xi)_{\beta}w_{\alpha}(\xi)-(y-\xi)_{\alpha}w_{\beta}(\xi))\nu_{ijk}\dif^{2}\xi, \\ 
\mathfrak{B}_{\alpha}(i,j,k) & := \dashint_{\langle x_{i},x_{j},x_{k}\rangle}w_{\alpha}(\xi)\cdot\nu_{ijk}\dif^{2}\xi
\end{split}
\end{align}
provided the simplex $\langle x_{i},x_{j},x_{k}\rangle$ is non-degenerate; if it is degenerate, we then define $\mathfrak{A}_{\alpha,\beta}(i,j,k):= 0$ and $\mathfrak{B}_{\alpha}(i,j,k):=0$. Here and in what follows, we use 
\begin{align}\label{eq:outer}
\nu_{x_{i},x_{j},x_{k}}:=\nu_{ijk} := \frac{1}{2}(x_{i}-x_{j})\times(x_{k}-x_{j}),
\end{align}
provided the simplex $\langle x_{i},x_{j},x_{k}\rangle$ is non-degenerate. Consider a three-tuple
\begin{align*}
(\alpha,\beta,\gamma) \in \{(1,2,3),(2,3,1),(3,1,2)\}.
\end{align*}
For $(i,j,k)\in\mathbb{N}^{3}$ and fixed projection points $x_{l}\in Q_{l}$ for $l\in\{i,j,k\}$, we then define
\begin{align} \label{def:nondiagonal}
\begin{split}
    \widetilde{w}_{\alpha\beta}^{(k)} &= 3 \sum_{i,j\in \N}(\partial_\gamma \phi_j \partial_\alpha \phi_i \mathfrak{B}_{\alpha}(i,j,k) + \partial_\beta \phi_j \partial_\gamma \phi_i \mathfrak{B}_\beta (i,j,k)) \\ 
   & + \sum_{i,j \in \N}(\partial_{\beta \gamma} \phi_j \partial_\gamma \phi_i - \partial_{\gamma \gamma} \phi_j \partial_\beta \phi_i) \mathfrak{A}_{\beta\gamma}(i,j,k) \\
    &+\sum_{i,j\in \N}(\partial_{\alpha \gamma} \phi_j \partial_\gamma\phi_i - \partial_{\gamma \gamma} \phi_j \partial_\alpha \phi_i) \mathfrak{A}_{\gamma\alpha}(i,j,k) \\
    &+\sum_{i,j\in \N} (\partial_{\alpha \gamma} \phi_j \partial_\beta \phi_i + \partial_{\beta \gamma} \phi_j \partial_\alpha \phi_i - 2 \partial_{\alpha \beta}  \phi_j \partial_\gamma \phi_i) \mathfrak{A}_{\alpha\beta}(i,j,k).
    \end{split}
\end{align}
We define $\widetilde{w}_{\beta\alpha }^{(k)} =\widetilde{w}_{\alpha\beta }^{(k)} $ by symmetry. For the diagonal terms, we put 
\begin{align} \label{def:diagonal}
\begin{split}
    \widetilde{w}_{\alpha\alpha}^{(k)} &= 6 \sum_{i,j\in \N}\partial_\beta \phi_j \partial_\gamma \phi_i \mathfrak{B}_{\alpha}(i,j,k) \\ 
   & + 2\sum_{i,j\in\mathbb{N}}(\partial_{\gamma \gamma}\varphi_{j}\partial_{\beta}\varphi_{i}-\partial_{\beta\gamma}\varphi_{j}\partial_{\gamma}\varphi_{i})\mathfrak{A}_{\gamma\alpha}(i,j,k)\\ 
   & + 2\sum_{i,j\in\mathbb{N}}(\partial_{\beta \beta}\varphi_{j}\partial_{\gamma}\varphi_{i}-\partial_{\beta \gamma}\varphi_{j}\partial_{\beta}\varphi_{i})\mathfrak{A}_{\alpha\beta}(i,j,k). 
    \end{split}
\end{align}
Note that, since at most $\mathtt{N}$ cubes $Q_{j}$ overlap by \ref{item:W3}, each of the sums in \eqref{def:nondiagonal} and \eqref{def:diagonal} are, in a neighbourhood of each point $x \in \mathcal{O}_{\lambda}$, actually \emph{finite} sums and hence $\widetilde{w}^{(k)}:=(w_{\alpha\beta}^{(k)})_{\alpha\beta}$ is well-defined. Based on \eqref{def:nondiagonal}, we define the truncation operator $T_{\lambda}$ by 
\begin{align}\label{eq:deftruncation}
T_{\lambda}w:= w-\sum_{k}\varphi_{k}(w-\widetilde{w}^{(k)})=\begin{cases} 
w&\;\text{in}\;\mathcal{O}_{\lambda}^{\complement},\\ 
\sum_{k}\varphi_{k}\widetilde{w}^{(k)}&\;\text{in}\;\mathcal{O}_{\lambda}.
\end{cases}
\end{align}
Note that on $\mathcal{O}_{\lambda}$, $T_{\lambda}w$ is a locally finite sum of $\hold^{\infty}$-maps and thus is equally of class $\hold^{\infty}(\mathcal{O}_{\lambda};\R_{\mathrm{sym}}^{3\times 3})$.
\subsection{Auxiliary properties of $\mathfrak{A}_{\alpha,\beta}$ and $\mathfrak{B}_{\alpha}$}
In this section, we record some useful properties and auxiliary bounds on the maps $\mathfrak{A}_{\alpha, \beta}(i,j,k)$ and  the (constant) maps $\mathfrak{B}_{\alpha}(i,j,k)$ that will play an instrumental role in the proof of Proposition~\ref{prop:smoothtrunc}. We begin by gathering elementary properties of $\mathfrak{A}_{\alpha,\beta}$ and $\mathfrak{B}_{\alpha}$ to be utilised crucially when performing index permutations for the sums appearing in~\eqref{eq:deftruncation}:
	\begin{lemma}\label{lem:ABprops}
		Let $w\in\hold^{1}(\R^{3};\R_{\sym}^{3\times 3})$ satisfy $\di (w)=0$,  $i,j,k,l\in\mathbb{N}$ and define $\mathfrak{A}_{\alpha\beta},\mathfrak{B}_{\alpha}$ for $\alpha,\beta\in\{1,2,3\}$ by \eqref{eq:ABdefine}. Then the following hold:
		\begin{enumerate}
			\item\label{item:aux1} $\partial_{\alpha}\AA_{\alpha,\beta}(i,j,k)=-\mathfrak{B}_{\beta}(i,j,k)$. 
			\item\label{item:aux2} $\partial_{\beta}\AA_{\alpha,\beta}(i,j,k)=\mathfrak{B}_{\alpha}(i,j,k)$.
			\item\label{item:aux3} \emph{Antisymmetry of $\mathfrak{A}_{\alpha,\beta}$:} $\mathfrak{A}_{\alpha,\beta}(i,j,k)=-\mathfrak{A}_{\alpha,\beta}(j,i,k)=\mathfrak{A}_{\alpha,\beta}(j,k,i)$.
			\item\label{item:aux4} \emph{Antisymmetry of $\mathfrak{B}_{\alpha}$:}
			$\mathfrak{B}_{\alpha}(i,j,k)=-\mathfrak{B}_{\alpha}(j,i,k)=\mathfrak{B}_{\alpha}(j,k,i)$.
			\item\label{item:aux6A} $\di _\xi ((y-\xi)_{\beta}w_{\alpha}(\xi)-(y-\xi)_{\alpha}w_{\beta}(\xi)) =0$.
			\item\label{item:aux6}  $\mathfrak{B}_\alpha(i,j,k)-\mathfrak{B}_\alpha(l,j,k)-\mathfrak{B}_\alpha(i,l,k)-\mathfrak{B}_\alpha(i,j,l)=0.$
			\item\label{item:aux7} $\AA_{\alpha , \beta} (i,j,k)-\AA_{\alpha , \beta} (l,j,k)-\AA_{\alpha , \beta} (i,l,k)-\AA_{\alpha , \beta} (i,j,l)=0.$
		\end{enumerate}
	\end{lemma}
	\begin{proof}
		Properties~\ref{item:aux1}--\ref{item:aux4} are immediate consequences of the definitions. Property \ref{item:aux6A} holds, since
		\begin{align*}
		\di _\xi ((y-\xi)_{\beta}w_{\alpha}(\xi)-(y-\xi)_{\alpha}w_{\beta}(\xi))=-w_{\alpha \beta}(\xi)-\xi_\beta \di (w_\alpha)+w_{\beta \alpha}(\xi)+\xi _\alpha \di (w_\beta )=0.
		\end{align*}
		To prove \ref{item:aux6} we use that by the definition of $\mathfrak{B}_\alpha$ and the Gau\ss-Green theorem we have
		\begin{equation*}
		\mathfrak{B}_\alpha(i,j,k)-\mathfrak{B}_\alpha(l,j,k)-\mathfrak{B}_\alpha(i,l,k)-\mathfrak{B}_\alpha(i,j,l)=\int _{\langle x_i,x_j,x_k,x_m \rangle } \di (w_\alpha )\dif x=0.
		\end{equation*}
		Note that this calculation also holds in the case that one or multiple of the simplices are degenerate. 
		Analogously, we can prove \ref{item:aux7} by applying the Gau\ss-Green theorem as well as \ref{item:aux6A} to get
		\begin{align*}
		\AA_{\alpha , \beta} (i,j,k)-\AA_{\alpha , \beta} (l,j,k)& -\AA_{\alpha , \beta} (i,l,k)-\AA_{\alpha , \beta} (i,j,l)\\
		&=\int _{\langle x_i,x_j,x_k,x_m \rangle }\di _\xi ((y-\xi)_{\beta}w_{\alpha}(\xi)-(y-\xi)_{\alpha}w_{\beta}(\xi)) \dif x=0.
		\end{align*}
The proof is complete.
	\end{proof}
\begin{lemma}\label{lem:JogiLoew}
Let $u \in (\lebe^1 \cap \hold^1)(\R^3;\R^3)$ satisfy $\di (u)=0$ and $z_0 \in \{ \mathcal{M}_{2R} u \leq \lambda\}$, where $R>0$. Let, in addition, $x_1,x_2,x_3 \in B_R(z_{0})$. Then \begin{align} \label{eq:markusgisdol}
   \left \vert \dashint_{\langle x_{1},x_{2},x_{3}\rangle} u (\xi)\cdot\nu_{123}\dif^{2}\xi \right \vert \leq C \lambda R^2.
\end{align}
Moreover, if $w\in(\lebe^{1}\cap\hold^{1})(\R^{3};\rdrei)$ satisfies $\mathrm{div}(w)=0$ and the cubes $Q_i$, $Q_j$, $Q_k$ have non-empty intersection, $y \in Q_i \cap Q_j \cap Q_k$, we have for $\mathfrak{A}_{\alpha,\beta}$ and $\mathfrak{B}_{\alpha}$ as defined in \eqref{eq:ABdefine}
\begin{enumerate}
\item\label{item:aux7} $\vert \mathfrak{A}_{\alpha,\beta}(i,j,k)(y) \vert \leq C \lambda\ell(Q_i)^4$.
\item\label{item:aux8} $\vert \mathfrak{B}_{\alpha}(i,j,k) \vert \leq C \lambda\ell(Q_i)^3$.
\end{enumerate}
The constant $C=C(3)$ is a dimensional constant, that does not depend on $u$, $i,j,k$ and the shape of $\mathcal{O}_{\lambda}$.
\end{lemma}
\begin{proof}
Let $x_1,x_2,x_3,z_0 \in\R^{3}$ be according to the assumption, $z_0 = (z_0^1,z_0^2,z_0^3)$. Then, using that $\di u =0$, we find by Gau{\ss}' theorem 
\begin{align}\label{eq:Gaussimply}
\left\vert\dashint_{\langle x_{1},x_{2},x_{3}\rangle}u\cdot \nu_{123}\dif^{2}\xi\right\vert \leq \Big(\int_{\langle \eta,x_{2},x_{3}\rangle}+\int_{\langle x_{1},\eta,x_{3}\rangle}+\int_{\langle x_{1},x_{2},\eta\rangle}\Big)|u|\dif^{2}\xi
\end{align}
We now establish the existence of some $\eta\in\R^{3}\setminus\mathrm{aff}(x_{i},x_{j},x_{k})$ such that the right-hand side of \eqref{eq:Gaussimply} is bounded by $cR^{2}\lambda$ for some $c>0$ solely depending on the underlying space dimension $n=3$.
Denote $Q_{R}(z_{0})$ the cube centered at $z_{0}$ with faces parallel to the coordinate planes and sidelength $2R$ so that $\ball_{R}(z_{0})\subset Q_{R}(z_{0})\subset\ball_{\sqrt{3}R}(z_{0})$. Then 
\begin{align}\label{eq:flatter}
\begin{split}
\int_{B_{R}(z_{0})}&\int_{\langle x_{1},x_{2},z\rangle} |u(\xi)|\dif^{2}\xi \dif z \leq \int_{Q_{R}(z_{0})}\int_{\langle x_{1},x_{2},z\rangle} |u(\xi)|\dif^{2}\xi \dif z \\ & = \int_{z_{0}^{1}-R}^{z_{0}^{1}+R}\int_{z_{0}^{2}-R}^{z_{0}^{2}+R}\int_{z_{0}^{3}-R}^{z_{0}^{3}+R}\int_{\langle x_{1},x_{2},(z^{1},z^{2},z^{3})\rangle}|u(\xi)|\dif^{2}\xi \dif z^{3}\dif z^{2}\dif z^{1} \\ 
& \leq \int_{z_{0}^{1}-R}^{z_{0}^{1}+R}\int_{z_{0}^{2}-R}^{z_{0}^{2}+R} \int_{Q_{R}(z_{0})}|u|\dif x \dif z^{2}\dif z^{1} \\ 
& \leq (\omega_{3}\sqrt{3}R)^{3}\int_{z_{0}^{1}-R}^{z_{0}^{1}+R}\int_{z_{0}^{2}-R}^{z_{0}^{2}+R} \dashint_{B_{\sqrt{3}R}(z_{0})}|u|\dif x \dif z^{2}\dif z^{1} \\ 
& \leq (2\omega_{3}R)^{3}(2R)^{2}\mathcal{M}_{2R}u(z_{0}) \\ 
& \leq c\lambda R^{5}
\end{split}
\end{align}
Here $c>0$ is a constant solely depending on the space dimension $n=3$.  In consequence, by Markov's inequality, 
\begin{align*}
\mathscr{L}^{3}(\mathcal{U}_{x_{1},x_{2},\cdot}[u,\lambda';\ball_{R}(z_{0})]) & := \mathscr{L}^{3}\Big(\Big\{z\in\ball_{R}(z_{0})\colon\;\int_{\langle x_{1},x_{2},z\rangle} |u(\xi)|\dif^{2}\xi >\lambda'  \Big\} \Big) \\ & \!\! \stackrel{\eqref{eq:flatter}}{\leq} c\frac{\lambda}{\lambda'}R^{5}\qquad \text{for any}\;\lambda'>0,
\end{align*}
where $\mathcal{U}_{x_{1},x_{2},\cdot}[w;\ball_{R}(z_{0})]$ is defined in the obvious manner. The same argument equally works for the remaining simplices that appear in \eqref{eq:Gaussimply}, and therefore, setting 
\begin{align*}
\mathcal{U}:=\mathcal{U}_{x_{1},x_{2},\cdot}[u,\lambda';\ball_{R}(z_{0})]\cup\mathcal{U}_{\cdot,x_{2},x_{3}}[u,\lambda';\ball_{R}(z_{0})]\cup \mathcal{U}_{x_{1},\cdot,x_3}[u,\lambda';\ball_{R}(z_{0})]
\end{align*}
with an obvious definition of the sets appearing on the right-hand side, we obtain 
 \begin{align*}
\mathscr{L}^{3}(\mathcal{U}) \leq \frac{4c\lambda}{\lambda'}R^{5}.
\end{align*}
We still have the freedom to choose $\lambda'>0$ and consequently put $\lambda':=\frac{16}{\omega_{3}}c\lambda R^{2}$ so that $\mathscr{L}^{3}(\mathcal{U}^{\complement})\geq \frac{3}{4}\mathscr{L}^{3}(\ball_{R}(z_{0}))$. We may thus pick $\eta\in\ball_{R}(z_{0})\setminus\mathrm{aff}(x_{i},x_{j},x_{k})$ such that $\eta\in\mathcal{U}^{\complement}$, and by definition of $\mathcal{U}$, this choice of $\eta$ gives 
\begin{align*}
\left\vert\int_{\langle x_{1},x_{2},x_{3}\rangle}u \cdot \nu_{123}\dif^{2}\xi\right\vert \leq c\lambda R^{2}
\end{align*}
with some purely dimension dependent constant $c>0$. This completes the proof of \eqref{eq:markusgisdol}.
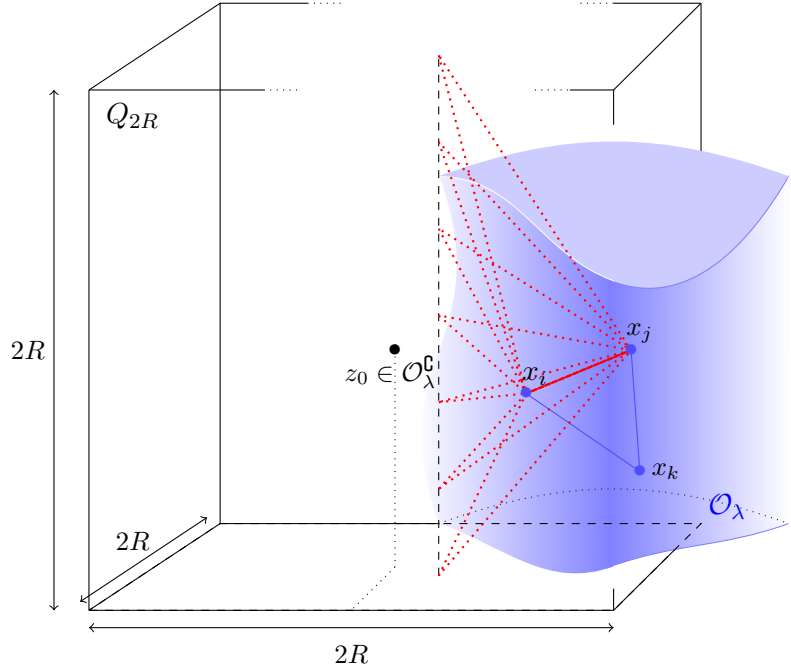
\begin{figure}
\begin{center}
\begin{tikzpicture}[scale=1.15]
\draw[black,<->] (-4.4,-1)--(-4.4,5);
\node at (-4.7,2) {$2R$};
\draw[black] (-4,-1)--(2,-1) -- (3,0) -- (-2.5,0) -- (-4,-1);
\draw[black] (3,0)--(3,6);
\draw[black] (3,6)--(2,6);
\draw[black,dotted] (2,6)--(1.6,6);
\draw[black] (-4,-1)--(-4,5);
\draw[black] (-1.5,6)--(-2.5,6);
\draw[black,dotted] (-1.5,6)--(-1.1,6);
\draw[black] (-4,5)--(-2,5);
\draw[black,dotted] (-2,5) -- (-1.6,5);
\draw[black] (-2.5,0)--(-2.5,6)--(-4,5);
\draw[black,<->] (-4,-1.2) -- (2,-1.2);
\draw[black,<->] (-4.1,-0.9) -- (-2.65,0.07);
\node at (-1,-1.5) {$2R$};
\node at (-3.5,-0.2) {$2R$};
\draw[fill=blue!20!white, blue!20!white] (0,4) to [out=20, in =160] (4,4) to (4,2) -- (0,2)--(2,2)--(0,4);
\draw[-,left color=white,right color=blue!50, shade, white] (0,0) to [bend left=20] (0,2) to [bend right=20] (0,4) to [out=0, in =160] (2,2.8) to (2,-0.5) to [out=200, in =-20] (0,0);
\draw[-,left color=blue!50,right color=white, shade, blue!50] (1.95,2.8) -- (1.95,-0.5) to [out=20, in =200] (4,0) -- (4,4) to [out=240, in =-20] (2,2.787);
\draw[black!90!white,dotted] (0,0) to [out=20, in =160] (4,0);
\draw[blue!70!white] (1,1.5) -- (2.2,2) -- (2.3,0.6) -- (1,1.5);
\node at (-0.5,2) {\textbullet};
\draw[dotted] (-0.5,2) -- (-0.5,-0.5);
\draw[dotted] (-0.5,-0.5)-- (-1,-1);
\node at (-0.55,1.75) {$z_{0}\in\mathcal{O}_{\lambda}^{\complement}$};
\draw[black,dashed] (0,-0.6) -- (0,5.4);
\draw[black,dashed] (-4,-1)--(2,-1) -- (3,0) -- (-2.5,0) -- (-4,-1);
\draw[black] (2,-1)--(2,-0.75);
\draw[black] (3,6)--(2,5)--(2,4.6);
\draw[black] (2,5)--(1.5,5);
\draw[black,dotted] (1.1,5)--(1.5,5);
\draw[dotted,thick,red] (0,5.4)--(1,1.5)--(2.2,2)--(0,5.4);
\draw[dotted,thick,red] (0,4.4)--(1,1.5)--(2.2,2)--(0,4.4);
\draw[dotted,thick,red] (0,3.4)--(1,1.5)--(2.2,2)--(0,3.4);
\draw[dotted,thick,red] (0,2.4)--(1,1.5)--(2.2,2)--(0,2.4);
\draw[dotted,thick,red] (0,1.4)--(1,1.5)--(2.2,2)--(0,1.4);
\draw[dotted,thick,red] (0,0.4)--(1,1.5)--(2.2,2)--(0,0.4);
\draw[dotted,thick,red] (0,-0.6)--(1,1.5)--(2.2,2)--(0,-0.6);
\node[black,thick] at (1.1,1.7) {\large $x_{i}$};
\node[blue!70!white] at (1,1.5) {\textbullet};
\node[blue!70!white] at (2.2,2) {\textbullet};
\node[black,thick] at (2.3,2.2) {$x_{j}$};
\node[blue!70!white] at (2.3,0.6) {\textbullet};
\node[black,thick] at (2.6,0.6) {$x_{k}$};
\draw[white] (4,4) to (4,0);
\node[blue] at (3.3,0.2) {\large $\mathcal{O}_{\lambda}$};
\node[black] at (-3.5,4.7) {\large $Q_{2R}$};
\end{tikzpicture}
\caption{The construction in the proof of Lemma~\ref{lem:JogiLoew}. The point $z_0 \in \mathcal{O}_{\lambda}^{\complement}$ is chosen such that it is close to $x_i$, $x_j$ and $x_k$ respectively. Instead of estimating the integral on the triangle with vertices $x_i$, $x_j$ and $x_k$ directly, we estimate integrals along triangles with vertices $x_i$, $x_j$ and $z \in Q_{2R}(z_0)$ (the triangles with red dashed lines) and use Gau{\ss}' theorem.}
\end{center}
\end{figure}
The estimates in \ref{item:aux7} and \ref{item:aux8} are consequences of \eqref{eq:markusgisdol}. For \ref{item:aux7} note that there is $z_0 \in \mathcal{O}_{\lambda}^{\complement}$ with $\dist(z_0, Q_i) \leq C \ell(Q_i)$, i.e. $Q_i \cap Q_j \cap Q_k \subset B_{C \ell(Q_i)}(z_0)$. Moreover, $\mathcal{M} w (z_{0})\leq \lambda$ by definition of $ \mathcal{O}_{\lambda}$ and therefore, for fixed $y \in Q_i$
\begin{align*}
    \mathcal{M}_{2R}((y-\cdot)_{\beta} w_{\alpha}(\cdot) - (y-\cdot)_{\alpha} w_{\beta})(z_{0}) \leq 2  \sup_{z \in B_R (z_{0})} \vert y - z \vert \cdot \mathcal{M} w(z_{0}). 
\end{align*}
Setting $R = C \ell(Q_i)$ and using Lemma \ref{lem:ABprops} \ref{item:aux6A} yields the estimate \ref{item:aux7}. The estimate for $\mathfrak{B}_{\alpha}$ directly uses the existence of a point $z_0 \in \mathcal{O}_{\lambda}^{\complement}$, such that $Q_i,Q_j,Q_k \subset B_{C \ell(Q_i)}(z_0)$ and that $w_{\alpha}$ is divergence-free. Applying \eqref{eq:markusgisdol} in this setting yields \ref{item:aux8}.
\end{proof}
\subsection{Elementary properties of $T_{\lambda}$}
We now record various properties of $T_{\lambda}$ that play an instrumental role in the proof of Theorem~\ref{thm:main2}. Throughout this section, we tacitly suppose that $w\in(\hold^{\infty}\cap\lebe^{1})(\R^{3};\R_{\mathrm{sym}}^{3\times 3})$, and begin with providing the corresponding $\lebe^{\infty}$-bounds: 
\begin{lemma} \label{lem:linftybound}
There exists a purely dimensional constant $c>0$ such that 
\begin{align}\label{eq:Linftybound}
\|T_{\lambda}w\|_{\lebe^{\infty}(\R^{3})}\leq c\lambda\qquad\text{holds for all}\;\lambda>0.
\end{align}
\end{lemma}
\begin{proof}
Since $|w|\leq \lambda$ on $\mathcal{O}_{\lambda}^{\complement}$, it suffices to prove $\|T_{\lambda}w\|_{\lebe^{\infty}(\mathcal{O}_{\lambda})}\leq c\lambda$ for some suitable $c>0$. Hence let $x\in\mathcal{O}_{\lambda}$. Then, by \ref{item:W1} and \ref{item:W3}, $x\in Q_{k}$ for some $k\in\mathbb{N}$, and there are only finitely many cubes $Q_{i},Q_{j}$ such that $Q_{i}\cap Q_{j}\cap Q_{k}\neq\emptyset$; note that the number of such cubes solely depends on the underlying space dimension $n=3$. For any choice of $\alpha',\beta',\gamma'\in\{1,2,3\}$ and $\ell_{1}+\ell_{2}=2$ we have 
\begin{align}\label{eq:hansiflick0}
|\varphi_{k}\partial_{\beta'}^{\ell_{1}}\varphi_{i}\partial_{\gamma'}^{\ell_{2}}\varphi_{j}|\leq c\frac{\mathbbm{1}_{Q_{i}\cap Q_{j}\cap Q_{k}}}{\ell(Q_{k})^{2}}
\end{align}
and similarly, if $\ell_{1}+\ell_{2}=3$, 
\begin{align}\label{eq:hansiflick1}
|\varphi_{k}\partial_{\beta'}^{\ell_{1}}\varphi_{i}\partial_{\gamma'}^{\ell_{2}}\varphi_{j}|\leq c\frac{\mathbbm{1}_{Q_{i}\cap Q_{j}\cap Q_{k}}}{\ell(Q_{k})^{3}}, 
\end{align}
which is seen by combining \ref{item:W4} and \ref{item:P3}. Again, $c>0$ is a purely dimensional constant. By definition of $\widetilde{w}^{(k)}$, cf.  \eqref{def:nondiagonal} and \eqref{def:diagonal}, every summand containg some $\mathfrak{B}_{\delta}(i,j,k)$, $\delta\in\{\alpha,\beta,\gamma\}$, is of the form $\varphi_{k}\partial_{\beta'}^{\ell_{1}}\varphi_{i}\partial_{\gamma'}^{\ell_{2}}\varphi_{j}\mathfrak{B}_{\delta}(i,j,k)$ with $\ell_{1}+\ell_{2}=2$. Here we may invoke Lemma~\ref{lem:JogiLoew}~\ref{item:aux7} in conjunction with~\eqref{eq:hansiflick0} to find 
\begin{align*}
|\varphi_{k}\partial_{\beta'}^{\ell_{1}}\varphi_{i}\partial_{\gamma'}^{\ell_{2}}\varphi_{j}\mathfrak{B}_{\delta}(i,j,k)|\leq c\lambda. 
\end{align*}
Conversely, every summand in the definition of $\widetilde{w}^{(k)}$ that contains some $\AA_{\delta,\kappa}(i,j,k)$, $\delta,\kappa\in\{\alpha,\beta,\gamma\}$, is of the form $
\varphi_{k}\partial_{\beta'}^{\ell_{1}}\varphi_{i}\partial_{\gamma'}^{\ell_{2}}\varphi_{j}\mathfrak{A}_{\delta,\kappa}(i,j,k)$ with $\ell_{1}+\ell_{2}=3$, and in this case Lemma~\ref{lem:JogiLoew}~\ref{item:aux8} in conjunction with \eqref{eq:hansiflick1} yields 
\begin{align*}
|\varphi_{k}\partial_{\beta'}^{\ell_{1}}\varphi_{i}\partial_{\gamma'}^{\ell_{2}}\varphi_{j}\mathfrak{A}_{\delta,\kappa}(i,j,k)|\leq c\lambda. 
\end{align*}
By the uniformly finite overlap of the cubes, this completes the proof. 
\end{proof}
\begin{lemma}\label{lem:divfreelocal}
For every $\alpha\in\{1,2,3\}$, $T_{\lambda}(w_{\alpha1},w_{\alpha2},w_{\alpha3})$ \emph{is solenoidal} on $\mathcal{O}_{\lambda}$.
\end{lemma} 
The proof of this lemma relies on a slightly elaborate computation, mutually hinging on index permutations and the properties of the maps $\mathfrak{A}_{\alpha,\beta}$ and $\mathfrak{B}_{\alpha}$ as gathered in Lemma~\ref{lem:ABprops}. For expository purposes, we thus accept Lemma~\ref{lem:divfreelocal} for the time being and refer the reader to the Appendix, Section~\ref{sec:divfreelocal}, for its proof.

\subsection{Global divsym-freeness} 
As the last ingredient towards Proposition~\ref{prop:smoothtrunc}, we next address the regularity of $\mathrm{div}(T_{\lambda}w)$: 
\begin{lemma}\label{lem:globaldivfree}
Let $w\in(\hold^{\infty}\cap\lebe^{1})(\R^{3};\R_{\mathrm{sym}}^{3\times 3})$ satisfy $\mathrm{div}(w)=0$ and define $T_{\lambda}w$ for $\lambda>0$ by \eqref{eq:deftruncation}. Then the distributional divergence of $T_{\lambda}$ is an $\R^{3}$-valued regular distribution, that is, $\mathrm{div}(T_{\lambda}w)\in\lebe^{1}(\R^{3};\R^{3})$. 
\end{lemma}
\begin{proof}
As in the previous proof, we focus on the first column of $w$; the other columns are treated by analogous means. Let $\psi\in\hold_{c}^{\infty}(\R^{3})$. By a technical, yet elementary computation to be explained in detail in the Appendix (cf.~Section~\ref{sec:App1}), we have 
\begin{align}\label{eq:Bwrite}
\begin{split}
\int_{\mathcal{O}_{\lambda}}\widetilde{w}_{1}\cdot\nabla\psi\dif x & = 2\sum_{i,j,k}\int_{\mathcal{O}_{\lambda}}\varphi_{k}(\partial_{2}\varphi_{j})(\partial_{3}\varphi_{i})\mathfrak{B}_{1}(i,j,k)\partial_{1}\psi\dif x \\ 
& + 2\sum_{i,j,k}\int_{\mathcal{O}_{\lambda}}\varphi_{k}(\partial_{3}\varphi_{j})(\partial_{1}\varphi_{i})\mathfrak{B}_{1}(i,j,k)\partial_{2}\psi\dif x \\ 
& + 2\sum_{i,j,k}\int_{\mathcal{O}_{\lambda}}\varphi_{k}(\partial_{1}\varphi_{j})(\partial_{2}\varphi_{i})\mathfrak{B}_{1}(i,j,k)\partial_{3}\psi\dif x \\ 
& =: \mathrm{I}+\mathrm{II}+\mathrm{III}. 
\end{split}
\end{align}
We focus on term $\mathrm{I}$ first and consider the functions
\begin{align}\label{eq:v1def}
\begin{split}
v_{\mathrm{I}}(y)&:=\sum_{i,j,k}v_{\mathrm{I}}^{ijk}(y):=\sum_{i,j,k}\varphi_{k}(\partial_{2}\varphi_{j})(\partial_{3}\varphi_{i})(\mathfrak{B}_{1}(i,j,k)-w_{1}(y)\cdot\nu_{ijk}), \\ 
w_{\mathrm{I}}(y)&:=\sum_{i,j,k}w_{\mathrm{I}}^{ijk}(y):=\sum_{i,j,k}\varphi_{k}(\partial_{2}\varphi_{j})(\partial_{3}\varphi_{i})(w_{1}(y)\cdot\nu_{ijk}).
\end{split}
\end{align}
We claim that $v_{\mathrm{I}}\in\sobo_{0}^{1,1}(\mathcal{O}_{\lambda})$. Note that each summand belongs to $\hold_{c}^{\infty}(\mathcal{O}_{\lambda})$, and so it suffices to establish that the overall sum in \eqref{eq:v1def} converges absolutely in $\sobo^{1,1}(\mathcal{O}_{\lambda})$. We give bounds on the single summands: For $i,j,k\in\mathbb{N}$, note that whenever $y\in Q_{i}\cap Q_{j}\cap Q_{k}$, then 
\begin{align}\label{eq:paravent}
\begin{split}
|\mathfrak{B}_{1}(i,j,k)-w_{1}(y)\cdot\nu_{ijk}| & \leq \dashint_{\langle x_{i},x_{j},x_{k}\rangle}|w_{1}(\xi)-w_{1}(y)|\,|\nu_{ijk}|\dif^{2}\xi \\ 
& \leq c\|\nabla w_{1}\|_{\lebe^{\infty}(\R^{3})}\ell(Q_{k})^{3}
\end{split}
\end{align}
as a consequence of the usual Lipschitz estimate and $\mathrm{dist}(y,\langle x_{i},x_{j},x_{k}\rangle)\leq c\ell(Q_{k})$. Now, by \ref{item:W4} and \ref{item:P3}, we consequently obtain by \eqref{eq:paravent}
\begin{align*}
&\|v_{\mathrm{I}}^{ijk}\|_{\lebe^{1}(Q_{k})} \leq c\ell(Q_{k})^{4}\|\nabla w_{1}\|_{\lebe^{\infty}(\R^{3})},\\
&\|\nabla v_{\mathrm{I}}^{ijk}\|_{\lebe^{1}(Q_{k})}  \leq c\ell(Q_{k})^{3}\|\nabla w_{1}\|_{\lebe^{\infty}(\R^{3})}, 
\end{align*}
so that, by the uniformly finite overlap of the cubes, 
\begin{align*}
\sum_{i,j,k}\|v_{\mathrm{I}}^{ijk}\|_{\sobo^{1,1}(\mathcal{O}_{\lambda})} & \leq c \sum_{k}(\ell(Q_{k})^{4}+\ell(Q_{k})^{3})\|\nabla w_{1}\|_{\lebe^{\infty}(\R^{3})} \\ 
& \leq c(1+\mathscr{L}^{3}(\mathcal{O}_{\lambda})^{\frac{1}{3}})\sum_{k}\ell(Q_{k})^{3}\|\nabla w_{1}\|_{\lebe^{\infty}(\R^{3})}\\ 
& \leq c(1+\mathscr{L}^{3}(\mathcal{O}_{\lambda})^{\frac{1}{3}})\mathscr{L}^{3}(\mathcal{O}_{\lambda})\|\nabla w_{1}\|_{\lebe^{\infty}(\R^{3})}
\end{align*}
Hence, $v_{\mathrm{I}}\in\sobo_{0}^{1,1}(\mathcal{O}_{\lambda})$. Extend $v_{\mathrm{I}}$ by zero to the entire $\R^{3}$ to obtain $\widetilde{v}_{\mathrm{I}}\in\sobo_{0}^{1,1}(\R^{3})$. Then an integration by parts yields 
\begin{align*}
\mathrm{I} & = 2\int_{\mathcal{O}_{\lambda}}v_{\mathrm{I}}\partial_{1}\psi\dif y + 2\int_{\mathcal{O}_{\lambda}}w_{\mathrm{I}}\partial_{1}\psi\dif y \\ 
& = 2\int_{\R^{3}}\widetilde{v}_{\mathrm{I}}\partial_{1}\psi\dif y + 2\int_{\mathcal{O}_{\lambda}}w_{\mathrm{I}}\partial_{1}\psi\dif y \\ 
& \!\!\!\!\!\!\!\!\!\!\!\stackrel{\widetilde{v}_{\mathrm{I}}\in\sobo_{0}^{1,1}(\R^{3})}{=} -2\int_{\R^{3}}(\partial_{1}\widetilde{v}_{\mathrm{I}})\psi\dif y + 2\int_{\mathcal{O}_{\lambda}}w_{\mathrm{I}}\partial_{1}\psi\dif y =: \mathrm{I}_{1}+\mathrm{I}_{2}, 
\end{align*}
and $\partial_{1}\widetilde{v}_{\mathrm{I}}\in\lebe^{1}(\R^{3})$. Towards term $\mathrm{I}_{2}$, observe that for all $y\in\R^{3}$, 
\begin{align}\label{eq:crossId}
\begin{split}
-2\nu_{ijk} &= -(x_{i}-x_{j})\times (x_{k}-x_{j}) \\ & = (y-x_{j})\times (x_{j}-x_{k}) + (x_{i}-y)\times (y-x_{k}) +(x_{i}-x_{j})\times(x_{j}-y), 
\end{split}
\end{align}
which follows by direct computation using that $(x_{j}-y)\times(y-x_{j})=0$. Working from the definition of $w_{\mathrm{I}}$ as in~\eqref{eq:v1def}, we consequently find by \eqref{eq:crossId}
\begin{align*}
\mathrm{I}_{2}=\int_{\mathcal{O}_{\lambda}}w_{\mathrm{I}}(y)\partial_{1}\psi\dif y & = \int_{\mathcal{O}_{\lambda}}\sum_{i,j,k}\varphi_{k}(\partial_{2}\varphi_{j})(\partial_{3}\varphi_{i})(w_{1}(y)\cdot\nu_{yx_{j}x_{k}})\partial_{1}\psi\dif y&(=0)\\ 
& + \int_{\mathcal{O}_{\lambda}}\sum_{i,j,k}\varphi_{k}(\partial_{2}\varphi_{j})(\partial_{3}\varphi_{i})(w_{1}(y)\cdot\nu_{x_{i}yx_{k}})\partial_{1}\psi\dif y&(=0)\\
& +  \int_{\mathcal{O}_{\lambda}}\sum_{i,j}\underbrace{(\partial_{2}\varphi_{j})(\partial_{3}\varphi_{i})(w_{1}(y)\cdot\nu_{x_{i}x_{j}y})}_{=:z_{\mathrm{I}}}\partial_{1}\psi\dif y =: \mathrm{I}_{3},& 
\end{align*}
where we have used that $\sum_{i}\partial_{3}\varphi_{i}=0$ on $\mathcal{O}_{\lambda}$ for the first, $\sum_{j}\partial_{2}\varphi_{j}=0$ on $\mathcal{O}_{\lambda}$ for the second and $\sum_{k}\varphi_{k}=1$ on $\mathcal{O}_{\lambda}$ for the ultimate term. By a similar argument as above, the sum in the integrand of $\mathrm{I}_{3}$ has an integrable majorant, whereby we may change the sum and the integral. Hence, integrating by parts with respect to $\partial_{2}$, 
\begin{align*}
\mathrm{I}_{3} = \mathrm{I}_{3}^{1} & := \sum_{ij}\int_{\mathcal{O}_{\lambda}}\partial_{2}(\varphi_{j}(\partial_{3}\varphi_{i})(w_{1}(y)\cdot\nu_{x_{i}x_{j}y})\partial_{1}\psi)\dif y & (=T_{1})\\
& -\sum_{ij}\int_{\mathcal{O}_{\lambda}}(\varphi_{j}(\partial_{23}\varphi_{i})(w_{1}(y)\cdot\nu_{x_{i}x_{j}y})\partial_{1}\psi)\dif y& (=T_{2})\\
& -\sum_{ij}\int_{\mathcal{O}_{\lambda}}(\varphi_{j}(\partial_{3}\varphi_{i})(\partial_{2}w_{1}(y)\cdot\nu_{x_{i}x_{j}y})\partial_{1}\psi)\dif y& (=T_{3})\\
& -\sum_{ij}\int_{\mathcal{O}_{\lambda}}(\varphi_{j}(\partial_{3}\varphi_{i})(w_{1}(y)\cdot\partial_{2}\nu_{x_{i}x_{j}y})\partial_{1}\psi)\dif y& (=T_{4})\\
& -\sum_{ij}\int_{\mathcal{O}_{\lambda}}(\varphi_{j}(\partial_{3}\varphi_{i})(w_{1}(y)\cdot\nu_{x_{i}x_{j}y})\partial_{12}\psi)\dif y& (=T_{5}), 
\end{align*}
but on the other hand, now integrating by parts with respect to $\partial_{3}$, 
\begin{align*}
\mathrm{I}_{3} = \mathrm{I}_{3}^{2} & := \sum_{ij}\int_{\mathcal{O}_{\lambda}}\partial_{3}(\varphi_{i}(\partial_{2}\varphi_{j})(w_{1}(y)\cdot\nu_{x_{i}x_{j}y})\partial_{1}\psi)\dif y& (=T_{6}) \\ 
& - \sum_{ij}\int_{\mathcal{O}_{\lambda}}(\varphi_{i}(\partial_{23}\varphi_{j})(w_{1}(y)\cdot\nu_{x_{i}x_{j}y})\partial_{1}\psi)\dif y& (=T_{7}) \\ 
& - \sum_{ij}\int_{\mathcal{O}_{\lambda}}(\varphi_{i}(\partial_{2}\varphi_{j})(\partial_{3}w_{1}(y)\cdot\nu_{x_{i}x_{j}y})\partial_{1}\psi)\dif y& (=T_{8}) \\ 
& - \sum_{ij}\int_{\mathcal{O}_{\lambda}}(\varphi_{i}(\partial_{2}\varphi_{j})(w_{1}(y)\cdot\partial_{3}\nu_{x_{i}x_{j}y})\partial_{1}\psi)\dif y& (=T_{9}) \\ 
& - \sum_{ij}\int_{\mathcal{O}_{\lambda}}(\varphi_{i}(\partial_{2}\varphi_{j})(w_{1}(y)\cdot\partial_{3}\nu_{x_{i}x_{j}y})\partial_{13}\psi)\dif y& (=T_{10}). 
\end{align*}
We then have $\mathrm{I}_{3}=\frac{1}{2}(\mathrm{I}_{3}^{1}+\mathrm{I}_{3}^{2})$. To proceed further, note that $T_{1}=T_{6}=0$ by the fundamental theorem of calculus. Moreover, $T_{2}+T_{7}=0$, which follows from permuting indices $i\leftrightarrow j$ in $T_{2}$ and using the antisymmetry property $\nu_{x_{i}x_{j}y}=-\nu_{x_{j}x_{i}y}$:
\begin{align*}
T_{2} & = -\sum_{ji}\int_{\mathcal{O}_{\lambda}}(\varphi_{i}(\partial_{23}\varphi_{j})(w_{1}(y)\cdot\nu_{x_{j}x_{i}y})\partial_{1}\psi)\dif y \\ 
& = \sum_{ji}\int_{\mathcal{O}_{\lambda}}(\varphi_{i}(\partial_{23}\varphi_{j})(w_{1}(y)\cdot\nu_{x_{i}x_{j}y})\partial_{1}\psi)\dif y = - T_{7}.
\end{align*}
For treating terms $T_{3}$ and $T_{8}$, define the smooth function $\overline{v}_{\mathrm{I}}\colon\mathcal{O}_{\lambda}\to\R$ by 
\begin{align*}
\overline{v}_{\mathrm{I}} & := \sum_{ij}(\varphi_{j}(\partial_{3}\varphi_{i})(\partial_{2}w_{1}(y)\cdot\nu_{x_{i}x_{j}y}))+(\varphi_{i}(\partial_{2}\varphi_{j})(\partial_{3}w_{1}(y)\cdot\nu_{x_{i}x_{j}y}))
\end{align*}
By an argument similar to the one employed in~\eqref{eq:v1def}ff., we have $\overline{v}_{\mathrm{I}}\in\sobo_{0}^{1,1}(\mathcal{O}_{\lambda})$. Extending it by zero to $\widetilde{\overline{v}}_{\mathrm{I}}\in\sobo_{0}^{1,1}(\R^{3})$, then obtain 
\begin{align}\label{eq:AdrianMonk}
T_{3}+T_{8} & = \int_{\R^{3}}\partial_{1}\widetilde{\overline{v}}_{\mathrm{I}} \cdot \psi\dif x.
\end{align}
Since $\mathrm{I}_{3}=\frac{1}{2}(\mathrm{I}_{3}^{1}+\mathrm{I}_{3}^{2})$, the above arguments, permuting $i \leftrightarrow j$ in $I_3^2$ and \eqref{eq:AdrianMonk} combine to
\begin{align*}
\mathrm{I}_{3}  = & -\frac{1}{2}\sum_{ij}\int_{\mathcal{O}_{\lambda}}(\varphi_{j}(\partial_{3}\varphi_{i})(w_{1}(y)\cdot ((x_{i}-x_{j})\times e_{2}))\partial_{1}\psi)\dif y & (=\tfrac{1}{2}T_{4})\\ 
& + \frac{1}{2}\sum_{ij}\int_{\mathcal{O}_{\lambda}}(\varphi_{j}(\partial_{2}\varphi_{i})(w_{1}(y)\cdot((x_{i}-x_{j})\times e_{3}))\partial_{1}\psi)\dif y & (=\tfrac{1}{2}T_{9}) \\ 
& -\frac{1}{2}\sum_{ij}\int_{\mathcal{O}_{\lambda}}(\varphi_{j}(\partial_{3}\varphi_{i})(w_{1}(y)\cdot\nu_{x_{i}x_{j}y})\partial_{12}\psi)\dif y & (=\tfrac{1}{2}T_{5})\\
& + \frac{1}{2}\sum_{ij}\int_{\mathcal{O}_{\lambda}}(\varphi_{j}(\partial_{2}\varphi_{i})(w_{1}(y)\cdot\nu_{x_{i}x_{j}y})\partial_{13}\psi)\dif y & (=\tfrac{1}{2}T_{10})\\
& + \frac{1}{2}\int_{\R^{3}}\partial_{1}\widetilde{\overline{v}}_{\mathrm{I}}\psi\dif y.
\end{align*}
Next note that
\begin{align*}
\frac{1}{2}T_{4} = & -\frac{1}{2}\sum_{ij}\int_{\mathcal{O}_{\lambda}}(\varphi_{j}(\partial_{3}\varphi_{i})(w_{1}(y)\cdot ((x_{i}-y)\times e_{2}))\partial_{1}\psi)\dif y \\ 
& -\frac{1}{2}\sum_{ij}\int_{\mathcal{O}_{\lambda}}(\varphi_{j}(\partial_{3}\varphi_{i})(w_{1}(y)\cdot ((y-x_{j})\times e_{2}))\partial_{1}\psi)\dif y \\ 
& = -\frac{1}{2}\sum_{i}\int_{\mathcal{O}_{\lambda}}((\partial_{3}\varphi_{i})(w_{1}(y)\cdot ((x_{i}-y)\times e_{2}))\partial_{1}\psi)\dif y\\
& = \frac{1}{2}\sum_{i}\int_{\mathcal{O}_{\lambda}}(\varphi_{i}\partial_{3}w_{1}(y)\cdot ((x_{i}-y)\times e_{2}))\partial_{1}\psi)\dif y\\
& + \frac{1}{2}\int_{\mathcal{O}_{\lambda}}(w_{1}(y)\cdot (-e_{3}\times e_{2})\partial_{1}\psi)\dif y \\ 
& + \frac{1}{2}\sum_{i}\int_{\mathcal{O}_{\lambda}}(\varphi_{i}w_{1}(y)\cdot ((x_{i}-y)\times e_{2})\partial_{13}\psi)\dif y.
\end{align*}
As above, we use $w\in\hold^{\infty}(\R^{3};\rdrei)$ to see that the function 
\begin{align*}
v_{\mathrm{I}}^2(y):= -\frac{1}{2}\sum_{i}\varphi_{i}\partial_{3}w_{1}(y)\cdot ((x_{i}-y)\times e_{2}))
\end{align*}
belongs to $\sobo_{0}^{1,1}(\mathcal{O}_{\lambda})$, and hence, again denoting its trivial extension to $\R^{3}$ by $v_{\mathrm{I}}^3$ and recalling that $e_{2}\times e_{3}=e_{1}$, 
\begin{align}\label{eq:AdrianMonk1}
\begin{split}
\frac{1}{2}(T_{4}) = \frac{1}{2}\int_{\R^{3}}\partial_{1}v_{\mathrm{I}}^3 \psi\dif x & + \frac{1}{2}\int_{\mathcal{O}_{\lambda}}(w_{11}(y)\partial_{1}\psi)\dif y \\ &  - \frac{1}{2}\sum_{i}\int_{\mathcal{O}_{\lambda}}(\varphi_{i}w_{1}(y)\cdot ((x_{i}-y)\times e_{2})\partial_{13}\psi)\dif y \\
\end{split}
\end{align}
Handling the summand $\frac12 T_{9}$ in the same fashion (with the roles of the indices $2$ and $3$ swapped) we arrive at
\begin{align}\label{eq:T4T10}
\begin{split}
\frac{1}{2}(T_{4}+T_{9}) = \int_{\R^{3}}\partial_{1} v_{\mathrm{I}}^4 \psi\dif x & + \int_{\mathcal{O}_{\lambda}}(w_{11}(y)\partial_{1}\psi)\dif y \\ &  + \frac{1}{2}\sum_{i}\int_{\mathcal{O}_{\lambda}}(\varphi_{i}w_{1}(y)\cdot ((x_{i}-y)\times e_{2})\partial_{13}\psi)\dif y\\
&+ \frac{1}{2}\sum_{i}\int_{\mathcal{O}_{\lambda}}(\varphi_{i}w_{1}(y)\cdot ((x_{i}-y)\times e_{3})\partial_{12}\psi)\dif y
\end{split}
\end{align}
for some $v_{\mathrm{I}}^4 \in W^{1,1}_0(\mathcal{O}_{\lambda},\R^3)$. To summarise, there exists $v_{\mathrm{I}}\in W_0^{1,1}(\mathcal{O}_{\lambda};\R^3)$, such that \begin{align} \label{eq:Isummary}
    \begin{split}
\textup{(I)}&= \int_{\R^{3}}\partial_{1} v_{\mathrm{I}} \psi\dif x  + \int_{\mathcal{O}_{\lambda}}(w_{11}(y)\partial_{1}\psi)\dif y \\ &  + \frac{1}{2}\sum_{i}\int_{\mathcal{O}_{\lambda}}(\varphi_{i}w_{1}(y)\cdot ((x_{i}-y)\times e_{2})\partial_{13}\psi)\dif y \\&
- \frac{1}{2}\sum_{i}\int_{\mathcal{O}_{\lambda}}(\varphi_{i}w_{1}(y)\cdot ((x_{i}-y)\times e_{3})\partial_{12}\psi)\dif y \\
&-\frac{1}{2}\sum_{ij}\int_{\mathcal{O}_{\lambda}}(\varphi_{j}(\partial_{3}\varphi_{i})(w_{1}(y)\cdot\nu_{x_{i}x_{j}y})\partial_{12}\psi)\dif y \\&
+\frac{1}{2}\sum_{ij}\int_{\mathcal{O}_{\lambda}}(\varphi_{j}(\partial_{2}\varphi_{i})(w_{1}(y)\cdot\nu_{x_{i}x_{j}y})\partial_{13}\psi)\dif y
\end{split}
\end{align}
The same calculations with the coordinates $1 \rightarrow 2 \rightarrow 3 \rightarrow 1$ permuted imply that there exist $v_{\mathrm{I}\mathrm{I}}, v_{\mathrm{I}\mathrm{I}\mathrm{I}} \in \sobo_0^{1,1}(\mathcal{O}_{\lambda})$, such that
 \begin{align} \label{eq:IIsummary}
    \begin{split}
\textup{(II)}&= \int_{\R^{3}}\partial_{2} v_{\mathrm{I}\mathrm{I}} \psi\dif x  + \int_{\mathcal{O}_{\lambda}}(w_{12}(y)\partial_{2}\psi)\dif y \\ &  + \frac{1}{2}\sum_{i}\int_{\mathcal{O}_{\lambda}}(\varphi_{i}w_{1}(y)\cdot ((x_{i}-y)\times e_{3})\partial_{21}\psi)\dif y \\&
- \frac{1}{2}\sum_{i}\int_{\mathcal{O}_{\lambda}}(\varphi_{i}w_{1}(y)\cdot ((x_{i}-y)\times e_{1})\partial_{23}\psi)\dif y \\
&-\frac{1}{2}\sum_{ij}\int_{\mathcal{O}_{\lambda}}(\varphi_{j}(\partial_{1}\varphi_{i})(w_{1}(y)\cdot\nu_{x_{i}x_{j}y})\partial_{23}\psi)\dif y \\&
+\frac{1}{2}\sum_{ij}\int_{\mathcal{O}_{\lambda}}(\varphi_{j}(\partial_{3}\varphi_{i})(w_{1}(y)\cdot\nu_{x_{i}x_{j}y})\partial_{21}\psi)\dif y
\end{split}
\end{align}
and
 \begin{align} \label{eq:IIIsummary}
    \begin{split}
\textup{(II)}&= \int_{\R^{3}}\partial_{3} v_{\mathrm{I}\mathrm{I}\mathrm{I}} \psi\dif x  + \int_{\mathcal{O}_{\lambda}}(w_{13}(y)\partial_{3}\psi)\dif y \\ &  + \frac{1}{2}\sum_{i}\int_{\mathcal{O}_{\lambda}}(\varphi_{i}w_{1}(y)\cdot ((x_{i}-y)\times e_{1})\partial_{32}\psi)\dif y \\&
- \frac{1}{2}\sum_{i}\int_{\mathcal{O}_{\lambda}}(\varphi_{i}w_{1}(y)\cdot ((x_{i}-y)\times e_{2})\partial_{31}\psi)\dif y \\
&-\frac{1}{2}\sum_{ij}\int_{\mathcal{O}_{\lambda}}(\varphi_{j}(\partial_{2}\varphi_{i})(w_{1}(y)\cdot\nu_{x_{i}x_{j}y})\partial_{31}\psi)\dif y \\&
+\frac{1}{2}\sum_{ij}\int_{\mathcal{O}_{\lambda}}(\varphi_{j}(\partial_{1}\varphi_{i})(w_{1}(y)\cdot\nu_{x_{i}x_{j}y})\partial_{32}\psi)\dif y
\end{split}
\end{align}
Combining \eqref{eq:Isummary}, \eqref{eq:IIsummary} and \eqref{eq:IIIsummary}, we get that there is $h \in \lebe^1(\mathcal{O}_{\lambda})$, $h= \partial_1 v_{\mathrm{I}}+ \partial_2 v_{\mathrm{I}\mathrm{I}} + \partial_3 v_{\mathrm{I}\mathrm{I}\mathrm{I}}$, such that
\begin{align} \label{eq:distdiv}
    \int_{\mathcal{O}_{\lambda}} \tilde{w}_1 \cdot \nabla \psi \dif x = \int_{\mathcal{O}_{\lambda}} h \psi \dif x + \int_{\mathcal{O}_{\lambda}} w_1 \cdot \nabla \psi \dif x.
\end{align}
Recall that $w$ satisfies $\di w=0$ and that $ \tilde{w}= w$ on $O_{\lambda}^{\complement}$. Therefore,
\begin{align*}
    \int_{\R^3} \tilde{w}_1 \cdot \nabla \psi \dif x &= \int_{\mathcal{O}_{\lambda}^{\complement}} \tilde{w}_1 \cdot \nabla \psi \dif x + \int_{\mathcal{O}_{\lambda}} \tilde{w}_1 \cdot \nabla \psi \dif x \\
    &= \int_{\mathcal{O}_{\lambda}^{\complement}} w_1 \cdot \nabla \psi \dif x +  \int_{\mathcal{O}_{\lambda}} w_1 \cdot \nabla \psi \dif x +  \int_{\mathcal{O}_{\lambda}} h \psi \dif x \\
    &=  \int_{\mathcal{O}_{\lambda}} h \psi \dif x.
\end{align*}
Therefore,  $\di(T_{\lambda} w) \in \lebe^1(\R^3;\R^3)$.
\end{proof} 
As an immediate consequence of Lemmas~\ref{lem:divfreelocal} and~\ref{lem:globaldivfree}, we obtain the following
\begin{corollary}\label{cor:divsymfreeglobal}
Let $w\in(\hold^{\infty}\cap\lebe^{1})(\R^{3};\rdrei)$ satisfy $\mathrm{div}(w)=0$ and define $T_{\lambda}w$ for $\lambda>0$ by \eqref{eq:deftruncation}. Then for $\mathscr{L}^{1}$-almost every $\lambda>0$, $\mathrm{div}(T_{\lambda}w)=0$ in $\mathscr{D}'(\R^{3};\R^{3})$. 
\end{corollary}
\begin{proof}
Observe that on $\R^3 \setminus \partial \mathcal{O}_{\lambda}$ the function $T_{\lambda} w$ is strongly differentiable and, as $u$ is solenoidal on $\R^3$ and $\di T_{\lambda} w=0$ on $\mathcal{O}_{\lambda}$ (Lemma \ref{lem:divfreelocal}), $\di(T_{\lambda}w) = 0$ on the open set $\R^3 \setminus \partial \mathcal{O}_{\lambda}$. As $w \in \hold^{\infty}$, $\mathcal{M} w \in \hold(\R^3)$ and the set 
\begin{align*}
    \{ \lambda > 0 \colon \mathscr{L}^3(\partial \mathcal{O}_{\lambda}) \neq 0\} \subset \{ \lambda > 0 \colon \mathcal{M} w = \lambda\}
\end{align*}
is an $\mathscr{L}^{1}$-null set. Hence, for all $\lambda$ not contained in this set, $\di (T_{\lambda} w) \in \lebe^1(\R^{3};\R^{3})$ and $\di (T_{\lambda}w)=0$ $\mathscr{L}^{3}$-a.e.. Thus, for $\mathscr{L}^{1}$-almost every $\lambda$, $\di (T_{\lambda} w)=0$ in $\mathscr{D}'(\R^3;\R^3)$.
\end{proof}
\begin{remarkise} It is not clear to us whether  $T_{\lambda}w$ belongs to $\sobo^{1,1}(\R^{3};\rdrei)$. This is so because $T_{\lambda}w$ is precisely constructed in a way such that handling of the divergence is possible (cf.~Lemma~\ref{lem:globaldivfree}), whereas the control of the full gradients does not come up as a consequence of Lemma~\ref{lem:ABprops}; in particular, there seems to be no reason for the series in \eqref{eq:deftruncation} to converge in $\sobo_{0}^{1,1}(\R^{3};\rdrei)$. Note that, if it did, we could directly infer from Lemma~\ref{lem:divfreelocal} that $\mathrm{div}(T_{\lambda}w)=0$. 
\end{remarkise}
\subsection{Strong stability and the proof of Proposition~\ref{prop:smoothtrunc}}\label{sec:strongstab}
In view of Lemma~\ref{lem:linftybound} and Corollary~\ref{cor:divsymfreeglobal}, Proposition~\ref{prop:smoothtrunc} will follow provided we can prove the strong stability (cf.~Proposition~\ref{prop:smoothtrunc}~\ref{item:truncsmooth1}). Towards this aim, we begin with 
\begin{lemma} \label{lem:setestimate}
Then there exists a purely dimensional constant $C>0$ such that, for each $w \in \lebe^1(\R^3;\rdrei)$ and each $\lambda >0$, we have
\begin{align*}
\mathscr{L}^3( \{ \mathcal{M} w > \lambda\}) \leq \frac{C}{\lambda}\int_{\{ |w| > \lambda/2\}} \vert w(x) \vert \dif x
\end{align*}
\end{lemma}
The rough idea of the proof of this statement is to use the weak-$(1,1)$-estimate for the Hardy-Littlewood maximal operator $\mathcal{M}$ (cf.~\eqref{eq:HLMO}) for the function $h$ defined via \begin{align}\label{eq:zhangtrick}
    h(x) = \max \{0,\vert w(x) \vert -\lambda/2\}, 
\end{align}
see \textsc{Zhang} \cite{Zhang92} for the details. 
As an important consequence of Lemma~\ref{lem:setestimate} and the $\lebe^{\infty}$-bound of $w_{\lambda}$ is the following:
\begin{corollary} \label{cor:l1estimate}
Let $w\in\lebe^{1}(\R^{3};\rdrei)$ satisfy $\mathrm{div}(w)=0$. Moreover, for $\lambda>0$, let $w_{\lambda}:=T_{\lambda}w$ be as in \eqref{eq:deftruncation}. Then we have with a purely dimensional constant $C>0$
\begin{align}\label{eq:l1estimate}
    \Vert w - w_{\lambda} \Vert_{\lebe^1(\R^{3})} \leq C \int_{\{\vert w \vert > \lambda/2\}} \vert w \vert \dif x.
\end{align}
\end{corollary}
\begin{proof}
Recall that $\mathcal{O}_{\lambda}:=\{\mathcal{M}w>\lambda\}$. By construction, $w=w_{\lambda}$ on $\mathcal{O}_{\lambda}^{\complement}$. Therefore, \begin{align}\label{eq:combinat0}
    \Vert w - w_{\lambda} \Vert_{\lebe^1(\R^{3})} \leq \int_{\mathcal{O}_{\lambda}} \vert w - w_{\lambda}\vert \dif x \leq \int_{\mathcal{O}_{\lambda}} \vert w \vert \dif x +\int_{\mathcal{O}_{\lambda}} \vert w_{\lambda}\vert \dif x.
\end{align}
On the one hand, Lemma~\ref{lem:setestimate} gives us
\begin{align}\label{eq:combinat1}
     \int_{\mathcal{O}_{\lambda}} \vert w \vert \dif x \leq \lambda \mathscr{L}^3 (\mathcal{O}_{\lambda}) + \int_{\{|w| > \lambda\}} \vert w \vert \dif x \leq C\int_{\{|w|>\lambda/2\}}|w|\dif x,
\end{align}
and, on the other hand, using Lemma \ref{lem:linftybound} and Lemma~\ref{lem:setestimate},
\begin{align}\label{eq:combinat2}
\int_{\mathcal{O}_{\lambda}} \vert w_{\lambda}\vert \dif x \leq \Vert w_{\lambda} \Vert_{\lebe^{\infty}(\R^{3})}  \mathscr{L}^{3}(\mathcal{O}_{\lambda}) \leq C \int_{\{|w|>\lambda/2\}}|w|\dif x, 
\end{align}
$C>0$ still being a purely dimensional constant. In view of \eqref{eq:combinat0}, \eqref{eq:combinat1} and \eqref{eq:combinat2}, we obtain \eqref{eq:l1estimate}, and this completes the proof. 
\end{proof}

\begin{proof}[Proof of Proposition \ref{prop:smoothtrunc}] Let $w \in (\hold^{\infty} \cap \lebe^1)(\R^3;\rdrei)$ satisfy $\di(w) =0$ and let $\lambda >0$. Pick some $\widetilde{\lambda} \in (\lambda, 2 \lambda)$ such that $\mathscr{L}^3(\partial \mathcal{O}_{\widetilde{\lambda}})=0$ and define $w_{\lambda} := T_{\widetilde{\lambda}}$ and $\mathcal{U}_{\lambda} := \mathcal{O}_{\widetilde{\lambda}}$. Then \begin{enumerate}
    \item $w = w_{\lambda}$ on $\mathcal{U}_{\lambda}^{\complement}$ by construction.
    \item Lemma \ref{lem:setestimate} implies that \begin{align*}
      \mathscr{L}^3(\{w \neq w_{\lambda}\}) \leq \frac{c}{\widetilde{\lambda}} \int_{\{\vert w \vert > \widetilde{\lambda}/2\}} \vert w \vert \dif x \leq \frac{c}{\lambda} \int_{\{\vert w \vert > \lambda/2\}} \vert w \vert \dif x  .
    \end{align*}
    \item $\di (w_{\lambda}) =0$ in $\mathscr{D}'(\R^{3};\R^{3})$ by Corollary \ref{cor:divsymfreeglobal}.
    \item $\Vert w_{\lambda} \Vert_{\lebe^{\infty}(\R^3)} \leq c \tilde{\lambda} \leq 2 c \lambda$ by Lemma \ref{lem:linftybound}.
\end{enumerate}
 To summarise, $w_{\lambda}$ satisfies all the required properties, and the proof is complete.
\end{proof}
\subsection{Proof of Theorem~\ref{thm:main2}}\label{sec:proofThmmain2}
We now establish Theorem~\ref{thm:main2}, and hence let $\lambda>0$ be given. Let $u\in\lebe^{1}(\R^{3};\R_{\mathrm{sym}}^{3\times 3})$ satisfy $\di(u) =0$ and pick a sequence $(w^{j})\subset(\hold^{\infty}\cap\lebe^{1})(\R^{3};\R_{\mathrm{sym}}^{3\times 3})$ such that $w^{j}\to u$ strongly in $\lebe^{1}(\R^{3};\R_{\mathrm{sym}}^{3\times 3})$  as $j\to\infty$, still satisfying $\di(w_j)=0$ for each $j\in\mathbb{N}$. Such a sequence can be constructed by convolution with smooth bumps.

For $\lambda >0$ consider the truncation $w_{4\lambda}^j$ of $w^j$ according to Proposition \ref{prop:smoothtrunc}. Note that this sequence
is uniformly bounded in $\lebe^{\infty}$ by $4c \lambda$. Therefore, a suitable, non-relabeled subsequence converges in the weak*-sense to some $u^{\lambda}$ in $\lebe^{\infty}(\R^3;\rdrei)$. First of all,
\begin{align*}
    \Vert u^{\lambda} \Vert_{\lebe^{\infty}(\R^3)} \leq \sup_{j \in \N} \Vert w_{4 \lambda}^j  \Vert_{\lebe^{\infty}(\R^3)}\leq 4c \lambda, \quad \di(u^{\lambda}) =0.
\end{align*}
We claim that $w^j_{4 \lambda} \to u$ strongly in $\lebe^1$ on the set $\{ \mathcal{M} u \leq 2\lambda\}$ as $j\to\infty$, and hence $u^{\lambda}=u$ on 
$\{ \mathcal{M} u \leq 2\lambda\}$. If this claim is proven, then Lemma \ref{lem:setestimate} and Corollary \ref{cor:l1estimate} imply the small change  and strong stability properties \ref{item:thmmain2}, \ref{item:thmmain3} of Theorem \ref{thm:main2}. Therefore $u^{\lambda}$ will satisfy all properties displayed in Theorem \ref{thm:main2} and thus finish the proof.

It remains to show the claim. Recall that the maximal function $\mathcal{M}$ is sublinear. Thus, \begin{align} \label{eq:complement}
    \{ \mathcal{M}w^j > 4 \lambda \} \setminus  \{ \mathcal{M} (w^j - u) > 2 \lambda\}\subset \{ \mathcal{M} u > 2 \lambda \}.
\end{align}
Note that $\mathscr{L}^{3}(\{\mathcal{M}(w^{j}-u)>2\lambda\})$ converges to zero as $j\to\infty$ since  $w^j -u \to 0$ in $\lebe^1$ and $\mathcal{M}$ is weak-$(1,1)$. After picking a suitable, non-relabeled subsequence of $(w^j)$ we may suppose that $\Vert w^j - u \Vert_{\lebe^1(\R^{3})} \leq 2^{-j} \lambda$ for all $j\in\mathbb{N}$ and hence \begin{align*}
\mathscr{L}^3 \{ \mathcal{M} (w^j - u) > 2 \lambda\} \leq C 2^{-j}\qquad\text{for all}\;j\in\mathbb{N}.
\end{align*}
Therefore, for each $J\in\mathbb{N}$, the $\mathscr{L}^{3}$-measure of the set \begin{align*}
    E_J := \bigcup_{j >J} \{ \mathcal{M} (w^j - u) > 2 \lambda\}
\end{align*}
 can be bounded by $C 2^{-J}$. Due to \eqref{eq:complement}, we have $\{ \mathcal{M} u \leq 2 \lambda \} \setminus E_J \subset \{ \mathcal{M}w^j \leq 4 \lambda \}$ for $j>J$. Let us fix $J \in \N$ and bound the $\lebe^1$-norm of $w_{4 \lambda}^j -u$ on $\{ \mathcal{M} u \leq 2 \lambda\}$ for $j >J$: \begin{align*}
     \int_{\{ \mathcal{M} u \leq 2 \lambda\}} \vert w_{4 \lambda}^j - u \vert \dif x 
       & \leq \int_{E_J} \vert w_{4 \lambda}^j - u \vert \dif x + \int_{\{ \mathcal{M} u \leq 2 \lambda\} \setminus E_J} \vert w_{4 \lambda}^j - u \vert \dif x \\
       & \leq \int_{E_J} \vert w_{4 \lambda}^j \vert + \vert u \vert \dif x +  \int_{\{\mathcal{M}w^j \leq  4 \lambda\}} \vert w_{4 \lambda}^j - u \vert \dif x \\
      & \leq C 2^{-J} \lambda + \int_{E_J} \vert u \vert \dif x +\int_{\{\mathcal{M}w^j \leq 4 \lambda\}} \vert w^j - u \vert \dif x  \\
      & \leq C 2^{-J} \lambda + \int_{E_J} \vert u \vert \dif x +\Vert w^j - u \Vert_{\lebe^1(\R^3)}.
       \end{align*}
 Letting $J \to \infty$ yields $w_{4 \lambda}^j - u \to 0$ in $\lebe^1(\{\mathcal{M} u \leq 2 \lambda \})$. As $(w_{4\lambda}^j)$ weakly*-converges to $u^{\lambda}$ in $\lebe^{\infty}(\R^3,\rdrei)$, we conclude that $u = u^{\lambda}$ on $\{\mathcal{M} u \leq 2 \lambda \}$, proving the claim.\hfill $\square$

\section{Proof of Theorem~\ref{thm:main1}}\label{sec:calcvar}
The proof of Theorem \ref{thm:main1} heavily depends on the validity of the truncation theorem \ref{thm:main2}. In fact, Theorem \ref{thm:main1} has been proven in a different setting, where the divergence is replaced by some other differential operator (e.g. \cite{Zhang97,Schiffer21}). For convenience of the reader, let us shortly present the argument here.
First of all, note that the statement of Theorem \ref{thm:main2} also holds if we consider functions $u \in \lebe^1(\T_3;\rdrei)$ instead of functions defined on $\R^3$.
\begin{proposition}
\label{prop:versionofmain}
There exists $C>0$ with the following property: For all $u \in \lebe^1(\T_3;\rdrei)$ with $\di(u)=0$ in $\mathscr{D}'(\T_3;\R^3)$ and $\lambda>0$, there is $u_{\lambda} \in \lebe^1(\T_3;\rdrei)$ satisfying \begin{enumerate}
\item $\Vert u_{\lambda} \Vert_{\lebe^{\infty}} \leq C \lambda$. \emph{($\lebe^{\infty}$-bound)}
\item $\Vert u - u_{\lambda} \Vert_{\lebe^1} \leq C \int_{\{\vert u \vert > \lambda\}} \vert u \vert \dif x $. \emph{(Strong stability)}
\item $\mathscr{L}^{3} (\{ u \neq u_{\lambda} \}) \leq C \lambda^{-1} \int_{\{\vert u \vert > \lambda\} } \vert u \vert \dif x$. \emph{(Small change)} 
\item $\di(u_{\lambda})= 0$, i.e., the differential constraint is still satisfied.
\end{enumerate}
\end{proposition}
To see this, one can either repeat the proof presented in Section \ref{sec:construction} or write $u \in \lebe^1(\T_3;\rdrei)$ as a $\mathbb{Z}^3$-periodic function on $\R^3$ and apply the obvious $\lebe^1_{\loc}$-version of Theorem~\ref{thm:main2}.
\begin{proof}[Proof of Theorem \ref{thm:main1}]
As $\mathscr{Q}_{\mathrm{sdqc}}f_1$ is a continuous symmetric div-quasiconvex function vanishing on $K$, all $y \in K^{(\infty)}$ are by definition also in $K^{(1)}$. It remains to show the other direction. Suppose that $\xi \in K^{(1)}$ and $(u_m)\subset \lebe^1(\T_3;\rdrei) \cap \mathcal{T}$ is a test sequence with
\begin{align} \label{eq:testsq}
    0= \mathscr{Q}_{\mathrm{sdqc}}f_{1}(\xi)= \lim_{m \to \infty } \int_{\T_3} f_{1}(\xi+u_m(x)) \dif x.
\end{align}
As $K$ is a compact set, we find $R>0$ with $K \subset \mathbb{B}_{R}(0)$ and $\xi \in \mathbb{B}_{R}(0)$. Thus, by \eqref{eq:testsq},
\begin{align} \label{conv}
    \lim_{m \to \infty} \int_{\{\vert u_m \vert >3R\}}   \vert u_m \vert \dif x =0.
\end{align}
Applying Proposition \ref{prop:versionofmain} gives a sequence $\widetilde{v}_m \in \lebe^{\infty}(\T_3;\rdrei)$, such that 
\begin{enumerate}
    \item\label{item:Finale1} $\di(\widetilde{v}_m) =0$.
    \item\label{item:Finale2} $\Vert \widetilde{v}_m - u_m \Vert_{\lebe^1 (\T_{3})} \rightarrow 0$ as $m \to \infty$.
    \item\label{item:Finale3} $\Vert \widetilde{v}_m \Vert_{\lebe^{\infty}(\T_{3})} \leq CR$.
\end{enumerate}
Mollification and subtracting the average gives a sequence $(v_m) \subset \lebe^{\infty}(\T_3;\rdrei) \cap \test$ also satisfying properties~\ref{item:Finale1}--\ref{item:Finale3}. Hence,
 \begin{align} \label{eq:testsq2}
    0= \mathscr{Q}_{\mathrm{sdqc}}f_{1}(\xi)= \lim_{m \to \infty } \int_{\T_3} f_{1}(\xi+v_m(x)) \dif x.
\end{align}
Take now a symmetric div-quasiconvex function $g \in \hold(\rdrei)$. We may suppose that $\max g(K) =0$ and, as $\max\,\{0,g\}$ is again symmetric div-quasiconvex, that $g\equiv 0$ on $K$. Using uniform boundedness of $v_m$ we may estimate with $C>0$ as in \ref{item:Finale3}
\begin{align}\label{eq:Finale4}
    \vert g(\xi +v_m(x)) \vert \leq \sup_{\eta \in \mathbb{B}_{(2C+1)R}(0)} \vert g(\eta) \vert < \infty.
\end{align}
Due to \eqref{eq:testsq2}, $\mathrm{dist}(\xi + v_m ,K)\to 0$ in measure, and by passing to a non-relabeled subsequence, we may assume that $\mathrm{dist}(\xi + v_m ,K)\to 0$ $\mathscr{L}^{3}$-a.e.. As $g$ is uniformly continuous on $\mathbb{B}_{(2C+1)R}(0)$, we get by \eqref{eq:Finale4} and dominated convergence
\begin{align}
    g(\xi) \leq \lim_{m \to \infty} \int_{\T_3} g(\xi + v_m(x)) \dif x \leq \int_{\T_3} \lim_{m \to \infty} g(\xi+v_m(x)) \dif x =0.
 \end{align}
Therefore, $\xi \in K^{(\infty)}$. The proof is complete. 
\end{proof}
Let us, for the sake of completeness, also discuss a proof of the statement $K^{(p)}=K^{(q)}$, $1<p,q<\infty$, which can be easily adapted to general  constant rank operators $\mA$ of the form \eqref{eq:Aform}. To this end, recall that a Borel measurable function $F\colon\R^{d}\to\R$ is called $\mA$-quasiconvex provided it satisfies \eqref{eq:defDivsymQC} for all $\xi\in\R^{d}$ and $\varphi\in\mathcal{T}$, where $\mathcal{T}=\mathcal{T}_{\mA}$ is now the set of all $\varphi\in\hold^{\infty}(\T_{n};\R^{d})$ with zero mean and $\mathscr{A}\varphi=0$. The $\mA$-quasiconvexifications $\mathscr{Q}_{\mA}f$ of functions $f$ and, for non-empty, compact sets $K\subset\R^{d}$, the corresponding sets $K^{(p)}$ for $1\leq p\leq\infty$ are defined as in \eqref{eq:DefEnvelope}, now  systematically replacing the divsym-quasiconvexity by $\mA$-quasiconvexity. In contrast to \cite{CMO19}, we even do not need to use potentials, but can directly appeal to Lemma~\ref{lem:proj}. Note that the construction of the projection $P_{\mA}$ from Lemma~\ref{lem:proj} crucially relies on Fourier multipliers and hence is not applicable for $p=1$ and $p=\infty$. Using this projection operator $P_{\mA}$, we can prove the following statement.
\begin{lemma}\label{lem:simple}
Let $\mA$ be a constant rank operator of the form \eqref{eq:Aform} and let $K \subset \R^d$ be compact. Then, for $1 < p< q <\infty$, $K^{(p)} = K^{(q)}$. 
\end{lemma}
\begin{proof} With slight abuse of notation, let $K \subset \mathbb{B}_{R}(0):=\{\eta\in\R^{d}\colon\;|\eta|<R\}$ and $y \in \mathbb{B}_{R}(0)$. Ad '$K^{(q)} \subset K^{(p)}$'. Let $y \in K^{(q)}$ and let $(u_m)\subset  \test_{\mA}$ be a test sequence such that \begin{align*}
    0= \mathscr{Q}_{\mA}f_{q}(y) = \lim_{m \to \infty} \int_{\T_n} f_{q}(y+u_m(x)) \dif x.
\end{align*}
As $K$ is compact, $(u_m)$ is bounded in $\lebe^q (\T_{n};\R^{d})$ and, as $q >p$, also bounded in $\lebe^p(\T_{n};\R^{d})$. Also note that for any $\varepsilon >0$, there is $C_{\varepsilon}>0$ such that $f_p \leq \varepsilon + C_{\varepsilon} f_q$. Therefore, \begin{align*}
\mathscr{Q}_{\mathrm{sdqc}}f_{p}(y) \leq \lim_{m \to \infty} \int_{\T_n} f_{p}(y+u_m(x)) \dif x \leq \lim_{m \to \infty} \int_{\T_n} \varepsilon + C_{\varepsilon} f_{q}(y+u_m(x)) \dif x \leq \varepsilon.
\end{align*}
Thus, $y \in K^{(p)}$. The direction $K^{(p)} \subset K^{(q)}$ uses a similar, yet easier truncation statement than Theorem \ref{thm:main1}.  Let $y \in K^{(p)}$ and let $(u_m)\subset\test_{\mA}$ be a test sequence, such that \begin{align*}
    0= \mathscr{Q}_{\mathrm{sdqc}}f_{p}(y) = \lim_{m \to \infty} \int_{\T_n} f_{p}(y+u_m(x)) \dif x.
\end{align*}
Note that $(u_m)$ is uniformly bounded in $\lebe^p(\T_{n};\R^{d})$ and that \[
\lim_{m \to \infty} \int_{\T_n} \mathrm{dist}^{p}(u_m(x),\mathbb{B}_{2R}(0)) \dif x=0.
\]
Write 
\begin{align*}
\widetilde{u}_m= \mathbbm{1}_{\{\vert u_m \vert \leq 2R\}} u_m - \dashint_{\T_n}   \mathbbm{1}_{\{\vert u_m \vert \leq 2R\}}(x) u_m(x) \dif x
\end{align*} and define $v_m := P_{\mA} \widetilde{u}_m$ with the projection operator $P_{\mA}$ from Lemma~\ref{lem:proj}. Observe that \begin{enumerate}
    \item $\mA v_m =0$ by Lemma~\ref{lem:proj}~\ref{item:proj1}.
    \item $(\widetilde{u}_m)$ is bounded in $\lebe^{\infty}(\T_{n};\R^{d})$ and $q$-equiintegrable. Since $1<q<\infty$, the projection $P_{\mA}\colon\lebe^{q}(\T_{n};\R^{d})\to\lebe^{q}(\T_{n};\R^{d})$ is bounded, $(v_m)$ is bounded in $\lebe^q(\T_{n};\R^{d})$, $q$-equiintegrable by Lemma~\ref{lem:proj}~\ref{item:proj3}, Moreover, by Lemma~\ref{lem:proj}~\ref{item:proj2} and $1<p<\infty$, 
\begin{align*}
        \Vert u_m - v_m \Vert_{\lebe^p (\T_{n})} & \leq \Vert u_m - \widetilde{u}_m \Vert_{\lebe^p (\T_{n})} + \Vert \widetilde{u}_m - v_m \Vert_{\lebe^p (\T_{n})} \\ & \leq \Vert u_m - \widetilde{u}_m \Vert_{\lebe^p (\T_{n})} + C_{\mA,p}\Vert \mA(\widetilde{u}_m -u_{m})\Vert_{\sobo^{-k,p}(\T_{n})} \\ 
        & \leq C_{\mA,p}\Vert u_m - \widetilde{u}_m \Vert_{\lebe^p (\T_{n})} \to 0.
    \end{align*}
\end{enumerate}
Hence, also \begin{align*}
    \lim_{m \to \infty} \int_{\T_n} f_{p}(y+v_m(x)) \dif x =0.
\end{align*}
We conclude that $f_{q}(y+v_m) \to 0$ in measure. Combining this with the $\lebe^q$-boundedness and $q$-equiintegrability, we obtain
\begin{align*}
    \lim_{m \to \infty} \int_{\T_n} f_{q}(y+v_m(x)) \dif x =0.
\end{align*}
Therefore, $y \in K^{(q)}$, concluding the proof.
\end{proof}

\section{Potential truncations}\label{sec:revisit}
In this concluding section, we come back to the potential truncations alluded to in the introduction and discuss the limitations of this strategy in view of Theorems~\ref{thm:main1} and \ref{thm:main2}. Let $\mA$ be a constant rank operator in the sense of Section~\ref{sec:HA}. Recall that the potential truncation strategy, originally pursued in \cite{BDS13} for $\mA=\mathrm{div}$, is to represent $u\in\lebe^{p}(\T_{n};\R^{d})$ with $\mA u=0$ and $\dashint_{\T_{n}}u\dif x =0$ as $u=\mathbb{A}v$ for some potential $\mathbb{A}$ of order $l\in\mathbb{N}$ (cf. Lemma~\ref{lem:potential}) and then performing a $\sobo^{l,p}$-$\sobo^{l,\infty}$-truncation on the potential $v$. We then write with slight abuse of notation\footnote{The notation $\mathbb{A}^{-1}$ is only symbolic as $\mathbb{A}$ might be non-invertible.} $v=\mathbb{A}^{-1}u$. Since it is of independent interest but also motivates the need for a different strategy for Theorem~\ref{thm:main2} for $p=1$, we record 
\begin{proposition} \label{prop:modtrunc}
Let $\mA$ be a constant rank differential operator of order $k \in \N$ and $\mathbb{A}$ be a potential of $\mA$ of order $l \in \N$. Let $1<p< \infty$. Then there exists a constant $C>0$ such that the following hold: If $u \in \lebe^{p}(\T_n;\R^d) \cap \ker \mA$ and $\lambda>0$
then there exists $u_\lambda \in \lebe^{\infty}(\T_n;\R^d) \cap \ker \mA$ satisfying the
\begin{enumerate}
    \item\label{item:modtrunc1} \emph{$\lebe^{\infty}$-bound:} $\Vert u_\lambda \Vert_{\lebe^{\infty}(\T_n)} \leq C\lambda$. 
    \item\label{item:modtrunc2} \emph{weak stability:} \begin{align*}
    \Vert u_\lambda - u \Vert_{\lebe^{p}(\T_n)}^{p} \leq C \int_{\{ \sum_{j=0}^{l}|\nabla^{j} \circ \mathbb{A}^{-1}u| >\lambda \}} \sum_{j=0}^{l} \vert \nabla^j \circ \mathbb{A}^{-1} u \vert^{p} \dif x .
    \end{align*}
    \item\label{item:modtrunc3} \emph{small change:}
    \begin{align*}
        \mathscr{L}^n (\{ u_\lambda \neq u \}) \leq \frac{C}{\lambda^{p}} \int_{\{ \sum_{j=0}^{l}|\nabla^{j} \circ \mathbb{A}^{-1}u| >\lambda \}} \sum_{j=0}^{l} \vert \nabla^j \circ \mathbb{A}^{-1} u \vert^{p} \dif x
    \end{align*}
\end{enumerate}
\end{proposition}
For simplicity, we state this result on $\T_{n}$; a version on $\R^{n}$ follows by analogous means.
\begin{proof}
We start by outlining the $\sobo^{m,p}$-$\sobo^{m,\infty}$-truncation that seems hard to be traced in the literature; here, we choose a direct approach instead of appealing to \textsc{McShane}-type extensions. Let $m\in\mathbb{N}$. For $v\in\sobo^{m,p}(\T_{n};\R^{d})$, let $\mathcal{O}_{\lambda}:=\{\sum_{j=0}^{m}\mathcal{M}(\nabla^{j}v)>\lambda\}$. Since the sum of lower semicontinuous functions is lower semicontinuous, $\mathcal{O}_{\lambda}$ is open. We choose a Whitney decomposition $\mathscr{W}=(Q_{j})$ of $\mathcal{O}_{\lambda}$ satisfying \ref{item:W1}--\ref{item:W4}, and a partition of unity $(\varphi_{j})$ subject to $\mathscr{W}$ with \ref{item:P1}--\ref{item:P3}. We note that the Whitney cover can be arranged in a way such that $\mathscr{L}^{n}(Q_{j}\cap Q_{j'})\geq c\max\{\mathscr{L}^{n}(Q_{j}),\mathscr{L}^{n}(Q_{j'})\}$ holds for some $c=c(n)>0$ and all $j,j'\in\mathbb{N}$ such that $Q_{j}\cap Q_{j'}\neq\emptyset$. For each $j\in\mathbb{N}$, we then denote $\pi_{j}[v]$ the $(m-1)$-th order averaged Taylor polynomial of $v$ over $Q_{j}$; cf. \cite[Chpt.~1.1.10]{Mazya}. In particular, we have the scaled version of Poincar\'{e}'s inequality
\begin{align}\label{eq:scaledPoinca}
\dashint_{Q_{j}}|\partial^{\alpha}(w-\pi_{j}[w])|^{q}\dif x \leq c(q,m,n)\ell(Q_{j})^{q(m-|\alpha|)}\dashint_{Q_{j}}|\nabla^{m}w|^{q}\dif x
\end{align}
for all $1\leq q <\infty$, $w\in\sobo^{m,q}(\T_{n};\R^{d})$ and $|\alpha|\leq m$. We then put 
\begin{align}\label{eq:modsum}
v_{\lambda}:=v-\sum_{j}\varphi_{j}(v-\pi_{j}[v])=\begin{cases} 
v&\;\text{in}\;\mathcal{O}_{\lambda}^{\complement},\\
\sum_{j}\varphi_{j}\pi_{j}[v]&\;\text{in}\;\mathcal{O}_{\lambda}. 
\end{cases}
\end{align}
Then $v_{\lambda}\in\sobo^{m,p}(\T_{n};\R^{d})$, which can be seen as follows: On $\mathcal{O}_{\lambda}$, $v_{\lambda}$ is a locally finite sum of $\hold^{\infty}$-maps and hence of class $\hold^{\infty}$ too. For an arbitrary $|\alpha|\leq m$, \eqref{eq:scaledPoinca} yields
\begin{align*}
\sum_{j}\|\partial^{\alpha}(\varphi_{j}(v-\pi_{j}[v])\|_{\lebe^{q}(\mathcal{O}_{\lambda})}^{q}  &\stackrel{\text{\ref{item:P3}}}{\leq} \sum_{j}\sum_{\beta+\gamma=\alpha}\frac{c(n,q)}{\ell(Q_{j})^{q(|\beta|+|\gamma|)}}\ell(Q_{j})^{q|\gamma|}\|\partial^{\gamma}(v-\pi_{j}[v])\|_{\lebe^{q}(Q_{j})}^{q} \\ & \;\,\leq c(n,m,q)\sum_{j}\ell(Q_{j})^{q(m-|\alpha|)}\|\nabla^{m}v\|_{\lebe^{q}(Q_{j})}^{q}\\ 
& \stackrel{\text{\ref{item:W3}}}{\leq} c(n,m,q)\mathscr{L}^{n}(\mathcal{O}_{\lambda})^{\frac{q(m-|\alpha|)}{n}}\|\nabla^{m}v\|_{\lebe^{q}(\mathcal{O}_{\lambda})}^{q}. 
\end{align*}
In conclusion, applying the previous inequality with $q=1$, on $(0,1)^{n}$ the series in \eqref{eq:modsum} converges absolutely in $\sobo_{0}^{m,1}((0,1)^{n};\R^{d})$ and hence $v_{\lambda}\in\sobo^{m,1}(\T_{n};\R^{d})$; then applying the previous inequality with $q=p$ yields $v_{\lambda}\in\sobo^{m,p}(\T_{n};\R^{d})$. Whenever $x\in Q_{j_{0}}$ for some $j_{0}\in\mathbb{N}$, \ref{item:W2} implies that we may blow up $Q_{j_{0}}$ by a fixed factor $c>0$ so that $cQ_{j_{0}}\cap\mathcal{O}_{\lambda}^{\complement}\neq\emptyset$. Fix some $z\in cQ_{j_{0}}\cap\mathcal{O}_{\lambda}^{\complement}$. Then, for some $c'=c'(n)>0$, $Q_{j_{0}}\subset\ball_{c'\ell(Q_{j_{0}})}(z)$ and so
\begin{align}\label{eq:steinmeier}
\dashint_{Q_{j_{0}}}|\partial^{\alpha}v|\dif x \leq c(n)\dashint_{\ball_{c'(n)\ell(Q_{j_{0}})}(z)}|\partial^{\alpha}v|\dif x \leq c(n)\mathcal{M}(\nabla^{|\alpha|}v)(z)\leq c(n)\lambda
\end{align}
for all $|\alpha|\leq m$. Now let $Q_{j}\in\mathscr{W}$ be another cube with $Q_{j}\cap Q_{j_{0}}\neq\emptyset$; by \ref{item:W3}, there are only $\mathtt{N}=\mathtt{N}(n)<\infty$ many such cubes. Since $\nabla^{m}\pi_{j_{0}}[v]=0$ and $\sum_{j}\varphi_{j}=1$ on $\mathcal{O}_{\lambda}$,
\begin{align}\label{eq:stab1}
\begin{split}
|\nabla^{m}&v_{\lambda}(x)| \leq \left\vert \sum_{j\colon\;Q_{j}\cap Q_{j_{0}}\neq\emptyset}\nabla^{m}(\varphi_{j}(\pi_{j}[v]-\pi_{j_{0}}[v]))(x) \right\vert \\ 
& \!\!\stackrel{\text{\ref{item:P3}}}{\leq} c\sum_{\substack{j\colon\;Q_{j}\cap Q_{j_{0}}\neq\emptyset\\ |\alpha|+|\beta|=m}}\frac{1}{\ell(Q_{j})^{|\alpha|}}\|\nabla^{|\beta|}(\pi_{j}[v]-\pi_{j_{0}}[v])\|_{\lebe^{\infty}(Q_{j}\cap Q_{j_{0}})} \\ 
& \!\stackrel{(*)}{\leq} c\sum_{\substack{j\colon\;Q_{j}\cap Q_{j_{0}}\neq\emptyset\\ |\alpha|+|\beta|=m}}\frac{1}{\ell(Q_{j})^{|\alpha|}}\dashint_{Q_{j}}|\nabla^{|\beta|}(\pi_{j}[v]-v)|\dif x + \dashint_{Q_{j_{0}}}|\nabla^{|\beta|}(v-\pi_{j_{0}}[v])|\dif x \\ 
& \leq c\sum_{\substack{j\colon\;Q_{j}\cap Q_{j_{0}}\neq\emptyset}}\dashint_{Q_{j}}|\nabla^{m}v|\dif x \;\;\;\;\text{(by~\eqref{eq:scaledPoinca})} \\ & \leq c\lambda \;\;\;\;\;\;\;\;\;\;\;\;\;\;\;\;\;\;\;\;\;\;\;\;\;\;\;\;\;\;\;\;\;\;\;\;\;\;\;\;\;\;\,\text{(by \eqref{eq:steinmeier} and \ref{item:W3})},
\end{split}
\end{align}
where have used at $(*)$ that on the polynomials of degree at most $(m-1)$ on cubes, all norms are equivalent (in particular, the $\lebe^{1}$- and $\lebe^{\infty}$-norms), and scaling (recall that $\mathscr{L}^{n}(Q_{j}\cap Q_{j_{0}})\geq c\max\{\mathscr{L}^{n}(Q_{j}),\mathscr{L}^{n}(Q_{j_{0}})\}$) whenever $Q_{j}\cap Q_{j_{0}}\neq\emptyset$, and \ref{item:W3}). Hence, 
\begin{itemize}
\item[(i)] $\|\nabla^{m}v\|_{\lebe^{\infty}(\T_{n})}\leq c(m,n)\lambda$, 
\item[(ii)] $\mathscr{L}^{n}(\{u\neq u_{\lambda}\})\leq\frac{c(m,n,p)}{\lambda^{p}}\sum_{j=0}^{m}\|\nabla^{j}v\|_{\lebe^{p}(\T_{n})}^{p}$. 
\end{itemize}
We now let $u\in\lebe^{p}(\T_{n};\R^{d})\cap\ker\mA $ satisfy $\int_{(0,1)}u\dif x = 0$. Since $\mathbb{A}^{-1}$ has a Fourier symbol of class $\hold^{\infty}$ off zero and homogeneous of degree $(-l)$, $\nabla^{l}\circ\mathbb{A}^{-1}$ has a Fourier symbol of class $\hold^{\infty}$ off zero and  homogeneous of degree zero. By Mihlin's theorem (cf.~\cite{Stein}), applicable because of $1<p<\infty$ and by Poincar\'{e}'s inequality, we thus find that $\mathbb{A}^{-1}u\in\sobo^{l,p}(\T_{n})$ together with $\|\mathbb{A}^{-1}u\|_{\sobo^{l,p}(\T_{n})}\leq c\|u\|_{\lebe^{p}(\T_{n})}$. We then perform a $\sobo^{l,p}$-$\sobo^{l,\infty}$-truncation on $\mathbb{A}^{-1}u$ as in the first part of the proof,  and define $u_{\lambda}:=\mathbb{A}v_{\lambda}$. By the properties gathered in the first part of the proof,  we may employ \textsc{Zhang}'s trick (see \eqref{eq:zhangtrick} ff.) to conclude \ref{item:modtrunc2} and \ref{item:modtrunc3} as well. The proof is complete. 
\end{proof}
\begin{remarkise}[Strong stability and $1<p<\infty$ versus $p=1$]\label{rem:nonnatural}
It is clear from the above proof that the potential truncation only works fruitfully in the case $1<p<\infty$ by the entering of Mihlin's theorem; indeed, the operator $\mathbb{A}^{-1}$ is defined via Fourier multipliers and by Ornstein's Non-Inequality, we cannot conclude that $\mathbb{A}^{-1}u\in\sobo^{l,1}$ provided $u\in\lebe^{1}$. However, the potential truncations from Proposition~\ref{prop:modtrunc} do \emph{not satisfy the strong stability property} $\|u-u_{\lambda}\|_{\lebe^{p}(\T_{n})}^{p}\leq C\int_{\{|u|>\lambda\}}|u|^{p}\dif x$.  The underlying reason is that $\nabla^{l}\circ\mathbb{A}^{-1}$ is a Fourier multiplication operator with symbol smooth off zero and homogeneous of degree zero; by Ornstein's Non-Inequality, we only have that $\nabla^{l}\circ\mathbb{A}^{-1}\colon\lebe^{\infty}\to\mathrm{BMO}$ in general, and here $\mathrm{BMO}$ \emph{cannot} be replaced by $\lebe^{\infty}$. The potential truncation is performed on the sets where $\sum_{j=0}^{l}\mathcal{M}(\nabla^{j}\circ\mathbb{A}^{-1}u)>\lambda$.
Thus, even if $u\in\lebe^{\infty}(\T_{n};\R^{d})$ is $\mA$-free with $\|u\|_{\lebe^{\infty}(\T_{n})}\leq \lambda$, the potential truncation might modify $u$ regardless of $\lambda>0$ and hence strong stability cannot be achieved. As established by \textsc{Conti, M\"{u}ller and Ortiz} \cite{CMO19}, in the case $1<p<\infty$ this issue still can be circumvented to arrive at Lemma~\ref{lem:simple}, but in the context of $p=1$ the underlying techniques break down. In essence, this was the original motivation for the different proof displayed in Sections~\ref{sec:consttrunc} and \ref{sec:construction}.
\end{remarkise}
We conclude the paper with possible other approaches and extensions of Theorem~\ref{thm:main2}. 
\begin{remarkise}\label{rem:BDF}
As mentioned in the introduction, \cite{BDF12} constructs a divergence-free $\sobo^{1,p}$-$\sobo^{1,\infty}$-truncation. Here a Whitney-type truncation is performed first, leading to a non-divergence-free truncation. To arrive at a divergence-free truncation, the local divergence overshoots are then corrected by substracting special solutions of suitable divergence equations. This is achieved by invoking the \emph{Bogovski\u{\i}} operator. In our situation, the main drawback of the Bogovski\u{\i} operator is that if equations $\mathrm{div}(Y)=f$ for $f\colon (0,1)^{n}\to\R^{n}$ are considered, then the solution $Y$ obtained by the Bogovski\u{\i} operator does not necessarily take values in $\R_{\mathrm{sym}}^{n\times n}$; note that passing to the symmetric part $Y^{\mathrm{sym}}$ destroys the validity of the divergence equation. While this, in principle, could be repaired by passing to different solution operators, the method requires tools that are not fully clear to us in the present lower regularity context of Theorem~\ref{thm:main2}. With our proof in Section~\ref{sec:construction} being taylored to divergence constraints, in principle it can be modified to yield divergence-free $\sobo^{1,p}$-$\sobo^{1,\infty}$-truncations as well. We shall pursue this together with possible extensions of the approach in \cite{BDF12} elsewhere. 
\end{remarkise}
We finally comment on possible extensions of the strategy explained in Section~\ref{sec:consttrunc} in the $\mA$-free context. As discussed in Section~\ref{sec:consttrunc}, the key ingredients for the underlying construction is the availability of a $\sobo^{\mathbb{A},1}$-$\sobo^{\mathbb{A},\infty}$-truncation for a suitable operator $\mathbb{A}$ and the analogue of \eqref{schweinsteiger}. Since for the class of $\mathbb{C}$-elliptic operators\footnote{This means that $\mathbb{A}[\xi]$ has trivial nullspace for each $\xi\in\mathbb{C}^{n}\setminus\{0\}$, cf. \textsc{Smith} \cite{Smith70}.}, such truncations are available \cite{Behn20} (see \cite{GMR19} for a similar strategy in view of extension operators), this should then give truncations along the whole exact sequence starting with $\mathbb{A}$. As a consequence, we expect Theorem \ref{thm:main2} to hold true for all operators with constant rank in $\C$:
\begin{conj}[Theorem~\ref{thm:main2} for operators with constant rank in $\C$] \label{speculations}
Let \begin{align*}
    0 \rightarrow \hold^{\infty,0}(\T_n;\R^{d_0}) \xrightarrow{\mA_1} \hold^{\infty,0}(\T_n;\R^{d_1}) \xrightarrow{\mA_2} ... \xrightarrow{\mA_k} \hold^{\infty,0}(\T_n;\R^{d_k}) \xrightarrow{\mA_{k+1}} ...
\end{align*}
be an exact sequence of differential operators with constant rank in $\C$, in particular, $\mA_1$ being $\C$-elliptic. This is equivalent to \begin{align*}
    0 \rightarrow \C^{d_0} \xrightarrow{\mA_1[\xi]}  \C^{d_1} \xrightarrow{\mA_2[\xi]} \C^{d_2} \xrightarrow{\mA_3[\xi]} ... \xrightarrow{\mA_k[\xi]} \C^{d_k} \xrightarrow{\mA_{k+1}[\xi]} ...
\end{align*}
being exact for all $\xi \in \C^n \setminus \{0\}$. Then for any differential operator $\mA_k$ contained in this exact sequence there is $C_k >0$, such that for $u \in \lebe^1(\T_{n};\R^{d_k})$ with $\mA_k u =0$ in $\mathscr{D}'(\T_n;\R^{d_{k+1}})$ and $\lambda>0$, there is $u_{\lambda} \in \lebe^1(\R^n;\R^{d_k})$ satisfying \begin{enumerate}
\item $\Vert u_{\lambda} \Vert_{\lebe^{\infty}} \leq C \lambda$. ($\lebe^{\infty}$-bound)
\item $\Vert u - u_{\lambda} \Vert_{\lebe^1} \leq C \int_{{\vert u \vert > \lambda}} \vert u \vert \dif x $. (Strong stability)
\item \label{property:bonus}$\mathscr{L}^{n} (\{ u \neq u_{\lambda} \}  \leq C \lambda^{-1} \int_{{\vert u \vert > \lambda}} \vert u \vert \dif x$. (Small change) 
\item $\mA_k u_{\lambda} = 0$, i.e. the differential constraint is still satisfied.
\end{enumerate}
If any differential operator $\mA$ with constant rank over $\C$ is a part of such an exact sequence, this means that the $\mA$-free truncation is possible for every such operator. 
\end{conj}

\section{Appendix}\label{sec:appendix}

In this appendix, we give the computational details for some of the identities used in the main part of the paper. 
We will need the following 
	\begin{lemma}\label{lem:summationABvanish}
		Let $a,b,c \in \mathbb{N}^3$ be multi-indices with $|a|,|b|,|c|\geq 1$ and $\alpha,\beta \in \lbrace 1,2,3\rbrace$. Then on the set $\mathcal{O}_\lambda$ have
		\begin{equation}\label{eq:summationBvanish}
		\sum _{ijk}\partial _a \phi _k \partial _b \phi _j \partial _c \phi _i \mathfrak{B}_\alpha (i,j,k)=0,
		\end{equation}
		and
		\begin{equation}\label{eq:summationAvanish}
		\sum _{ijk}\partial _a \phi _k \partial _b \phi _j \partial _c \phi _i \mathfrak{A}_{\alpha ,\beta} (i,j,k)=0.
		\end{equation}
	\end{lemma}
	\begin{proof}
		Recall from the definition of the $\phi _l$ that $\sum \phi _l\equiv 1$ on $\mathcal{O}_\lambda$. We therefore have $\sum \partial _a \phi_l=\sum \partial _b \phi _l =\sum \partial _c \phi _l=0$. We can use this to get
		\begin{align*}
		&\sum _{ijk}\partial _a \phi _k \partial _b \phi _j \partial _c \phi _i \mathfrak{B}_\alpha (i,j,k)\\
		= &\sum _{ijkm}\partial _a \phi _k \partial _b \phi _j \partial _c \phi _i \Big( \mathfrak{B}_\alpha (i,j,k)-\mathfrak{B}_\alpha (m,j,k)-\mathfrak{B}_\alpha (i,m,k)-\mathfrak{B}_\alpha (i,j,m) \Big)
		\end{align*}
		Now \eqref{eq:summationBvanish} follows from Lemma \ref{lem:ABprops} \ref{item:aux6}; \eqref{eq:summationAvanish} can be shown completely analogously. 
	\end{proof}
\subsection{Proof of Lemma~\ref{lem:divfreelocal}}\label{sec:divfreelocal}
We focus on the case $\alpha=1$. Let thus $D:=\di  (T_\lambda w)_1 $. To avoid notational overload we omit the arguments $i,j$ and $k$ of $\AA _{\alpha , \beta} (i,j,k)$ and $\mathfrak{B}_\alpha  (i,j,k)$ in the following equation. Thus, all $\AA _{\alpha , \beta}$ and $\mathfrak{B}_\alpha$ implicitly depend on the summation indices. By the definition of $T_\lambda w$ on $\mathcal{O}_\lambda$, \eqref{eq:deftruncation}, we have 
		\begin{align*}
		D =& 6 \sum_{ijk}\partial_{1}(\varphi_{k}\partial_2 \phi_j \partial_3 \phi_i )\mathfrak{B}_{1} &(=T_{1}) \\ 
		& + 2\sum_{ijk}\partial_{1}(\varphi_{k}(\partial_{33}\varphi_{j}\partial_{2}\varphi_{i}-\partial_{23}\varphi_{j}\partial_{3}\varphi_{i}))\mathfrak{A}_{3,1} & (=T_{2})\\ 
		& + 2\sum_{ijk}\varphi_{k}(\partial_{33}\varphi_{j}\partial_{2}\varphi_{i}-\partial_{23}\varphi_{j}\partial_{3}\varphi_{i})\partial_{1}\mathfrak{A}_{3,1} &(=T_{3})\\
		& + 2\sum_{ijk}\partial_{1}(\varphi_{k}(\partial_{22}\varphi_{j}\partial_{3}\varphi_{i}-\partial_{23}\varphi_{j}\partial_{2}\varphi_{i}))\mathfrak{A}_{1,2}&(=T_{4})\\ 
		& +2\sum_{ijk}\varphi_{k}(\partial_{22}\varphi_{j}\partial_{3}\varphi_{i}-\partial_{32}\varphi_{j}\partial_{2}\varphi_{i})\partial_{1}\mathfrak{A}_{1,2}&(=T_{5})\\
		& + 3 \sum_{ijk}\partial_{2}(\varphi_{k}\partial_3 \phi_j \partial_1 \phi_i )\mathfrak{B}_{1} &(=T_{6})\\ 
		& + 3\sum_{ijk}\partial_{2}(\varphi_{k}\partial_2 \phi_j \partial_3 \phi_i)\mathfrak{B}_2  & (=T_{7})\\ 
		& + \sum_{ijk}\partial_{2}(\varphi_{k}(\partial_{23} \phi_j \partial_3 \phi_i - \partial_{33} \phi_j \partial_2 \phi_i)) \mathfrak{A}_{2,3} &(=T_{8})\\
		& + \sum_{ijk}\varphi_{k}(\partial_{23} \phi_j \partial_3 \phi_i - \partial_{33} \phi_j \partial_2 \phi_i) \partial_{2}\mathfrak{A}_{2,3}&(=T_{9})\\
		&+\sum_{ijk}\partial_{2}(\varphi_{k}(\partial_{13}\phi_j \partial_3\phi_i - \partial_{33} \phi_j \partial_1 \phi_i)) \mathfrak{A}_{3,1}&(=T_{10}) \\
		&  +\sum_{ijk}\varphi_{k}(\partial_{13}  \phi_j \partial_3\phi_i - \partial_{33} \phi_j \partial_1 \phi_i) \partial_{2}\mathfrak{A}_{3,1}&(=T_{11})\\
		&+\sum_{ijk} \partial_{2}(\varphi_{k}(\partial_{13}  \phi_j \partial_2 \phi_i + \partial_{23} \phi_j \partial_1 \phi_i - 2 \partial_{12} \phi_j \partial_3 \phi_i)) \mathfrak{A}_{1,2} &(=T_{12})\\ 
		& + \sum_{ijk} (\varphi_{k}(\partial_{13}  \phi_j \partial_2 \phi_i + \partial_{23} \phi_j \partial_1 \phi_i - 2 \partial_{12}  \phi_j \partial_3 \phi_i)) \partial_{2}\mathfrak{A}_{1,2}&(=T_{13})\\
		& + 3 \sum_{ijk}\partial_{3}(\varphi_{k}\partial_2 \phi_j \partial_3 \phi_i) \mathfrak{B}_{3} & (=T_{14})\\ 
		& + 3\sum_{ijk}\partial_{3}(\varphi_{k}\partial_1 \phi_j \partial_2 \phi_i )\mathfrak{B}_1  &(=T_{15})\\ 
		& + \sum_{ijk}\partial_{3}(\varphi_{k}(\partial_{12} \phi_j \partial_2 \phi_i - \partial_{22} \phi_j \partial_1 \phi_i)) \mathfrak{A}_{1,2} &(=T_{16})\\
		& + \sum_{ijk}(\varphi_{k}(\partial_{12} \phi_j \partial_2 \phi_i - \partial_{22} \phi_j \partial_1 \phi_i)) \partial_{3}\mathfrak{A}_{1,2} &(=T_{17})\\
		&+\sum_{ijk}\partial_{3}(\varphi_{k}(\partial_{23}  \phi_j \partial_2\phi_i - \partial_{22} \phi_j \partial_3 \phi_i)) \mathfrak{A}_{2,3}&(=T_{18}) \\
		& +\sum_{ijk}(\varphi_{k}(\partial_{23} \phi_j \partial_2\phi_i - \partial_{22} \phi_j \partial_3 \phi_i)) \partial_{3}\mathfrak{A}_{2,3} &(=T_{19})\\
		&+\sum_{ijk} \partial_{3}(\varphi_{k}(\partial_{23} \phi_j \partial_1 \phi_i + \partial_{12} \phi_j \partial_3 \phi_i - 2 \partial_{13}  \phi_j \partial_2 \phi_i)) \mathfrak{A}_{3,1} &(=T_{20})\\ 
		& + \sum_{ijk} (\varphi_{k}(\partial_{23} \phi_j \partial_1 \phi_i + \partial_{12} \phi_j \partial_3 \phi_i - 2 \partial_{13} \phi_j \partial_2 \phi_i)) \partial_{3}\mathfrak{A}_{3,1}&(=T_{21})\\
		& = \sum_{ijk}f_{ijk}^{(1)}\mathfrak{B}_{1}+f_{ijk}^{(2)}\mathfrak{B}_{2}+f_{ijk}^{(3)}\mathfrak{B}_{3}+f_{ijk}^{(1,2)}\mathfrak{A}_{1,2}+f_{ijk}^{(2,3)}\mathfrak{A}_{2,3}+f_{ijk}^{(3,1)}\mathfrak{A}_{3,1}& =: (*)
		\end{align*}
		for suitable coefficient maps $f_{ijk}^{(\cdot)}$ or $f_{ijk}^{(\cdot,\cdot)}$, respectively. To achieve this grouping we use Lemma~ \ref{lem:ABprops} \ref{item:aux1} and \ref{item:aux2} as well as the fact that $T_{11}=T_{17}=0$. In the following we will show that each of the six sums in $(*)$ vanishes individually. This is done by a very similar calculation every time. \\
		Ad~$f_{ijk}^{(1)}$. Here the coefficients are determined by terms $T_{1},T_{6},T_{13},T_{15}$ and $T_{21}$. Therefore, 
		\begin{align*}
		f_{ijk}^{(1)} = &~~6 \partial_{1}\varphi_{k}\partial_2 \phi_j \partial_3 \phi_i +6 \varphi_{k}\partial_{12} \phi_j \partial_3 \phi_i +6 \varphi_{k}\partial_2 \phi_j \partial_{13}\phi_i  + 3 \partial_{2}\varphi_{k}\partial_3 \phi_j \partial_1 \phi_i  \\
		&+ 3 \varphi_{k}\partial_{23} \phi_j \partial_1 \phi_i  + 3 \varphi_{k}\partial_3 \phi_j \partial_{12} \phi_i   +  \varphi_{k}\partial_{13}\phi_j \partial_2 \phi_i + \varphi_{k}\partial_{23} \phi_j \partial_1 \phi_i \\
		&+ (-2) \varphi_{k} \partial_{12}  \phi_j \partial_3 \phi_i  + 3\partial_{3}\varphi_{k}\partial_1 \phi_j \partial_2 \phi_i +3\varphi_{k}\partial_{13} \phi_j \partial_2 \phi_i +3\varphi_{k}\partial_1 \phi_j \partial_{23} \phi_i \\ 
		& + (-1)  \varphi_{k} \partial_{23} \phi_j \partial_1 \phi_i + (-1) \varphi_{k} \partial_{12} \phi_j \partial_3 \phi_i + 2 \varphi_{k}\partial_{13}  \phi_j \partial_2 \phi_i  =: P_{1}^{ijk}+...+P_{15}^{ijk}. 
		\end{align*}
		In the next step we group those of the $P_{l}^{ijk}$ together, that have the same structure apart from a permutation of the indices $i, j$ and $k$. For example, we have
		\begin{equation*}
		P_1^{ijk}=2P_4^{jki}=2P_{10}^{kij}.
		\end{equation*}
		We now group all the terms and then perform the corresponding index permutations:\\
		\begin{align*}
		\sum _{ijk}f_{ijk}^{(1)}\mathfrak{B}_1(i,j,k)&=\sum_{ijk} \Big[  (P_1^{ijk}+P_4^{ijk}+P_{10}^{ijk})+(P_2^{ijk}+P_6^{ijk}+P_9^{ijk}+P_{14}^{ijk})\\
		&~~+(P_3^{ijk}+P_7^{ijk}+P_{11}^{ijk}+P_{15}^{ijk})+(P_5^{ijk}+P_8^{ijk}+P_{12}^{ijk}+P_{13}^{ijk}) \Big]\mathfrak{B}_1(i,j,k) \\
		&=\sum_{ijk} P_1^{ijk}\big(\mathfrak{B}_1(i,j,k)+\tfrac{1}{2}\mathfrak{B}_1(j,k,i)+\tfrac{1}{2}\mathfrak{B}_1(k,i,j)\big)\\
		&\quad\quad+P_2^{ijk}\big(\mathfrak{B}_1(i,j,k)+\tfrac{1}{2}\mathfrak{B}_1(j,i,k)-\tfrac{1}{3}\mathfrak{B}_1(i,j,k)-\tfrac{1}{6}\mathfrak{B}_1(i,j,k) \big)\\
		&\quad\quad+P_3^{ijk}\big(\mathfrak{B}_1(i,j,k)+\tfrac{1}{6}\mathfrak{B}_1(j,i,k)+\tfrac{1}{2}\mathfrak{B}_1(j,i,k)+\tfrac{1}{3}\mathfrak{B}_1(j,i,k) \big)\\
		&\quad\quad+P_5^{ijk}\big(\mathfrak{B}_1(i,j,k)+\tfrac{1}{3}\mathfrak{B}_1(i,j,k)+\mathfrak{B}_1(j,i,k)-\tfrac{1}{3}\mathfrak{B}_1(i,j,k) \big)\\
		&=2\sum _{ijk} P_1^{ijk}\mathfrak{B}_1(i,j,k)=:(**),
		\end{align*}
		where we used Lemma \ref{lem:ABprops} \ref{item:aux4} to get the last equality. Finally, Lemma \ref{lem:summationABvanish} implies that $(**)$ vanishes identically.\\
		Ad~$f_{ijk}^{(2)}$. For the corresponding coefficients, only terms $T_{5},T_{7}$ and $T_{19}$ matter here. Therefore, 
		\begin{align*}
		f_{ijk}^{(2)} & = -2\phi_{k}\partial_{22}\phi_{j}\partial_{3}\varphi_{i}+2\varphi_{k}\partial_{23}\varphi_{j}\partial_{2}\varphi_{i}  + 3\partial_{2}\varphi_{k}\partial_2 \phi_j \partial_3 \phi_i+3\varphi_{k}\partial_{22} \phi_j \partial_3 \phi_i\\
		&\quad+3\varphi_{k}\partial_2 \phi_j \partial_{23} \phi_i + \varphi_{k}\partial_{23} \phi_j \partial_2\phi_i+ (-1) \varphi_{k}\partial_{22} \phi_j \partial_3 \phi_i \quad =:  Q_{1}^{ijk} + ... +Q_{7}^{ijk}.
		\end{align*}
		Grouping similar terms and permuting indices as above we get
		\begin{align*}
		\sum _{ijk}f_{ijk}^{(1)}\mathfrak{B}_2(i,j,k)&=\sum_{ijk} \Big[  (Q_1^{ijk}+Q_4^{ijk}+Q_{7}^{ijk})+(Q_2^{ijk}+Q_5^{ijk}+Q_6^{ijk})+Q_3^{ijk} \Big]\mathfrak{B}_2(i,j,k) \\  
		&=\sum_{ijk} Q_1^{ijk}\big(\mathfrak{B}_2(i,j,k)-\tfrac{3}{2}\mathfrak{B}_2(i,j,k)+\tfrac{1}{2}\mathfrak{B}_2(i,j,k)\big)\\
		&\quad\quad+Q_2^{ijk}\big(\mathfrak{B}_2(i,j,k)+\tfrac{3}{2}\mathfrak{B}_2(j,i,k)+\tfrac{1}{2}\mathfrak{B}_2(i,j,k) \big)+Q_3^{ijk}\mathfrak{B}_2(i,j,k)\\    
		&=\sum _{ijk} Q_3^{ijk}\mathfrak{B}_2(i,j,k)=0,
		\end{align*}
		where we again used Lemma \ref{lem:ABprops} \ref{item:aux4} and in the last step Lemma \ref{lem:summationABvanish}.\\
		Ad~$f_{ijk}^{(3)}$. Here, only terms $T_{3},T_{9},T_{14}$ contribute to the corresponding coefficients. Thus, 
		\begin{align*}
		f_{ijk}^{(3)}  =& ~~2\varphi_{k}\partial_{33}\varphi_{j}\partial_{2}\varphi_{i} + (-2)\varphi_{k}\partial_{23}\varphi_{j}\partial_{3}\varphi_{i}  + (-1) \varphi_{k}\partial_{23} \phi_j \partial_3 \phi_i+\varphi_{k} \partial_{33} \phi_j \partial_2 \phi_i\\
		&+ 3 \partial_{3}\varphi_{k}\partial_2 \phi_j \partial_3 \phi_i + 3 \varphi_{k}\partial_{23} \phi_j \partial_3 \phi_i + 3 \varphi_{k}\partial_2 \phi_j \partial_{33} \phi_i \quad  =: S_{1}^{ijk}+...+S_7^{ijk}.
		\end{align*}
		We thus get
		\begin{align*}
		\sum _{ijk}f_{ijk}^{(3)}\mathfrak{B}_3(i,j,k)&=\sum _{ijk} \Big[ (S_1^{ijk}+S_4^{ijk}+S_7^{ijk})+(S_2^{ijk}+S_3^{ijk}+S_6^{ijk})+S_5^{ijk}  \Big]\mathfrak{B}_3(i,j,k)\\
		&=\sum _{ijk} S_1^{ijk}(\mathfrak{B}_3(i,j,k)+\tfrac{1}{2}\mathfrak{B}_3(i,j,k)+\tfrac{3}{2}\mathfrak{B}_3(j,i,k))\\
		&\qquad+S_2^{ijk}(\mathfrak{B}_3(i,j,k)+\tfrac{1}{2}\mathfrak{B}_3(i,j,k)-\tfrac{3}{2}\mathfrak{B}_3(i,j,k))+S_5^{ijk}\mathfrak{B}_3(i,j,k)\\
		&=\sum _{ijk}S_5^{ijk}\mathfrak{B}_3(i,j,k)=0. 
		\end{align*}
		Ad~$f_{ijk}^{(1,2)}$. These coefficients are determined by $T_{4},T_{12}$ and $T_{16}$. In consequence, 
		\begin{align*}
		f_{ijk}^{(1,2)} =&~~ 2\partial_{1}\varphi_{k}\partial_{22}\varphi_{j}\partial_{3}\varphi_{i} +  2\varphi_{k}\partial_{122}\varphi_{j}\partial_{3}\varphi_{i} +  2\varphi_{k}\partial_{22}\varphi_{j}\partial_{13}\varphi_{i} +(-2)\partial_{1}\varphi_{k}\partial_{23}\varphi_{j}\partial_{2}\varphi_{i}\\
		&+(-2)\varphi_{k}\partial_{123}\varphi_{j}\partial_{2}\varphi_{i}+(-2)\varphi_{k}\partial_{23}\varphi_{j}\partial_{12}\varphi_{i}+  \partial_{2}\varphi_{k}\partial_{13}  \phi_j \partial_2 \phi_i+ \varphi_{k}\partial_{123} \phi_j \partial_2 \phi_i\\
		&+ \varphi_{k}\partial_{13}  \phi_j \partial_{22} \phi_i+ \partial_{2}\varphi_{k}\partial_{23}  \phi_j \partial_1 \phi_i+ \varphi_{k}\partial_{223} \phi_j \partial_1 \phi_i+ \varphi_{k}\partial_{23} \phi_j \partial_{12} \phi_i\\
		&+(-2)\partial_{2}\varphi_{k}\partial_{12}  \phi_j \partial_3 \phi_i +(-2) \varphi_{k}\partial_{122} \phi_j \partial_3 \phi_i +(-2) \varphi_{k}\partial_{12}  \phi_j \partial_{23} \phi_i  +\partial_{3}\varphi_{k}\partial_{12}  \phi_j \partial_2 \phi_i\\ &+\varphi_{k}\partial_{123} \phi_j \partial_2 \phi_i +\varphi_{k}\partial_{12} \phi_j \partial_{23}\phi_i+(-1)\partial_{3}\varphi_{k}\partial_{22} \phi_j \partial_1 \phi_i +(-1)\varphi_{k}\partial_{223} \phi_j \partial_1 \phi_i\\
		&+(-1)\varphi_{k}\partial_{22} \phi_j \partial_{13} \phi_i \quad =:  U_1^{ijk}+...+U_{21}^{ijk}.
		\end{align*}
		Here we can first note that by Lemma \ref{lem:summationABvanish} for each $l \in \lbrace 1,4,7,10,13,16,19 \rbrace$ the terms $U_{l}^{ijk}\AA_{1,2}(i,j,k)$ sum up to zero. We thus have
		\begin{align*}
		\sum _{ijk}f_{ijk}^{(1,2)}\mathfrak{A}_{1,2}(i,j,k)&=\sum _{ijk} \Big[ (U_{2}^{ijk}+U_{14}^{ijk})+(U_{3}^{ijk}+U_{9}^{ijk}+U_{21}^{ijk})+(U_{5}^{ijk}+U_{8}^{ijk}+U_{17}^{ijk})\\
		&\qquad+(U_{6}^{ijk}+U_{12}^{ijk}+U_{15}^{ijk}+U_{18}^{ijk})+(U_{11}^{ijk}+U_{20}^{ijk}) \Big]\mathfrak{A}_{1,2}(i,j,k)\\
		&=\sum _{ijk}~ U_{2}^{ijk}(\AA_{1,2}(i,j,k)-\AA_{1,2}(i,j,k) )\\ &\qquad+U_{3}^{ijk}(\AA_{1,2}(i,j,k)+\tfrac{1}{2}\AA_{1,2}(j,i,k)-\tfrac{1}{2}\AA_{1,2}(i,j,k) )\\
		&\qquad+U_{5}^{ijk}( \AA_{1,2}(i,j,k)-\tfrac{1}{2}\AA_{1,2}(i,j,k)-\tfrac{1}{2}\AA_{1,2}(i,j,k)) \\
		&\qquad+U_{6}^{ijk}(\AA_{1,2}(i,j,k)-\tfrac{1}{2}\AA_{1,2}(i,j,k)+\AA_{1,2}(j,i,k)-\tfrac{1}{2}\AA_{1,2}(j,i,k) )\\ &\qquad+U_{11}^{ijk}(\AA_{1,2}(i,j,k)-\AA_{1,2}(i,j,k) ) \quad=0.
		\end{align*}
		Ad $f_{ijk}^{(2,3)}$. Only the terms $T_8$ and $T_{18}$ matter here. In particular,
		\begin{align*}
		f_{ijk}^{(2,3)}=&~~  \partial_2 \varphi_k \partial_{23} \varphi_j \partial_3 \varphi_i +(-1) \partial_2  \varphi_k \partial_{33} \phi_j \partial_2\phi_i + \partial_3 \phi_k  \partial_{23} \phi_j \partial_2 \phi_i +(-1)\partial_3 \phi_k \partial_{22} \phi_j \partial_3 \phi_i\\
		&+  2 \phi_k \partial_{23} \phi_j \partial_{23} \phi_i +(-1)  \phi_k \partial_{33} \phi_j \partial_{22} \phi_i  +(-1)  \phi_k \partial_{22} \phi_j \partial_{33}\phi _i =:  V_{1}^{ijk} + ...+ V_{7}^{ijk}
		\end{align*}
		We first note that the terms $V_{l}^{ijk}\AA_{2,3}(i,j,k)$ for $l \in \lbrace 1,2,3,4 \rbrace$ all sum up to zero (Lemma \ref{lem:summationABvanish}). Consequently,
		\begin{align*}
		\sum _{ijk}f_{ijk}^{(2,3)}\mathfrak{A}_{2,3}(i,j,k)&=\sum _{ijk} \Big[(V_{6}^{ijk}+V_{7}^{ijk})+V_{5}^{ijk} \Big]\AA_{2,3}(i,j,k)\\
		&=\sum _{ijk} V_{6}^{ijk}(\AA_{2,3}(i,j,k)+\AA_{2,3}(j,i,k))+V_{5}^{ijk}\AA_{2,3}(i,j,k)\\
		&=\sum _{ijk} V_{5}^{ijk}\AA_{2,3}(i,j,k).
		\end{align*}
		To see that the final term vanishes, we notice $V_5^{ijk}=V_5^{jik}$ and thus
		\begin{equation*}
		\sum _{ijk} V_{5}^{ijk}\AA_{2,3}(i,j,k)=\sum _{ijk} V_{5}^{ijk}(\tfrac{1}{2}\AA_{2,3}(i,j,k)+\tfrac{1}{2}\AA_{2,3}(j,i,k))=0.
		\end{equation*}
		Ad $f_{ijk}^{(3,1)}$. Here, only the terms $T_2$, $T_{10}$ and $T_{20}$ are relevant and therefore
		\begin{align*}
		f_{ijk}^{(3,1)}=&~~ 2\partial_1 \phi_k \partial_{33} \phi_j \partial_2 \phi_1 +(-2) 2\partial_1 \phi_k \partial_{23} \phi_j \partial_3 \phi_i +  2 \phi_k \partial_{133}  \phi_j \partial_2 \phi_i +(-2)  \phi_k \partial_{123} \phi_j \partial_3 \phi_i\\
		&+  2 \phi_k \partial_{33} \phi_j \partial_{12} \phi_i +(-2) \phi_k \partial_{23} \phi_j \partial_{13} \phi_i + \partial_2 \phi_k \partial_{13} \phi_j \partial_3 \phi_i +(-1) \partial_2 \phi_k \partial_{33} \phi_j \partial_1 \phi_i \\
		&+ \phi_k \partial_{123} \phi_j \partial_3 \phi_i +(-1) \phi_k \partial_{233} \phi_j \partial_1 \phi_i + \phi_k \partial_{13} \phi_j \partial_{23} \phi_i+(-1)\phi_k \partial_{33} \phi_j \partial_{12} \phi_i \\
		&+  \partial_3 \phi_k \partial_{23}  \phi_j \partial_1 \phi_i +  \partial_3 \phi_k \partial_{12} \phi_j \partial_3 \phi_i +(-2) \partial_3 \phi_k \partial_{13} \phi_j \partial_2 \phi_i + \phi_k \partial_{233}  \phi_j \partial_1 \phi_i \\
		&+  \phi_k \partial_{123} \phi_j  \partial_3 \phi_i +(-2) \phi_k \partial_{133} \phi_j \partial_2 \phi_i +  \phi_k \partial_{23} \phi_j \partial_{13} \phi_i +  \phi_k \partial_{12} \phi_j \partial_{33} \phi_i \\
		&+(-2) \phi_k \partial_{13} \phi_j \partial_{23} \phi_i \qquad \quad=:W_1^{ijk} +...+W_{21}^{ijk}
		\end{align*}
		We first apply Lemma $\ref{lem:summationABvanish}$ to see that we can ignore the terms corresponding to $W_l^{ijk}$ for $l \in \lbrace 1,2,7,8,13,14,15 \rbrace$. For the remaining terms we calculate
		\begin{align*}
		\sum _{ijk}f_{ijk}^{(3,1)}\mathfrak{A}_{3,1}(i,j,k)&=\sum _{ijk} \Big[ (W_{3}^{ijk}+W_{18}^{ijk})+(W_{4}^{ijk}+W_{9}^{ijk}+W_{17}^{ijk})+(W_{5}^{ijk}+W_{12}^{ijk}+W_{20}^{ijk})\\
		&\qquad+(W_{6}^{ijk}+W_{11}^{ijk}+W_{19}^{ijk}+W_{21}^{ijk})+(W_{10}^{ijk}+W_{16}^{ijk}) \Big]\AA_{3,1}(i,j,k)\\
		&=\sum _{ijk}~ W_{3}^{ijk}(\AA_{3,1}(i,j,k)-\AA_{3,1}(i,j,k))\\
		&\qquad +W_{4}^{ijk}(\AA_{3,1}(i,j,k)-\tfrac{1}{2}\AA_{3,1}(i,j,k)-\tfrac{1}{2}\AA_{3,1}(i,j,k))\\
		&\qquad +W_{5}^{ijk}(\AA_{3,1}(i,j,k)-\tfrac{1}{2}\AA_{3,1}(i,j,k)+\tfrac{1}{2}\AA_{3,1}(j,i,k))\\
		&\qquad+W_{6}^{ijk}(\AA_{3,1}(i,j,k)-\tfrac{1}{2}\AA_{3,1}(j,i,k)-\tfrac{1}{2}\AA_{3,1}(i,j,k)+\AA_{3,1}(j,i,k))\\
		&\qquad+W_{10}^{ijk}(\AA_{3,1}(i,j,k)-\AA_{3,1}(i,j,k)) \quad =0.
		\end{align*}
		We thus have shown that $D=(*)=0$, yielding that the truncation is solenoidal on $\mathcal{O}_{\lambda}$.
\subsection{Proof of the identity~\eqref{eq:Bwrite}}\label{sec:App1}
Let $\psi\in\hold_{c}^{\infty}(\R^{3})$ be arbitrary. In order to obtain formula~\eqref{eq:Bwrite}, we write 
\begin{align*}
\int_{\mathcal{O}_{\lambda}}w_{1}\cdot\nabla\psi\dif x & = \int_{\mathcal{O}_{\lambda}}\mathbf{T}(\mathfrak{A}_{1,2},\nabla\psi)\dif x + \int_{\mathcal{O}_{\lambda}}\mathbf{T}(\mathfrak{A}_{2,3},\nabla\psi)\dif x  + \int_{\mathcal{O}_{\lambda}}\mathbf{T}(\mathfrak{A}_{3,1},\nabla\psi)\dif x \\ 
& +  \int_{\mathcal{O}_{\lambda}}\mathbf{T}(\mathfrak{B}_{1},\nabla\psi)\dif x + \int_{\mathcal{O}_{\lambda}}\mathbf{T}(\mathfrak{B}_{2},\nabla\psi)\dif x  + \int_{\mathcal{O}_{\lambda}}\mathbf{T}(\mathfrak{B}_{3},\nabla\psi)\dif x \\ 
& =: \sum_{\ell=1}^{6}S_{\ell}, 
\end{align*}
where we indicate e.g. by $\mathbf{T}(\mathfrak{A}_{1,2},\nabla\psi)$ that, when writing out $w_{1}\cdot\nabla\psi$ directly by means of \eqref{def:nondiagonal} and \eqref{def:diagonal}, $\mathbf{T}(\mathfrak{A}_{1,2},\nabla\psi)$ contains all appearances of $\mathfrak{A}_{1,2}(i,j,k)$ and analogously for the remaining terms. The underlying procedure of dealing with the different terms is analogous for the remaining columns $w_{2}$ and $w_{3}$, which is why we exclusively focus on $w_{1}$ but give all the details in this case.

In the following, we will frequently interchange the triple sum $\sum_{ijk}$ and the integral over $\mathcal{O}_{\lambda}$, which allows us treat the single terms via integration by parts. This interchanging of sums and integrals is allowed since every sum $\sum_{ijk}(...)$ has an integrable majorant, in turn being seen similarly to the reasoning that underlies the proof of Lemma~\ref{lem:globaldivfree}. 

We begin with $S_{1}$. This term is constituted by three parts $S_{1}^{1},S_{1}^{2},S_{1}^{3}$ given below, which stem from $w_{11}\partial_{1}\psi$, $w_{12}\partial_{2}\psi$ and $w_{13}\partial_{3}\psi$ (in this order). Here we have 
\begin{align*}
S_{1}^{1}=2&\sum_{ijk}\int_{\mathcal{O}_{\lambda}}\varphi_{k}(\partial_{22}\varphi_{j}\partial_{3}\varphi_{i}-\partial_{23}\varphi_{j}\partial_{2}\varphi_{i})\AA_{1,2}(i,j,k)\partial_{1}\psi \dif x &\\
& = -2\sum_{ijk}\int_{\mathcal{O}_{\lambda}}(\partial_{2}\varphi_{j})(\partial_{2}\varphi_{k}\partial_{3}\varphi_{i}\AA_{1,2}(i,j,k)\partial_{1}\psi)\dif x & (=T_{1}^{1})\\ 
& -2\sum_{ijk}\int_{\mathcal{O}_{\lambda}}(\partial_{2}\varphi_{j})(\varphi_{k}\partial_{23}\varphi_{i}\AA_{1,2}(i,j,k)\partial_{1}\psi)\dif x& (=T_{2}^{1})\\ 
& -2\sum_{ijk}\int_{\mathcal{O}_{\lambda}}(\partial_{2}\varphi_{j})(\varphi_{k}\partial_{3}\varphi_{i}\partial_{2}\AA_{1,2}(i,j,k)\partial_{1}\psi)\dif x& (=T_{3}^{1})\\ 
& -2\sum_{ijk}\int_{\mathcal{O}_{\lambda}}(\partial_{2}\varphi_{j})(\varphi_{k}\partial_{3}\varphi_{i}\AA_{1,2}(i,j,k)\partial_{12}\psi)\dif x& (=T_{4}^{1})\\ 
& - 2\sum_{i,j,k}\int_{\mathcal{O}_{\lambda}}\varphi_{k}\partial_{23}\varphi_{j}\partial_{2}\varphi_{i}\AA_{1,2}(i,j,k)\partial_{1}\psi \dif x &(=T_{5}^{1}).
\end{align*}
Permuting indices $j\leftrightarrow k$ and using the antisymmetry from Lemma~\ref{lem:ABprops}~\ref{item:aux3}, we obtain 
\begin{align}\label{eq:permuvanish1}
\begin{split}
T_{1}^{1} & = -2\sum_{ijk}\int_{\mathcal{O}_{\lambda}}(\partial_{2}\varphi_{j})(\partial_{2}\varphi_{k}\partial_{3}\varphi_{i}\AA_{1,2}(i,j,k)\partial_{1}\psi)\dif x \\ 
& = 2\sum_{ijk}\int_{\mathcal{O}_{\lambda}}(\partial_{2}\varphi_{j})(\partial_{2}\varphi_{k}\partial_{3}\varphi_{i}\AA_{1,2}(i,k,j)\partial_{1}\psi)\dif x  \\ 
& = 2\sum_{ikj}\int_{\mathcal{O}_{\lambda}}(\partial_{2}\varphi_{j})(\partial_{2}\varphi_{k}\partial_{3}\varphi_{i}\AA_{1,2}(i,k,j)\partial_{1}\psi)\dif x = - T_{1}^{1}, 
\end{split}
\end{align}
and hence $T_{1}^{1}=0$. Equally, permuting $i\leftrightarrow j$, we find that $T_{2}^{1}+T_{5}^{1}=0$. Therefore, using Lemma~\ref{lem:ABprops}~\ref{item:aux2} for $T_{3}^{1}$ and integrating by parts in term $T_{4}^{1}$ with respect to $\partial_{1}$, 
\begin{align*}
S_{1}^{1} = T_{3}^{1}+T_{4}^{1} & = -2\sum_{ijk}\int_{\mathcal{O}_{\lambda}}(\partial_{2}\varphi_{j})(\varphi_{k}\partial_{3}\varphi_{i}\mathfrak{B}_{1}(i,j,k)\partial_{1}\psi)\dif x & (=T_{6}^{1})\\ 
& + 2 \sum_{ijk}\int_{\mathcal{O}_{\lambda}}(\partial_{12}\varphi_{j})\varphi_{k}\partial_{3}\varphi_{i}\AA_{1,2}(i,j,k)\partial_{2}\psi\dif x & (=T_{7}^{1})\\
& + 2 \sum_{ijk}\int_{\mathcal{O}_{\lambda}}(\partial_{2}\varphi_{j})\partial_{1}\varphi_{k}\partial_{3}\varphi_{i}\AA_{1,2}(i,j,k)\partial_{2}\psi\dif x & (=T_{8}^{1})\\ 
& + 2 \sum_{ijk}\int_{\mathcal{O}_{\lambda}}(\partial_{2}\varphi_{j})\varphi_{k}\partial_{13}\varphi_{i}\AA_{1,2}(i,j,k)\partial_{2}\psi\dif x & (=T_{9}^{1})\\
& \!\!\!\!\!\!\!\!\!\!\!\stackrel{\text{Lem.~\ref{lem:ABprops}~\ref{item:aux1}}}{-} 2 \sum_{ijk}\int_{\mathcal{O}_{\lambda}}(\partial_{2}\varphi_{j})\varphi_{k}\partial_{3}\varphi_{i}\mathfrak{B}_{2}(i,j,k)\partial_{2}\psi\dif x & (=T_{10}^{1}).
\end{align*}
On the other hand, 
\begin{align*}
S_{1}^{2} = & \sum_{ijk}\int_{\mathcal{O}_{\lambda}}\varphi_{k}(\partial_{13}\varphi_{j}\partial_{2}\varphi_{i})\AA_{1,2}(i,j,k)\partial_{2}\psi\dif x & (=T_{1}^{2})\\ 
&+\sum_{ijk}\int_{\mathcal{O}_{\lambda}}\varphi_{k}(\partial_{23}\varphi_{j}\partial_{1}\varphi_{i})\AA_{1,2}(i,j,k)\partial_{2}\psi\dif x & (=T_{2}^{2})\\
&-\sum_{ijk}\int_{\mathcal{O}_{\lambda}}\varphi_{k}(2\partial_{12}\varphi_{j}\partial_{3}\varphi_{i})\AA_{1,2}(i,j,k)\partial_{2}\psi\dif x & (=T_{3}^{2})
\end{align*}
We finally turn to $S_{1}^{3}$. Here we have 
\begin{align*}
S_{1}^{3} & = \sum_{ijk}\int_{\mathcal{O}_{\lambda}}\varphi_{k}(\partial_{12}\varphi_{j}\partial_{2}\varphi_{i}-\partial_{22}\varphi_{j}\partial_{1}\varphi_{i})\AA_{1,2}(i,j,k)\partial_{3}\psi\dif x & \\ 
& = -\sum_{ijk}\int_{\mathcal{O}_{\lambda}}\partial_{1}\varphi_{j}\partial_{2}\varphi_{k}\partial_{2}\varphi_{i}\AA_{1,2}(i,j,k)\partial_{3}\psi\dif x &(=T_{1}^{3})\\ 
&  -\sum_{ijk}\int_{\mathcal{O}_{\lambda}}\partial_{1}\varphi_{j}\varphi_{k}\partial_{22}\varphi_{i}\AA_{1,2}(i,j,k)\partial_{3}\psi\dif x &(=T_{2}^{3})\\
& -\sum_{ijk}\int_{\mathcal{O}_{\lambda}}\partial_{1}\varphi_{j}\varphi_{k}\partial_{2}\varphi_{i}\partial_{2}\AA_{1,2}(i,j,k)\partial_{3}\psi\dif x &(=T_{3}^{3})\\ 
& - \sum_{ijk}\int_{\mathcal{O}_{\lambda}}\partial_{1}\varphi_{j}\varphi_{k}\partial_{2}\varphi_{i}\AA_{1,2}(i,j,k)\partial_{23}\psi\dif x &(=T_{4}^{3})\\ 
& - \sum_{ijk}\int_{\mathcal{O}_{\lambda}}\varphi_{k}(\partial_{22}\varphi_{j}\partial_{1}\varphi_{i})\AA_{1,2}(i,j,k)\partial_{3}\psi\dif x &(=T_{5}^{3})
\end{align*}
Again, $T_{1}^{3}$ vanishes by the same argument as for \eqref{eq:permuvanish1}, $T_{2}^{3}+T_{5}^{3}=0$ by permuting indices $i\leftrightarrow j$, and so we obtain analogously to above 
\begin{align*}
S_{1}^{3} & = -\sum_{ijk}\int_{\mathcal{O}_{\lambda}}\partial_{1}\varphi_{j}\varphi_{k}\partial_{2}\varphi_{i}\mathfrak{B}_{1}(i,j,k)\partial_{3}\psi\dif x &(=T_{6}^{3})\\ 
& + \sum_{ijk}\int_{\mathcal{O}_{\lambda}}\partial_{13}\varphi_{j}\varphi_{k}\partial_{2}\varphi_{i}\AA_{1,2}(i,j,k)\partial_{2}\psi\dif x&(=T_{7}^{3})\\ 
& + \sum_{ijk}\int_{\mathcal{O}_{\lambda}}\partial_{1}\varphi_{j}\partial_{3}\varphi_{k}\partial_{2}\varphi_{i}\AA_{1,2}(i,j,k)\partial_{2}\psi\dif x&(=T_{8}^{3})\\
& + \sum_{ijk}\int_{\mathcal{O}_{\lambda}}\partial_{1}\varphi_{j}\varphi_{k}\partial_{23}\varphi_{i}\AA_{1,2}(i,j,k)\partial_{2}\psi\dif x&(=T_{9}^{3})\\
& + \sum_{ijk}\int_{\mathcal{O}_{\lambda}}\partial_{1}\varphi_{j}\varphi_{k}\partial_{2}\varphi_{i}\underbrace{\partial_{3}\AA_{1,2}(i,j,k)}_{=0}\partial_{2}\psi\dif x.&
\end{align*}
Permuting indices $i\leftrightarrow j$ in $T_{1}^{2}$ and $T_{7}^{3}$ yields by virtue of the antisymmetry property of $\mathfrak{A}_{1,2}$ that $
T_{9}^{1}+T_{1}^{2}+T_{7}^{3} = 0$, and we directly find that $T_{7}^{1}+T_{3}^{2} = 0$. For terms $T_{8}^{1}$ and $T_{8}^{3}$, we permute indices $i\leftrightarrow j$ and $j\leftrightarrow k$ in term $T_{8}^{3}$ to obtain 
\begin{align}\label{eq:permu3}
T_{8}^{1}+T_{8}^{3} = 3\sum_{ijk}\int_{\mathcal{O}_{\lambda}}(\partial_{1}\varphi_{k})(\partial_{2}\varphi_{j})(\partial_{3}\varphi_{i})\AA_{1,2}(i,j,k)\partial_{2}\psi\dif x
\end{align}
For terms $T_{2}^{2}$ and $T_{9}^{3}$, we permute indices $i\leftrightarrow j$ in $T_{9}^{3}$ to obtain 
$T_{2}^{2}+T_{9}^{3} = 0$. Having left $T_{6}^{1}$ and $T_{6}^{3}$ untouched, we thus obtain 
\begin{align}\label{eq:S4vanish}
\begin{split}
S_{1} & = -2\sum_{ijk}\int_{\mathcal{O}_{\lambda}}(\partial_{2}\varphi_{j})(\varphi_{k}\partial_{3}\varphi_{i}\mathfrak{B}_{1}(i,j,k)\partial_{1}\psi)\dif x \;\;\;\;\;\;\;\;\;\,\,(=T_{6}^{1})\\
& -\sum_{ijk}\int_{\mathcal{O}_{\lambda}}\partial_{1}\varphi_{j}\varphi_{k}\partial_{2}\varphi_{i}\mathfrak{B}_{1}(i,j,k)\partial_{3}\psi\dif x \;\;\;\;\;\;\;\;\;\;\;\;\;\;\;\;\;\;\;\;\;(=T_{6}^{3}) \\ 
& - 2 \sum_{ijk}\int_{\mathcal{O}_{\lambda}}(\partial_{2}\varphi_{j})\varphi_{k}\partial_{3}\varphi_{i}\mathfrak{B}_{2}(i,j,k)\partial_{2}\psi\dif x  \;\;\;\;\;\;\;\;\;\;\;\;\;\;\;\;\,(=T_{10}^{1})\\
& +3\sum_{ijk}\int_{\mathcal{O}_{\lambda}}(\partial_{1}\varphi_{k})(\partial_{2}\varphi_{j})(\partial_{3}\varphi_{i})\AA_{1,2}(i,j,k)\partial_{2}\psi\dif x \;\;\;\;(=T_{8}^{1}+T_{8}^{3}) \\ 
& =: \mathbf{S}_{1}+\mathbf{S}_{2}+\mathbf{S}_{3} + \mathbf{S}'_{4}. 
\end{split}
\end{align}
We now claim that $\mathbf{S}'_{4}=0$. Let us first note that the overall sum in the definition of $\mathbf{S}'_{4}$ converges absolutely in $\lebe^{1}(\mathcal{O}_{\lambda})$. This can be seen similarly to the proof of Lemma~\ref{lem:linftybound}, and is a consequence of \ref{item:P3}, Lemma~\ref{lem:JogiLoew}~\ref{item:aux8} and $\mathscr{L}^{3}(\mathcal{O}_{\lambda})<\infty$, together with the bound 
\begin{align*}
\sum_{ijk}\int_{\mathcal{O}_{\lambda}}|(\partial_{1}\varphi_{k})(\partial_{2}\varphi_{j})(\partial_{3}\varphi_{i})\AA_{1,2}(i,j,k)\partial_{2}\psi|\dif x \leq c\lambda\|\nabla w_{1}\|_{\lebe^{1}(\R^{3})}\mathscr{L}^{3}(\mathcal{O}_{\lambda}), 
\end{align*}
where $c=c(3)>0$ is a constant only depending on the underlying space dimension $n=3$. By Lemma~\ref{lem:summationABvanish}, we have
\begin{align}\label{eq:permu6}
\sum_{ijk}(\partial_{1}\varphi_{k})(\partial_{2}\varphi_{j})(\partial_{3}\varphi_{i})\AA_{1,2}(i,j,k)\partial_{2}\psi \equiv 0\qquad\text{pointwisely in}\;\mathcal{O}_{\lambda}, 
\end{align}
to be understood as the limit of the corresponding partial sums. Therefore, 
\begin{align}\label{eq:permuS1}
\begin{split}
S_{1} & = -2\sum_{ijk}\int_{\mathcal{O}_{\lambda}}(\partial_{2}\varphi_{j})(\varphi_{k}\partial_{3}\varphi_{i}\mathfrak{B}_{1}(i,j,k)\partial_{1}\psi)\dif x \;\;\;\;\;\;\;\;\;\,\,(=T_{6}^{1})\\
& -\sum_{ijk}\int_{\mathcal{O}_{\lambda}}\partial_{1}\varphi_{j}\varphi_{k}\partial_{2}\varphi_{i}\mathfrak{B}_{1}(i,j,k)\partial_{3}\psi\dif x \;\;\;\;\;\;\;\;\;\;\;\;\;\;\;\;\;\;\;\;\;(=T_{6}^{3}) \\ 
& - 2 \sum_{ijk}\int_{\mathcal{O}_{\lambda}}(\partial_{2}\varphi_{j})\varphi_{k}\partial_{3}\varphi_{i}\mathfrak{B}_{2}(i,j,k)\partial_{2}\psi\dif x  \;\;\;\;\;\;\;\;\;\;\;\;\;\;\;\;\,(=T_{10}^{1}) \\ 
& =: \mathbf{S}_{1}+\mathbf{S}_{2}+\mathbf{S}_{3}. 
\end{split}
\end{align}
We now turn to $S_{2}$. Our line of action is similar to that for dealing with $S_{1}$ and so, integrating by parts twice, we successively obtain
\begin{align*}
S_{2} & = \sum_{ijk} \int_{\mathcal{O}_{\lambda}}\varphi_{k}(\partial_{23}\varphi_{j}\partial_{3}\varphi_{i}-\partial_{33}\varphi_{j}\partial_{2}\varphi_{i})\AA_{2,3}(i,j,k)\partial_{2}\psi\dif x &\\ 
& + \sum_{ijk}\int_{\mathcal{O}_{\lambda}}\varphi_{k}(\partial_{23}\varphi_{j}\partial_{2}\varphi_{i}-\partial_{22}\varphi_{j}\partial_{3}\varphi_{i})\AA_{2,3}(i,j,k)\partial_{3}\psi \dif x &\\ 
& = \sum_{ijk} (-1)\int_{\mathcal{O}_{\lambda}} (\partial_{2}\varphi_{j})(\partial_{3}\varphi_{k}\partial_{3}\varphi_{i}\AA_{2,3}(i,j,k)\partial_{2}\psi)\dif x & (=T_{1})\\ 
& -\sum_{ijk} \int_{\mathcal{O}_{\lambda}}(\partial_{2}\varphi_{j})(\varphi_{k}\partial_{33}\varphi_{i}\AA_{2,3}(i,j,k)\partial_{2}\psi) \dif x & (=T_{2})\\ 
& -\sum_{ijk} \int_{\mathcal{O}_{\lambda}}(\partial_{2}\varphi_{j})(\varphi_{k}\partial_{3}\varphi_{i}\partial_{3}\AA_{2,3}(i,j,k)\partial_{2}\psi) \dif x&(=T_{3}) \\ 
& -\sum_{ijk}\int_{\mathcal{O}_{\lambda}}(\partial_{2}\varphi_{j})(\varphi_{k}\partial_{3}\varphi_{i}\AA_{2,3}(i,j,k)\partial_{23}\psi) \dif x &(=T_{4})\\ 
& -\sum_{ijk} \int_{\mathcal{O}_{\lambda}} (\varphi_{k}\partial_{33}\varphi_{j}\partial_{2}\varphi_{i})\AA_{2,3}(i,j,k)\partial_{2}\psi\dif x & (=T_{5})\\  
& -\sum_{ijk} \int_{\mathcal{O}_{\lambda}}(\partial_{3}\varphi_{j}) (\partial_{2}\varphi_{k}\partial_{2}\varphi_{i}\AA_{2,3}(i,j,k)\partial_{3}\psi) \dif x  & (=T_{6})\\
&-\sum_{ijk} \int_{\mathcal{O}_{\lambda}}(\partial_{3}\varphi_{j}) (\varphi_{k}\partial_{22}\varphi_{i}\AA_{2,3}(i,j,k)\partial_{3}\psi) \dif x  &(=T_{7})\\
&-\sum_{ijk} \int_{\mathcal{O}_{\lambda}}(\partial_{3}\varphi_{j}) (\varphi_{k}\partial_{2}\varphi_{i}\partial_{2}\AA_{2,3}(i,j,k)\partial_{3}\psi) \dif x & (=T_{8})\\
&-\sum_{ijk} \int_{\mathcal{O}_{\lambda}}(\partial_{3}\varphi_{j}) (\varphi_{k}\partial_{2}\varphi_{i}\AA_{2,3}(i,j,k)\partial_{23}\psi) \dif x &(=T_{9})\\
& -\sum_{ijk} \int_{\mathcal{O}_{\lambda}}(\varphi_{k}\partial_{22}\varphi_{j}\partial_{3}\varphi_{i})\AA_{2,3}(i,j,k)\partial_{3}\psi \dif x & (=T_{10}).
\end{align*} 
Terms $T_{1}$ and $T_{6}$ vanish by the same argument as in \eqref{eq:permuvanish1}. Permuting indices $i\leftrightarrow j$, we then obtain $T_{2}+T_{5}=0$, and in a similar manner we see that $T_{7}+T_{10}=0$ and $T_{4}+T_{9}=0$. To conclude, we use Lemma~\ref{lem:ABprops} to obtain 
\begin{align}\label{eq:permuS2}
\begin{split}
S_{2}=T_{3} +T_{8}  = &-\sum_{ijk}\int_{\mathcal{O}_{\lambda}}(\partial_{2}\varphi_{j})(\varphi_{k}\partial_{3}\varphi_{i}\mathfrak{B}_{2}(i,j,k)\partial_{2}\psi) \dif x \\ 
& + \sum_{ijk} \int_{\mathcal{O}_{\lambda}}(\partial_{3}\varphi_{j}) (\varphi_{k}\partial_{2}\varphi_{i}\mathfrak{B}_{3}(i,j,k)\partial_{3}\psi) \dif x =: \mathbf{S}_{4}+\mathbf{S}_{5}. 
\end{split}
\end{align}
Term $S_{3}$ is given by 
\begin{align*}
S_{3} & :=2\sum_{ijk}\int_{\mathcal{O}_{\lambda}}(\partial_{33}\varphi_{j}\partial_{2}\varphi_{i}-\partial_{23}\varphi_{i}\partial_{3}\varphi_{i})\AA_{3,1}(i,j,k)\partial_{1}\psi \\ 
& + \sum_{ijk}\int_{\mathcal{O}_{\lambda}}(\partial_{13}\varphi_{j}\partial_{3}\varphi_{i}-\partial_{33}\varphi_{j}\partial_{1}\varphi_{i})\AA_{3,1}(i,j,k)\partial_{2}\psi\\ 
& + \sum_{ijk}\int_{\mathcal{O}_{\lambda}}(\partial_{23}\varphi_{j}\partial_{1}\varphi_{i} + \partial_{12}\varphi_{j}\partial_{3}\varphi_{i} - 2\partial_{13}\varphi_{j}\partial_{2}\varphi_{i})\AA_{3,1}(i,j,k)\partial_{3}\psi \\ & =: S_{3}^{1}+S_{3}^{2}+S_{3}^{3}
\end{align*}
Terms $S_{3}^{1}$ and $S_{3}^{2}$ are treated as as term $S_{1}^{1}$, where we now integrate by parts with respect to $\partial_{3}$ in $S_{3}^{1}$ or with respect to $\partial_{1}$ in $S_{3}^{2}$, respectively. Similary to the computation underlying $S_{1}$, this gives us 
\begin{align*}
S_{3} & = 2 \sum_{ijk}\int_{\mathcal{O}_{\lambda}}(\partial_{3}\varphi_{j})\varphi_{k}(\partial_{2}\varphi_{i})\mathfrak{B}_{1}(i,j,k)\partial_{1}\psi & (=T'_{1}) \\ 
& + 2\sum_{ijk}\int_{\mathcal{O}_{\lambda}}(\partial_{13}\varphi_{j})\varphi_{k}\partial_{2}\varphi_{i}\AA_{3,1}(i,j,k)\partial_{3}\psi & (=T'_{2})\\
& + 2\sum_{ijk}\int_{\mathcal{O}_{\lambda}}(\partial_{3}\varphi_{j})\partial_{1}\varphi_{k}\partial_{2}\varphi_{i}\AA_{3,1}(i,j,k)\partial_{3}\psi& (=T'_{3}) \\ 
& + 2\sum_{ijk}\int_{\mathcal{O}_{\lambda}}(\partial_{3}\varphi_{j})\varphi_{k}\partial_{12}\varphi_{i}\AA_{3,1}(i,j,k)\partial_{3}\psi & (=T'_{4})\\ 
& + 2\sum_{ijk}\int_{\mathcal{O}_{\lambda}}(\partial_{3}\varphi_{j})\varphi_{k}\partial_{2}\varphi_{i}\mathfrak{B}_{3}(i,j,k)\partial_{3}\psi & (=T'_{5})\\ 
& + \sum_{ijk}\int_{\mathcal{O}_{\lambda}}\partial_{1}\varphi_{j}\varphi_{k}\partial_{3}\varphi_{i}\mathfrak{B}_{1}(i,j,k)\partial_{2}\psi\dif x & (=T'_{6})\\ 
& + \sum_{ijk}\int_{\mathcal{O}_{\lambda}}(\partial_{2}\partial_{1}\varphi_{j})\varphi_{k}\partial_{3}\varphi_{i}\AA_{3,1}(i,j,k)\partial_{3}\psi & (=T'_{7})\\
& + \sum_{ijk}\int_{\mathcal{O}_{\lambda}}(\partial_{1}\varphi_{j})\partial_{2}\varphi_{k}\partial_{3}\varphi_{i}\AA_{3,1}(i,j,k)\partial_{3}\psi & (=T'_{8})\\ 
& + \sum_{ijk}\int_{\mathcal{O}_{\lambda}}(\partial_{1}\varphi_{j})\varphi_{k}\partial_{23}\varphi_{i}\AA_{3,1}(i,j,k)\partial_{3}\psi & (=T'_{9})\\ 
& + \sum_{ijk}\int_{\mathcal{O}_{\lambda}}\varphi_{k}(\partial_{23}\varphi_{j}\partial_{1}\varphi_{i})\AA_{3,1}(i,j,k)\partial_{3}\psi & (=T'_{10})\\ 
& + \sum_{ijk}\int_{\mathcal{O}_{\lambda}}\varphi_{k}(\partial_{12}\varphi_{j}\partial_{3}\varphi_{i})\AA_{3,1}(i,j,k)\partial_{3}\psi & (=T'_{11})\\ 
& -2\sum_{ijk}\int_{\mathcal{O}_{\lambda}}\varphi_{k}\partial_{13}\varphi_{j}\partial_{2}\varphi_{i}\AA_{3,1}(i,j,k)\partial_{3}\psi & (=T'_{12}).
\end{align*}
By an argument analogous to~\eqref{eq:S4vanish}ff., $T'_{3}=T'_{8}=0$. Moreover, permuting indices yields as above $T'_{4}+T'_{7}+T'_{11}=0$ and $T'_{9}+T'_{10}=0$, whereas $T'_{2}+T'_{12}=0$ follows directly. Therefore, 
\begin{align}\label{eq:permuS3}
\begin{split}
S_{3} & = 2 \sum_{ijk}\int_{\mathcal{O}_{\lambda}}(\partial_{3}\varphi_{j})\varphi_{k}(\partial_{2}\varphi_{i})\mathfrak{B}_{1}(i,j,k)\partial_{1}\psi \\ 
& + 2\sum_{ijk}\int_{\mathcal{O}_{\lambda}}(\partial_{3}\varphi_{j})\varphi_{k}\partial_{2}\varphi_{i}\mathfrak{B}_{3}(i,j,k)\partial_{3}\psi\\ 
& + \sum_{ijk}\int_{\mathcal{O}_{\lambda}}\partial_{1}\varphi_{j}\varphi_{k}\partial_{3}\varphi_{i}\mathfrak{B}_{1}(i,j,k)\partial_{2}\psi\dif x =: \mathbf{S}_{6}+\mathbf{S}_{7}+\mathbf{S}_{8}
\end{split}
\end{align} 
Until now, we have only considered the contributions from $\AA_{1,2}$, $\AA_{3,1}$ and $\AA_{2,3}$. The contributions containing $\mathfrak{B}_{1},\mathfrak{B}_{2},\mathfrak{B}_{3}$ then read as 
\begin{align*}
S_{4}+S_{5}+S_{6} & = 6\sum_{ijk}\int_{\mathcal{O}_{\lambda}}\varphi_{k}\partial_{2}\varphi_{j}\partial_{3}\varphi_{i}\mathfrak{B}_{1}(i,j,k)\partial_{1}\psi \\ 
& + 3 \sum_{ijk}\int_{\mathcal{O}_{\lambda}}\varphi_{k}\partial_{3}\varphi_{j}\partial_{1}\varphi_{i}\mathfrak{B}_{1}(i,j,k)\partial_{2}\psi\\
& + 3\sum_{ijk}\int_{\mathcal{O}_{\lambda}}\varphi_{k}\partial_{1}\varphi_{j}\partial_{2}\varphi_{i}\mathfrak{B}_{1}(i,j,k)\partial_{3}\psi \\ 
& + 3\sum_{ijk}\int_{\mathcal{O}_{\lambda}}\varphi_{k}\partial_{2}\varphi_{j}\partial_{3}\varphi_{i}\mathfrak{B}_{2}(i,j,k)\partial_{2}\psi\\
& + 3\sum_{ijk}\int_{\mathcal{O}_{\lambda}}\varphi_{k}\partial_{2}\varphi_{j}\partial_{3}\varphi_{i}\mathfrak{B}_{3}(i,j,k)\partial_{3}\psi\\
& = \mathbf{S}_{9}+\mathbf{S}_{10}+\mathbf{S}_{11}+\mathbf{S}_{12}+\mathbf{S}_{13}.
\end{align*}
Combining this with~\eqref{eq:permuS1},~ \eqref{eq:permuS2} and~\eqref{eq:permuS3}, we may then build the overall sum $S_{1}+...+S_{6}=\mathbf{S}_{1}+...+\mathbf{S}_{13}$. Summing up all terms, we note by an analogous permutation argument that $\mathbf{S}_{3}+\mathbf{S}_{4}+\mathbf{S}_{12}=0$, $\mathbf{S}_{5}+\mathbf{S}_{7}+\mathbf{S}_{13}=0$, and so 
\begin{align*}
\int_{\mathcal{O}_{\lambda}}w_{1}\cdot\nabla\psi\dif x & = 2\sum_{ijk}\int_{\mathcal{O}_{\lambda}}\varphi_{k}\partial_{2}\varphi_{j}\partial_{3}\varphi_{i}\mathfrak{B}_{1}(i,j,k)\partial_{1}\psi\dif x & (\sim \mathbf{S}_{1}+\mathbf{S}_{6}+\mathbf{S}_{9})\\ 
& + 2\sum_{ijk}\int_{\mathcal{O}_{\lambda}}\varphi_{k}\partial_{1}\varphi_{i}\partial_{3}\varphi_{j}\mathfrak{B}_{1}(i,j,k)\partial_{2}\psi\dif x &(\sim \mathbf{S}_{8}+\mathbf{S}_{10})\\ 
& + 2\sum_{ijk}\int_{\mathcal{O}_{\lambda}}\varphi_{k}\partial_{1}\varphi_{j}\partial_{2}\varphi_{i}\mathfrak{B}_{1}(i,j,k)\partial_{3}\psi\dif x &(\sim \mathbf{S}_{2}+\mathbf{S}_{11}), 
\end{align*}
where we use the symbol '$\sim$' to indicate where the single terms stem from. This is precisely~\eqref{eq:Bwrite}, and so the proof is complete.

\end{document}